\documentclass[reqno, 11pt, centertags,draft]{amsart}
\usepackage{amsmath,amsthm,amscd,amssymb,latexsym,upref,stmaryrd}
\usepackage{color,soul}

\numberwithin{equation}{section}

\newtheorem{theorem}{Theorem}[section]
\newtheorem{proposition}[theorem]{Proposition}
\newtheorem{corollary}[theorem]{Corollary}
\newtheorem{lemma}[theorem]{Lemma}
\newtheorem{remark}[theorem]{Remark}
\newtheorem{definition}[theorem]{Definition}

\DeclareMathOperator{\col}{col}
\DeclareMathOperator{\diag}{diag}

\DeclareMathOperator{\dom}{dom}
\DeclareMathOperator{\range}{ran}
\DeclareMathOperator{\rank}{rank}
\DeclareMathOperator{\Span}{span}
\DeclareMathOperator{\Lip}{Lip}

\DeclareMathOperator{\tr}{tr}
\DeclareMathOperator{\GL}{GL}
\DeclareMathOperator{\supp}{supp}
\DeclareMathOperator{\const}{const}
\DeclareMathOperator{\ctg}{ctg}

\newcommand{\eps}{\varepsilon}
\renewcommand{\k}{\kappa}
\renewcommand{\l}{\lambda}

\renewcommand{\Im}{{\rm Im}}
\renewcommand{\Re}{{\rm Re}}
\newcommand{\wt}{\widetilde}
\newcommand{\ol}{\overline}

\def\cB{\mathcal{B}}

\def\cH{\mathcal{H}}
\def\cI{\mathcal{I}}
\def\cK{\mathcal{K}}
\def\cL{\mathcal{L}}
\def\cR{\mathcal{R}}
\def\cS{\mathcal{S}}

\def\fH{\mathfrak{H}}

\def\fS{\mathfrak{S}}

\def\bC{\mathbb{C}}
\def\bD{\mathbb{D}}
\def\bN{\mathbb{N}}
\def\bR{\mathbb{R}}
\def\bZ{\mathbb{Z}}

\begin{document}

\sloppy

\title[On completeness and Riesz basis property]
{On the completeness and Riesz basis property \\
of root subspaces of boundary value problems \\
for first order systems and applications}

\author{Anton A. Lunyov}
\address{Institute of Applied Mathematics and Mechanics, NAS of Ukraine,
R. Luxemburg str. 74, \, 83114 Donetsk, Ukraine} \curraddr{}
\email{A.A.Lunyov@gmail.com}

\author{Mark~M.~Malamud}
\address{Institute of Applied Mathematics and Mechanics, NAS of Ukraine,
R. Luxemburg str. 74, \, 83114 Donetsk, Ukraine} \curraddr{}
\email{mmm@telenet.dn.ua}

\date{}

\begin{abstract}
The paper is concerned with the completeness property of root
functions of general boundary value problems for $n \times n$
first order systems of ordinary differential equations on a
finite interval. In comparison with the recent
paper~\cite{MalOri12} we substantially relax the assumptions on
boundary conditions guarantying the completeness of root
vectors, allowing them to be non-weakly regular and even
degenerate. Emphasize that in this case the completeness
property substantially depends on the values of a potential
matrix at the endpoints of the interval.

It is also shown that the system of root vectors of the general
$n \times n$ Dirac type system subject to certain boundary
conditions forms a Riesz basis with parentheses. We also show
that arbitrary complete dissipative boundary value problem for
Dirac type operator with a summable potential matrix admits the
spectral synthesis in $L^2([0,1]; \mathbb{C}^n)$. Finally, we
apply our results to investigate completeness and the Riesz
basis property of the dynamic generator of spatially
non-homogenous damped Timoshenko beam model.
\end{abstract}

\maketitle

\renewcommand{\contentsname}{Contents}
\tableofcontents

\section{Introduction}
%
%
Spectral theory of non-selfadjoint boundary value problems (BVP)
on a finite interval $\cI=(a,b)$ for $n$th order ordinary
differential equations (ODE)
\begin{equation}\label{eq:ODE}
    y^{(n)} + q_1y^{(n-2)} + ... + q_{n-1}y = \l ^n y, \qquad x\in (a,b),
\end{equation}
with coefficients $q_j\in L^1(a,b)$ takes its origin in the
classical papers by Birkhoff~\cite{Bir08, Bir08exp} and
Tamarkin~\cite{Tam12, Tam17, Tam28}. They introduced the concept
of \emph{regular boundary conditions} for ODE and investigated
the asymptotic behavior of eigenvalues and eigenfunctions of
related BVP. Moreover, they proved that the system of root
functions, i.e. eigenfunctions and associated functions, of the
regular BVP is complete. Their results are also treated in the
classical monographs (see~\cite[Section 2]{Nai69}
and~\cite[Chapter 19]{DunSch71}).

The completeness property of non-regular BVP for $n$th order
ODE~\eqref{eq:ODE} has been studied by
M.V.~Keldysh~\cite{Kel51}, A.A.~Shkalikov~\cite{Shk76},
A.G.~Kostyuchenko and A.A.~Shkalikov~\cite{KosShk78},
G.M.~Gubreev~\cite{Gub03}, A.P.~Khromov~\cite{Khr77,Khr03},
V.S.~Rykhlov~\cite{Ryh09} and many others (see references
in~\cite{Khr03}). On the other hand, the Riesz basis property
for regular BVP were investigated by N.~Dunford~\cite{Dun58},
V.P.~Mikhailov~\cite{Mikh62}, G.M.~Kesel'man~\cite{Kes64},
N.~Dunford and J.~Schwartz~\cite[Chapter 19.4]{DunSch71},
A.A.~Shkalikov~\cite{Shk79,Shk82,Shk83}. Numerous papers are
devoted to the completeness and Riesz basis property for the
Sturm-Liouville operator (see the recent paper  \cite{ShkVel09}
by A. Shkalikov and O. Veliev and the review~\cite{Mak12} by
A.S.~Makin and the references therein). We especially mention
the recent achievements for periodic (anti-periodic)
Sturm-Liouville operator $-\frac{d^2}{dx^2} + q(x)$ on
$[0,\pi]$. Namely, F.~Gesztesy and
V.A.~Tkachenko~\cite{GesTka09,GesTka11} for $q \in L^2[0,\pi]$
and later on P.~Djakov and B.S.~Mityagin~\cite{DjaMit12Crit} for
$q \in W^{-1,2}[0,\pi]$ established by different methods a
\emph{criterion} for the system of root functions to contain a
Riesz basis (see Remark~\ref{rem:Riesz.basis} for detailed
discussion).

In this paper we consider first order system of ODE of the form
\begin{equation} \label{eq:system}
    Ly := L(Q)y := -i B^{-1} y' + Q(x)y = \l y,
    \quad y = \col(y_1,...,y_n),
\end{equation}
where $B$ is a nonsingular diagonal $n\times n$ matrix with
complex entries,
\begin{equation} \label{eq:B.def}
    B = \diag(b_1, b_2, \ldots, b_n) \in \bC^{n\times n},
\end{equation}
and $Q(\cdot) =: (q_{jk}(\cdot))_{j,k=1}^n \in L^1([0,1];
\bC^{n\times n})$ is a potential matrix.

Note that, systems~\eqref{eq:system} form a more general object
than ordinary differential equations. Namely, the $n$th order
ODE~\eqref{eq:ODE} can be reduced to the
system~\eqref{eq:system} with ${b}_j = \exp\left(2\pi
ij/n\right)$ (see~\cite{Mal99}). Nevertheless, in general a BVP
for ODE~\eqref{eq:ODE} is not reduced to a
BVP~\eqref{eq:system}--\eqref{eq:CY(0)+DY(1)=0} (see below).
Systems~\eqref{eq:system} are of significant interest in some
theoretical and practical questions. For instance, if $n=2m$,
$B=\diag(-I_m, I_m)$ and $Q=\begin{pmatrix} 0 & Q_{12} \\ Q_{21}
& 0 \end{pmatrix}$, system~\eqref{eq:system} is equivalent to
the Dirac system (see~\cite[{Section
VII.1}]{LevSar88},~\cite[{Section 1.2}]{Mar77}). Note also that
equation~\eqref{eq:system} is used to integrate the $N$-waves
problem arising in nonlinear optics~\cite[{Sec.
III.4}]{ZMNovPit80}.

With system~\eqref{eq:system} one associates, in a natural way,
the maximal operator $L = L(Q)$ acting in $L^2([0,1]; \bC^n)$ on
the domain
\begin{equation*}
    \dom(L) = \{y \in W^{1,1}([0,1]; \bC^n) : Ly \in L^2([0,1]; \bC^n)\}.
\end{equation*}
To obtain a BVP, equation~\eqref{eq:system} is subject to the
following boundary conditions
\begin{equation} \label{eq:CY(0)+DY(1)=0}
    Cy(0) + D y(1)=0, \qquad C = (c_{jk}),\ \
    D = (d_{jk}) \in \bC^{n\times n}.
\end{equation}
Denote by $L_{C, D} := L_{C, D}(Q)$ the operator associated in
$L^2([0,1]; \bC^n)$ with the
BVP~\eqref{eq:system}--\eqref{eq:CY(0)+DY(1)=0}. It is defined
as the restriction of $L = L(Q)$ to the domain
\begin{equation} \label{eq:dom}
    \dom(L_{C, D}) = \{y \in \dom(L) : \ Cy(0) + D y(1)=0\}.
\end{equation}
Moreover, in what follows we always impose the maximality
condition
\begin{equation} \label{eq:rankCD}
    {\rm rank}\begin{pmatrix} C & D \end{pmatrix} = n,
\end{equation}
which is equivalent to $\ker(CC^*+DD^*)=\{0\}$.

To the best of our knowledge, the spectral problem
\eqref{eq:system}--\eqref{eq:CY(0)+DY(1)=0} has first been
investigated by G.D.~Birkhoff and R.E.~Langer~\cite{BirLan23}.
Namely, they have extended some previous results of~Birkhoff and
Tamarkin on non-selfadjoint boundary value problem for
ODE~\eqref{eq:ODE} to the case of
BVP~\eqref{eq:system}--\eqref{eq:CY(0)+DY(1)=0}. More precisely,
they introduced the concepts of \emph{regular and strictly
regular boundary conditions}~\eqref{eq:CY(0)+DY(1)=0} and
investigated the asymptotic behavior of eigenvalues and
eigenfunctions of the corresponding operator $L_{C, D}$.
Moreover, they proved \emph{a pointwise convergence result} on
spectral decompositions of the operator $L_{C, D}$ corresponding
to the BVP~\eqref{eq:system}--\eqref{eq:CY(0)+DY(1)=0} with
regular boundary conditions.

The problem of the completeness of the system of root functions
\emph{of general
BVP}~\eqref{eq:system}--\eqref{eq:CY(0)+DY(1)=0} has first been
investigated in the recent papers~\cite{MalOri00, MalOri10,
MalOri12} by one of the authors and L.L.~Oridoroga. In these
papers the concept of \emph{weakly regular} boundary conditions
for the system~\eqref{eq:system} was introduced and the
completeness of root vectors for this class of BVP was proved.
During the last decade there appeared numerous papers devoted
mainly to the Riesz basis property for $2\times 2$ Dirac system
subject to the \emph{regular or strictly regular} boundary
conditions (see~\cite{TroYam01, TroYam02, Mit04, DjaMit06Zone,
HasOri09, Bask11, DjaMit10BariDir, DjaMit12UncDir,
DjaMit12TrigDir, DjaMit12Equi, DjaMit12Crit, DjaMit13CritDir}).

Let us recall the definition of regular (see~\cite[{p.
89}]{BirLan23}) and weakly regular
(see~\cite{MalOri00,MalOri12}) boundary conditions. To this end
we need the following construction. Let $A=\diag(a_1,\dots,a_n)$
be a diagonal matrix with entries $a_k$ (not necessarily
distinct) that are not lying on the imaginary axis, $\Re\, a_k
\ne 0$. Starting from arbitrary matrices $C, D\in \bC^{n\times
n}$, we define the auxiliary $n \times n$ matrix $T_A(C,D)$ as
follows:
\begin{itemize}
    \item if $\Re\, a_k<0$, then the $k$th column in the
        matrix $T_A(C,D)$ coincides with the $k$th column of
        the matrix $C$,
    \item if $\Re\, a_k>0$, then the $k$th column in the
        matrix $T_A(C,D)$ coincides with the $k$th column of
        the matrix $D$.
\end{itemize}

Now consider the lines
\begin{equation}
    l_j := \{\lambda \in \bC : \Re(i b_j \lambda)=0\}, \quad j \in \{1,\ldots,n\},
\end{equation}
of the complex plane. They divide the complex plane into $m=2r
\le 2n$ sectors. Denote these sectors by $\sigma_1, \sigma_2,
\ldots \sigma_m$. Let $z_j$ lie in the interior of $\sigma_j,
j\in \{1,\ldots, m\}$. The boundary
conditions~\eqref{eq:CY(0)+DY(1)=0} are called regular whenever
\begin{equation} \label{eq:detTizBCD}
    \det T_{i z_j B}(C,D) \ne 0, \qquad j \in \{1,\ldots, m\}.
\end{equation}
We call $z \in \bC$ \emph{admissible} if $\Re(i b_j z) \ne 0$
for $j \in \{1,\ldots,n\}$. Since $T_{i z_j B}(C,D)$ does not
depend on a particular choice of the point $z_j \in \sigma_j$,
the boundary conditions~\eqref{eq:CY(0)+DY(1)=0} are regular if
and only if $\det T_{izB}(C,D) \ne 0$ for each admissible $z$.
%
%
\begin{definition} \emph{(\cite{MalOri12})} \label{def:weak.regular}
The boundary conditions~\eqref{eq:CY(0)+DY(1)=0} are called
\emph{weakly $B$-regular} (or, simply, weakly regular) if there
exist three admissible complex numbers $z_1, z_2, z_3$
satisfying the following conditions:

(a) the origin is an interior point of the triangle $\triangle
_{z_1z_2z_3};$

(b) $\det \,T_{i z_j B}(C, D) \ne 0$ for $j\in \{1,2,3\}$.
\end{definition}
%
%
In the case of Dirac type system ($B = B^*$) the weak regularity
of boundary conditions~\eqref{eq:CY(0)+DY(1)=0} is equivalent to
their regularity~\eqref{eq:detTizBCD} and turns into
\begin{equation} \label{eq:detT+-}
    \det T_{\pm} := \det(C P_{\mp} + D P_{\pm}) \ne 0.
\end{equation}
Here $P_+$ and $P_-$ denote the spectral projections onto
"positive"\ and "negative"\ parts of the spectrum of $B=B^*$,
respectively. Therefore, by~\cite[Theorem 1.2]{MalOri12}, this
condition implies the completeness and minimality in $L^2([0,1];
\bC^n)$ of the root functions of
BVP~\eqref{eq:system}--\eqref{eq:CY(0)+DY(1)=0}. In special
cases this statement has earlier been obtained by
V.A.~Marchenko~\cite[\S1.3]{Mar77} ($2\times 2$ Dirac system)
and V.P.~Ginzburg~\cite{Gin71} $(B=I_n, Q=0)$.

Our first main result (Theorem~\ref{th:explicit.nxn}) states the
completeness property for the general
BVP~\eqref{eq:system}--\eqref{eq:CY(0)+DY(1)=0} with non-weakly
regular boundary conditions. It substantially generalizes the
corresponding results from~\cite{MalOri12}
and~\cite{AgiMalOri12}. \emph{Emphasize that in the case of
non-weakly regular boundary conditions the completeness property
substantially depends on the values $Q(0)$ and $Q(1)$.} The
latter means that Theorem~\ref{th:explicit.nxn} cannot be
treated as a perturbation theory result: the operator
$L_{C,D}(Q)$ satisfying the conditions of this theorem is
complete while the system of root vectors of the unperturbed
operator $L_{C,D}(0)$ may have infinite defect in $L^2([0,1];
\bC^{n \times n})$. We demonstrate this fact by the
corresponding examples (cf. Corollary~\ref{cor:y1(0)=0}).

Our second main achievement is the Riesz basis property for
general $n \times n$ Dirac type system with $Q \in
L^{\infty}([0,1]; \bC^{n \times n})$ subject to certain boundary
conditions. These conditions form rather broad class that
covers, in particular, periodic, antiperiodic, and regular
splitting (not necessarily selfadjoint) boundary conditions for
$2n \times 2n$ Dirac system $(B = \diag(-I_n,I_n))$ (see
Theorem~\ref{th:basis.LCD} and
Proposition~\ref{prop:basis.LCD.per} for the precise
statements). Emphasize that to the best of our knowledge
\emph{even for $2n \times 2n$ Dirac systems with $n>1$ the
results on the Riesz basis property are obtained here for the
first time}.

In this connection we mention the series of recent papers by
P.~Djakov and B.S.~Mityagin~\cite{DjaMit12UncDir,
DjaMit12TrigDir, DjaMit12Equi, DjaMit12Crit, DjaMit13CritDir}.
In~\cite{DjaMit12UncDir} the authors proved that the system of
root functions for $2 \times 2$ Dirac system with $Q \in
L^2([0,1]; \bC^{2 \times 2})$ subject to the \emph{regular}
boundary conditions forms \emph{a Riesz basis with parentheses}
while this system forms \emph{ordinary Riesz basis} provided
that the boundary conditions are \emph{strictly regular}.
Moreover,
in~\cite[Theorem~13]{DjaMit12TrigDir},~\cite[Theorem~19]{DjaMit12Crit}
and~\cite{DjaMit13CritDir} it is established a \emph{criterion}
for the system of root functions to contain a Riesz basis for
periodic (resp., antiperiodic) $2 \times 2$ Dirac operator in
terms of the Fourier coefficients of $Q$ as well as in terms of
periodic (resp., antiperiodic) and Dirichlet spectra.

Further, it is worth to mention one more our result: any
dissipative BVP~\eqref{eq:system}--\eqref{eq:CY(0)+DY(1)=0}
admits the spectral synthesis in $L^2([0,1]; \bC^n)$ (see
Theorem~\ref{th:synthesis}). The spectral synthesis problem was
originated by J.~Wermer~\cite{Wermer52} and then studied by many
authors (see~\cite{Brod66,Markus70,Nik80} and references
therein). We also mention recent
preprints~\cite{BarBelBor12,BarBelBorYak12,BarYak12,BarYak13}
devoted to the problems of completeness and spectral synthesis
for singular perturbations of selfadjoint operators and systems
of exponents and to problems of removal of spectrum.

Note in this connection that each dissipative boundary value
problem for equation~\eqref{eq:ODE} with $n \geqslant 2$ admits
the spectral synthesis. The proof of this fact substantially
relies on two main ingredients:
\begin{itemize}
\item the resolvent of any BVP for equation~\eqref{eq:ODE}
    is the trace class operator;
\item the dissipative boundary value problem is always
    complete
\end{itemize}
(see Remark~\ref{rem:ODE.synth} for details).

As distinct from the situation above, the resolvent of the Dirac
type operator $L_{C,D}(Q)$ is no longer in trace class (see
Proposition~\ref{prop:RL.in.S1+KyFan}). Moreover, the system of
root vectors of the dissipative operator $L_{C,D}(Q)$ may be
incomplete (see, for instance~\cite[Remark 5.10]{MalOri12}).
Thus, the problem of spectral synthesis is non-trivial in this
case.

Finally, we apply our main abstract results with $B = B^* \in
\bC^{4 \times 4}$ to the Timoshenko beam model investigated
under the different restrictions in numerous papers
(see~\cite{Tim21,Tim55,KimRen87,Shub02,Souf03,XuYung04,XuHanYung07,WuXue11}
and the references therein). We show in
Proposition~\ref{prop:Tim.similar} that the dynamic generator of
this model is similar to the special $4 \times 4$ Dirac type
operator. It allows us to derive completeness property in both
regular and non-regular cases. Moreover, \emph{in the regular
case we obtain also the Riesz basis property with parentheses}.

The paper is organized as follows. In Section~\ref{sec:prelim}
we obtain the general result on completeness that
generalizes~\cite[Theorem 1.2]{MalOri12}. In
Section~\ref{sec:asymp} we obtain refined asymptotic formulas
for solutions of system~\eqref{eq:system} and the characteristic
determinant $\Delta(\cdot)$ of the
problem~\eqref{eq:system}--\eqref{eq:CY(0)+DY(1)=0}, provided
that the potential matrix $Q(\cdot)$ is continuous at the
endpoints 0 and 1.

In Section~\ref{sec:completeness} we prove our main result on
completeness, Theorem~\ref{th:explicit.nxn}. We illustrate this
result in $2 \times 2$ case by deriving completeness and
minimality in $L^2([0,\pi]; \bC^2)$ of the system
\begin{equation}
    \left\{\col(e^{a n x} \sin n x, n e^{(a-i) n x})\right\}_{n \in \bZ \setminus \{0\}}.
\end{equation}
We also obtain some necessary conditions on completeness for
general BVP~\eqref{eq:system}--\eqref{eq:CY(0)+DY(1)=0}
generalizing~\cite[Proposiiton 5.12]{MalOri12} and coinciding
with it in the case of $2 \times 2$ Dirac system.

In Section~\ref{sec:riesz} we prove the mentioned above results
on the Riesz basis property with parentheses for
BVP~\eqref{eq:system}--\eqref{eq:CY(0)+DY(1)=0} with a bounded
potential matrix.
In Section~\ref{sec:resolv} we discuss different properties of
the resolvent operator $(L_{C,D}(Q) - \l)^{-1}$. In particular,
we show that the resolvent difference of two operators
$L_{C_1,D_1}(Q_1)$ and $L_{C_2,D_2}(Q_2)$ is trace class
operator (Theorem~\ref{th:Resolv.dif.in.S1}). Using this result
we prove mentioned above result on spectral synthesis for
dissipative Dirac type operators as well as obtain some explicit
conditions in terms of the matrices $B, C, D, Q(\cdot)$ for the
operator $L_{C,D}(Q)$ to admit the spectral synthesis (see
Theorem~\ref{th:accum}).

Finally, in Section~\ref{sec:Timoshenko} we prove mentioned
above results on the completeness and Riesz basis property with
parentheses for the dynamic generator of spatially
non-homogenous Timoshenko beam model with both boundary and
locally distributed damping.

The main results of
Sections~\ref{sec:prelim}--\ref{sec:completeness} have been
announced in~\cite{LunMal13Dokl}.

{\bf Notation.} $\left\langle \cdot,\cdot\right\rangle$ denotes
the inner product in $\bC^{n}$; $\bC^{n\times n}$ denotes the
set of $n\times n$ matrices with complex entries. $I_n (\in
\bC^{n\times n})$ denotes the identity matrix; $\GL (n, \bC)$
denotes the set of nonsingular matrices from $\bC^{n\times n}$;
$W^{n,p}[a,b]$ is Sobolev space of functions $f$ having $n-1$
absolutely continuous derivatives on $[a,b]$ and satisfying
$f^{(n)}\in L^p[a,b]$.

$T$ is a closed operator in a Hilbert space $\fH$; $\sigma(T)$
and $\rho(T) = \bC \setminus \sigma(T)$ denote the spectrum and
resolvent set of the operator $T$, respectively.

$\fS_p(\fH)$, $0 < p \leqslant \infty$, denotes the
Neumann-Schatten ideals in a Hilbert space $\fH$. In particular,
$\fS_{\infty}(\fH)$ is the ideal of compact operators.
$\fS_p(\fH)$ is a two-sided ideal in algebra $\cB(\fH)$ of
bounded linear operators.
%
%
\section{Preliminaries} \label{sec:prelim}
%
%
In what follows we will systematically use the following simple
lemma.
%
%
\begin{lemma} \label{lem:adjoint}
Let $L_{C,D}(Q)$ be the operator defined
by~\eqref{eq:system}--\eqref{eq:rankCD}. Then there exist
matrices $C_*, D_* \in \bC^{n \times n}$ such that $\rank
\begin{pmatrix} C_* & D_* \end{pmatrix} = n$ and the adjoint
operator $L_{C,D}^* := (L_{C,D}(Q))^*$ coincides with the
restriction of the maximal differential operator
\begin{equation}
\begin{split}
    L_* y &:= -i (B^*)^{-1} y' + Q^*(x) y, \\
    \dom(L_*) &= \{y \in AC([0,1]; \bC^n) : L_* y \in L^2([0,1];\bC^n)\},
\end{split}
\end{equation}
to the domain
\begin{equation}
    \dom(L_{C,D}^*) = \{y \in \dom(L_*) : C_* y(0) + D_* y(1) = 0\}.
\end{equation}
In particular, if $B=B^*$ then $L_{C,D}^* = L_{C_*,D_*}(Q^*)$.
\end{lemma}
%
%
Let $\beta_1,\ldots,\beta_r$ be all different values among
$b_1,\ldots,b_n$. Note that the lines
\begin{equation} \label{eq:ljk}
    l_{jk} := \{ \l \in \bC : \Re(i \beta_j \l) = \Re(i \beta_k \l) \},
    \quad 1 \leqslant j < k \leqslant r,
\end{equation}
together with the lines
\begin{equation} \label{eq:lj}
    l_j := \{ \l \in \bC : \Re(i \beta_j \l) = 0 \}, \quad j \in \{1,\ldots,r\},
\end{equation}
separate $\nu \leqslant r^2+r$ open sectors $S_p$ with vertexes
at the origin, such that for any $p \in \{1,\ldots,\nu\}$ the
numbers $\beta_1, \ldots, \beta_r$ can be renumbered so that the
following inequalities hold:
\begin{equation} \label{eq:i.bj.sorted}
    \Re(i \beta_{j_1}\l) < \ldots < \Re(i \beta_{j_{\k}}\l) < 0
    < \Re(i \beta_{j_{\k+1}}\l) < \ldots < \Re(i \beta_{j_r}\l),
    \quad \l \in S_p.
\end{equation}
Here $\k = \k_p$ is the number of negative values among $\Re(i
\beta_1 \l), \ldots, \Re(i \beta_r \l)$ in the sector $S_p$. We
call $z \in \bC$ \emph{feasible} if $z$ does not belong to any
of the lines~\eqref{eq:ljk} and~\eqref{eq:lj}, that is, $z$ lies
strictly inside some sector $S_p$. Note that feasible point is
more restrictive notion than admissible point.

Clearly, each of the sectors $S_p$ is of the form $S_p = \{ z :
\varphi_{1p} < \arg z < \varphi_{2p}\}$. Denote by $S_{p,\eps}$
a sector strictly embedded into the latter, i.e.,
\begin{align}
    \label{eq:Spe} S_{p,\eps} &:=\{z: \varphi_{1p} + \eps < \arg z < \varphi_{2p} - \eps\},
    \quad \text{where}\ \ \eps>0 \ \ \text{is sufficiently small}; \\
    \label{eq:SpeR} S_{p,\eps,R} &:= \{z\in S_{p,\eps}: |z|>R\}.
\end{align}
%
%
\begin{proposition} \label{prop:BirkSys}
\emph{\cite[Proposition 2.2]{MalOri12}} Let $\delta_{jk}$ be a
Kronecker symbol, let
\begin{gather}
    \label{eq:B.beta}
        B = \diag(\beta_1 I_{n_1}, \ldots, \beta_r I_{n_r}), \quad n_1+\ldots+n_r=n, \\
    \label{eq:Q=Qjk}
        Q = (Q_{jk})^r_{j,k=1}, \quad Q_{jk} \in L^1([0,1]; \bC^{n_j \times n_k}), \\
    \label{eq:Qjj=0}
        Q_{jj}(\cdot) \equiv 0, \quad j \in \{1,\ldots,r\}.
\end{gather}
Let also $p \in \{1,\ldots,\nu\}$ and let $\eps > 0$ be
sufficiently small. Then for a sufficiently large $R$,
equation~\eqref{eq:system} has a fundamental matrix solution
\begin{equation} \label{eq:Yxl}
    Y(x,\l) = \begin{pmatrix} Y_1 & \ldots & Y_n \end{pmatrix},
    \quad Y_k(x,\l) = \col(y_{1k}, \ldots, y_{nk}),
    \quad k \in \{1,\ldots,n\},
\end{equation}
which is analytic in $\l \in S_{p,\eps,R}$ and satisfies
$($uniformly in $x \in [0,1])$
\begin{equation} \label{eq:yjk=(1+o(1)).e(i.bk.l)}
    y_{jk}(x,\l) = (\delta_{jk}+o(1)) e^{i b_k \l x},
    \quad\text{as}\quad \l \to \infty,\ \l \in S_{p,\eps,R},
    \quad j,k \in \{1,\ldots,n\}.
\end{equation}
\end{proposition}
%
%
In what follows we will systematically use a concept of the
similarity of unbounded operators.
%
%
\begin{definition} \label{def:similar}
Let $\fH_j$ be a Hilbert space, $A_j$ a closed operator in
$\fH_j$ with domain $\dom(A_j)$, $j\in \{1,2\}$. The operators
$A_1$ and $A_2$ are called similar if there exists a bounded
operator $T$ (a similarity transformation operator) from $\fH_1$
onto $\fH_2$ with bounded inverse, such that $A_2 = T A_1
T^{-1}$, i.e.
\begin{equation}
    \dom(A_2) = T \dom(A_1) \quad\text{and}\quad A_2 f = T A_1 T^{-1} f, \quad f \in \dom(A_2).
\end{equation}
\end{definition}
%
%
Note that similar operators $A_1$ and $A_2$ ($A_2 = T A_1
T^{-1}$) have the same spectra, algebraic and geometric
multiplicities of eigenvalues, while the systems of their root
vectors $\{e_k^{(j)}\}$, $j\in \{1,2\}$, are related by
$e_k^{(2)} = T e_k^{(1)}$. Therefore, they also have the same
geometric properties (completeness, minimality, basis property,
etc.).

Let $\Phi(x,\l)$ be a fundamental matrix solution of
equation~\eqref{eq:system} satisfying
\begin{equation}\label{eq:Phi0=In}
    \Phi(0,\l)=I_n, \quad \l \in \bC.
\end{equation}
The characteristic determinant $\Delta(\cdot)$ of the
problem~\eqref{eq:system}--\eqref{eq:CY(0)+DY(1)=0} is given by
\begin{equation} \label{eq:Delta(l).def}
    \Delta(\l) := \det(C+D\Phi(1,\l)), \quad \l \in \bC.
\end{equation}
Next we prove the completeness result which slightly
generalizes~\cite[Theorem~1.2]{MalOri12}.
%
%
\begin{theorem} \label{th:compl.gen}
Let $Q(\cdot) \in L^1([0,1]; \bC^{n\times n})$. Assume that
there exist $C,R>0$, $s \in \bZ_+$ and three feasible numbers
${z_1,z_2,z_3}$ satisfying the following conditions:

(i) the origin is the interior point of the triangle
$\Delta_{z_1z_2z_3}$;

(ii) for $k \in \{1,2,3\}$ we have
\begin{equation} \label{eq:Delta>=e/l}
    |\Delta(\l)| \geqslant \frac{C e^{\Re(i \tau_k \l)}}{|\l|^s},
    \quad \tau_k = \sum_{j=1 \atop{\Re(i b_j z_k) > 0}}^n b_j,
    \quad |\l| > R, \ \arg\l=\arg z_k.
\end{equation}
Then the system of root functions of the
BVP~\eqref{eq:system}--\eqref{eq:CY(0)+DY(1)=0} $($of the
operator $L_{C,D}(Q))$ is complete and minimal in $L^2([0,1];
\bC^n)$.
\end{theorem}
%
%
\begin{remark}
In the case $s=0$ Theorem~\ref{th:compl.gen} is implicitly
contained in~\cite[Theorem 1.2]{MalOri12}. Our proof follows the
scheme proposed in~\cite[Theorem 1.2]{MalOri12}.
\end{remark}
%
%
\begin{proof}[Proof of Theorem~\ref{th:compl.gen}]
By renumbering $y_1, \ldots, y_n$ we can assume that the matrix
$B$ satisfies~\eqref{eq:B.beta} and hence $Q$ has
representation~\eqref{eq:Q=Qjk}. Let
\begin{equation} \label{eq:Q1.def}
    Q_1(x) := \diag(Q_{11}(x), \ldots, Q_{rr}(x))
\end{equation}
and let $W(\cdot)$ be the solution to the Cauchy problem
\begin{equation} \label{eq:iBW'=Q1.W}
    i B^{-1} W' = Q_1(x) W, \qquad W(0)=I_n.
\end{equation}
Due to the block structure of the matrices $B$ and $Q_1$ one
easily derives
\begin{equation} \label{eq:Wx}
    W(x) = \diag(W_{11}(x),\ldots,W_{rr}(x)), \quad W_{jj}(x) \in \GL(n_j, \bC), \quad x\in [0,1].
\end{equation}
Denoting by  $W : y \to W(x)y$ the gauge transform  and letting
\begin{equation} \label{eq:wtD.wtQ}
    \wt{D} := D W(1) \quad\text{and}\quad  \wt{Q}(x) = W^{-1}(x)(Q(x)-Q_1(x))W(x) =: (\wt{q}_{jk}(x))_{j,k=1}^n
\end{equation}
we get
\begin{equation}
    L_{C,\wt{D}}(\wt{Q}) = W^{-1} L_{C,{D}}({Q}) W,
\end{equation}
i.e. $L_{C,D}(Q)$ and $L_{C,\wt{D}}(\wt{Q})$ are similar.
Clearly, $\wt{\Phi} =: W^{-1}\Phi$ is a fundamental solution of
equation~\eqref{eq:system} with $\wt{Q}$ in place of $Q$ and the
corresponding characteristic determinant $\wt{\Delta}(\cdot)$
(see~\eqref{eq:Delta(l).def}) is
\begin{equation} \label{eq:wtDelta=Delta}
    \wt{\Delta}(\l) :=  \det(C + \wt{D} \wt{\Phi}(1,\l))
    = \det(C + {D} W(1) W^{-1}(1)\Phi(1,\l)) = {\Delta}(\l).
\end{equation}
So, above gauge transformation does not change the
characteristic determinant. Therefore, replacing if necessary,
$L_{C,D}(Q)$ by ${L}_{C,\wt{D}}(\wt{Q})$ we can assume that
conditions~\eqref{eq:B.beta}--\eqref{eq:Qjj=0} are satisfied.

Further, let $\Psi(x,\l)$ be a fundamental $n\times n$ matrix
solution of equation~\eqref{eq:system} in a domain $S$, i.e.
\begin{equation} \label{eq:det(Psi(0,l))!=0}
    \det(\Psi(0,\l)) \ne 0, \quad \l \in S.
\end{equation}
Denote by $\Psi_k(x,\l)$ the $k$th vector column of the matrix
$\Psi(x,\l)$, i.e.,
\begin{equation}\label{eq:Psi=Psi1...Psin}
    \Psi(x,\l)=\begin{pmatrix}\Psi_1 & \ldots & \Psi_n\end{pmatrix},\quad
    \Psi_k(x,\l)=\col(\psi_{1k}, \ldots, \psi_{nk}).
\end{equation}
Further, denote
\begin{gather}
    \label{eq:A.Psi(l).def}
        A_{\Psi}(\l) := C \Psi(0,\l)+D \Psi(1,\l), \\
    \label{eq:Delta.Psi(l).def}
        \Delta_{\Psi}(\l) := \det A_{\Psi}(\l)=\det (C \Psi(0,\l)+D \Psi(1,\l)).
\end{gather}
Clearly $\Delta(\cdot) = \Delta_{\Phi}(\cdot)$. Denote by
$\wt{A}_{\Psi}(\l)=(\Delta^{jk}_{\Psi}(\l))_{j,k=1}^n$ the
adjugate matrix, that is,
\begin{equation} \label{eq:A.Psi.adjugate}
    A_{\Psi}(\l) \cdot \wt{A}_{\Psi}(\l)
    = \wt{A}_{\Psi}(\l) \cdot A_{\Psi}(\l) = \Delta_{\Psi}(\l) I_n,
\end{equation}
and introduce the vector functions
\begin{align}
    \label{eq:UPsij(x,l).def}
        U_{\Psi,j}(x,\l) &:= \sum_{k=1}^n \Delta^{jk}_{\Psi}(\l) \Psi_k(x,\l),
        \qquad j \in \{1, 2, \dots, n\}, \\
    \label{eq:UPsi(x,l).def}
        U_{\Psi}(x,\l) &:= \begin{pmatrix} U_{\Psi,1}(x,\l) & \ldots & U_{\Psi,n}(x,\l)\end{pmatrix}
        = \Psi(x,\l) \wt{A}_\Psi(\l).
\end{align}
The spectrum $\sigma(L_{C,D})$ of the
problem~\eqref{eq:system}--\eqref{eq:CY(0)+DY(1)=0} coincides
with the set of roots of the characteristic determinant
$\Delta(\cdot)=\Delta_{\Phi}(\cdot)$. Assumption (ii) of Theorem
\ref{th:compl.gen} yields the relation $\Delta(\l)\not \equiv
0$. Therefore, the spectrum $\sigma(L_{C,D})$ of the
problem~\eqref{eq:system}--\eqref{eq:CY(0)+DY(1)=0} is discrete,
i.e., $\sigma(L_{C,D})$ consists of at most countably many
eigenvalues $\{\l_k\}_{k=1}^{N},\ N \leqslant \infty$, of finite
algebraic multiplicities. Let $\l_k$ be an $m_k$-multiple zero
of the function $\Delta(\l)$. As shown in the step (i) of the
proof of~\cite[Theorem 1.2]{MalOri12} the system of functions
\begin{equation}\label{eq:DpUj}
    \left\{\left. \frac{\partial^p}{\partial\l^p} U_{\Phi,j}(x,\l) \right|_{\l=\l_k}
    : \ \ p \in \{0,1,\ldots,m_k-1\},\ \ j \in \{1,\ldots,n\}\right\}
\end{equation}
spans the root subspace $\cR_{\l_k}(L_{C,D})$ of the operator
$L_{C,D}$, where $U_{\Phi,j}(x,\l)$ is defined by the
formula~\eqref{eq:UPsij(x,l).def} for the solution $\Phi(x,\l)$
in place of $\Psi(x,\l)$. Note that $\Phi(x,\l)$ as well as
$U_{\Phi,j}(x,\l)$ and $\Delta(\l)$ are entire functions of
exponential type.

We prove the completeness of union of systems~\eqref{eq:DpUj}
for all $k$ by contradiction. To this end, we assume that there
exists a non-zero vector function $f=\col(f_1,\dots, f_n) \in
L^2([0,1]; \bC^n)$ orthogonal to this system. Consider the
entire functions
\begin{equation} \label{eq:Fj(l).def}
    F_j(\l):=(U_{\Phi,j}(\cdot,\l),f(\cdot))_{L^2([0,1]; \bC^n)},
    \quad j \in \{1,\ldots,n\}.
\end{equation}
Since $f$ is orthogonal to the system~\eqref{eq:DpUj} then each
$\l_k(\in\sigma(L_{C,D}))$ is a zero of $F_j(\cdot)$ of
multiplicity at least $m_k$, i.e. for $\l_k \in \sigma(L_{C,D})$
\begin{equation} \label{eq:Fjpl=0}
    \left. F_j^{(p)}(\l) \right|_{\l=\l_k}=0, \qquad
    p\in\{0,1,\dots,m_k-1\},\quad j \in \{1,\ldots,n\}.
\end{equation}
Thus, the ratio
\begin{equation} \label{eq:Gj(l).def}
    G_j(\l):=\frac{F_j(\l)}{\Delta(\l)}, \quad j \in \{1,\ldots,n\},
\end{equation}
is an entire function. Moreover, since functions
$U_{\Phi,j}(x,\l)$ and $\Delta(\l)$ are entire functions of
exponential type then so are $G_1(\l), \ldots, G_n(\l)$. Let us
prove that these functions are polynomials in $\l$ by estimating
their growth. Denote
\begin{equation}
    G(\l) := \begin{pmatrix} G_1(\l) & \ldots & G_n(\l) \end{pmatrix}.
\end{equation}
It follows from~\eqref{eq:Fj(l).def} and~\eqref{eq:Gj(l).def}
that
\begin{equation} \label{eq:int.f*U=Delta.G}
    \int_0^1 f^*(x) U_{\Phi}(x,\l) dx = \Delta(\l) G(\l), \quad \l \in \bC,
\end{equation}
where $f^*(x) := \begin{pmatrix}\overline{f_1(x)} & \ldots &
\overline{f_n(x)}\end{pmatrix} = \overline{f(x)}^T$.

Multiplying~\eqref{eq:int.f*U=Delta.G} by the matrix
$A_{\Phi}(\cdot)$ from the right we get in view
of~\eqref{eq:UPsi(x,l).def} and~\eqref{eq:A.Psi.adjugate}
\begin{equation}
    \Delta(\l) \int_0^1 f^*(x) \Phi(x,\l) dx
    = \Delta(\l) G(\l) A_{\Phi}(\l), \quad \l \in \bC,
\end{equation}
or equivalently
\begin{equation}
    \int_0^1 f^*(x) \Phi(x,\l) dx
    = G(\l) A_{\Phi}(\l), \quad \l \not \in \sigma(L_{C,D}).
\end{equation}
Now the continuity of the integral in the last equality with
respect to $\l$, the discreteness of the set $\sigma(L_{C,D})$
and definition of $A_{\Phi}(\l)$ (see
formula~\eqref{eq:A.Psi(l).def}) yields the following relation
\begin{equation} \label{eq:int.f*Phi=G.A_Phi}
    \int_0^1 f^*(x) \Phi(x,\l) dx
    = G(\l) (C \Phi(0,\l) + D \Phi(1,\l)), \quad \l \in \bC.
\end{equation}
Let $\Psi(x,\l)$ be a fundamental $n\times n$ matrix solution of
the equation~\eqref{eq:system} in a domain $S$. Due to the
initial condition $\Phi(0,\l) = I_n$ the matrix functions
$\Phi(x,\l)$ and $\Psi(x,\l)$ are related by
\begin{equation}\label{eq:Psi(x)=Phi.Psi(0)}
    \Psi(x,\l) = \Phi(x,\l) \Psi(0,\l),\qquad x\in [0,1],
    \quad \l\in S,
\end{equation}
where $\Psi(0,\l)$ is invertible matrix function for $\l \in S$.
Multiplying~\eqref{eq:int.f*Phi=G.A_Phi} by $\Psi(0,\l)$ from
the right we get
\begin{equation} \label{eq:int.f*Psi=G.A_Psi}
    \int_0^1 f^*(x) \Psi(x,\l) dx
    = G(\l) (C \Psi(0,\l) + D \Psi(1,\l)), \quad \l \in S.
\end{equation}
Now multiplying~\eqref{eq:int.f*Psi=G.A_Psi} by
$\wt{A}_{\Psi}(x,\l)$ from the right we get with account
of~\eqref{eq:UPsi(x,l).def} and~\eqref{eq:A.Psi.adjugate}
\begin{equation}
    \int_0^1 f^*(x) U_{\Psi}(x,\l) dx = \Delta_{\Psi}(\l) G(\l), \quad \l \in S,
\end{equation}
or equivalently
\begin{equation} \label{eq:UPsij.f=Gj.Delta}
    F_{\Psi,j}(\l) := (U_{\Psi,j}(\cdot,\l), f(\cdot))_{L^2([0,1]; \bC^n)}
    = G_j(\l) \Delta_{\Psi}(\l), \quad \l \in S, \quad j \in \{1,\ldots,n\}.
\end{equation}
Let us estimate $G_j(\l)$ from above on the rays
\begin{equation}
    \Gamma_k := \{ \l \in \bC :\arg \l = \arg z_k\}, \quad k \in \{1,2,3\},
\end{equation}
using relation~\eqref{eq:UPsij.f=Gj.Delta} for appropriate
solutions $\Psi(x,\l)$.

Let $k \in \{1,2,3\}$ be fixed. Since $z_k$ is feasible then
$\Gamma_k$ lies inside some sector $S_p$ and hence $\Gamma_k \in
S_{p,\eps}$ for some $\eps>0$. According to
Proposition~\ref{prop:BirkSys} there exists a fundamental matrix
solution $Y(x,\l)$ of the system~\eqref{eq:system} with
asymptotic behavior~\eqref{eq:yjk=(1+o(1)).e(i.bk.l)} at the
domain $S_{p,\eps,R}$ for some $R>0$. It was shown in the proof
of~\cite[Theorem 1.2]{MalOri12} that for the function
$F_{Y,j}(\l)$ defined by~\eqref{eq:UPsij.f=Gj.Delta} one has
\begin{equation} \label{eq:FYj(l)=o(e)}
    F_{Y,j}(\l) = o(e^{i \tau_k \l}),
    \quad\text{as}\quad \l \to \infty,\ \l \in S_{p,\eps}.
\end{equation}
Next, it follows from~\eqref{eq:yjk=(1+o(1)).e(i.bk.l)} that
\begin{equation} \label{eq:Y(0)=I+o(1)}
    Y(0,\l) = I_n + o_n(1), \quad\text{as}\quad \l \to \infty,\ \l \in S_{p,\eps},
\end{equation}
where $o_n(1)$ denotes an $n\times n$ matrix function with
entries of the form $o(1)$ as $\l \to \infty$.
Now~\eqref{eq:A.Psi(l).def},~\eqref{eq:Delta.Psi(l).def},
\eqref{eq:Psi(x)=Phi.Psi(0)} with $Y$ in place of $\Psi$
and~\eqref{eq:Y(0)=I+o(1)} yields
\begin{equation} \label{eq:DeltaY(l)=(1+o(1))Delta(l)}
    \Delta_{Y}(\l) = \Delta(\l) \det Y(0,\l) = (1+o(1)) \Delta(\l),
    \quad\text{as}\quad \l \to \infty,\ \l \in S_{p,\eps}.
\end{equation}
Inserting~\eqref{eq:FYj(l)=o(e)},
\eqref{eq:DeltaY(l)=(1+o(1))Delta(l)},~\eqref{eq:Delta>=e/l}
into~\eqref{eq:UPsij.f=Gj.Delta} we get
\begin{equation} \label{eq:Gj<=l^m}
    \left|G_j(\l)\right| = \left|\frac{o(e^{i \tau_k \l})}{(1+o(1))\Delta(\l)}\right|
    \leqslant \frac{C_1 e^{\Re(i \tau_k \l)}|\l|^s}{C e^{\Re(i \tau_k \l)}} = C_2 |\l|^s,
    \quad \arg \l = \arg z_k, \ |\l|>R,
\end{equation}
for some $C_1 > 0$ with $C_2 = C_1 / C$.

Since zero is the interior point of the triangle $\triangle_{z_1
z_2 z_3}$, the rays $\Gamma_1, \Gamma_2, \Gamma_3$ divide the
complex plane into three closed sectors $\Omega_1, \Omega_2,
\Omega_3$ of opening less than $\pi$. Fix $k \in \{1,2,3\}$ and
apply the Phragm\'{e}n-Lindel\"{o}f theorem~\cite[Theorem
6.1]{Lev96} to the function $\wt{G}_j(\l)$ considered in the
sector $\Omega_k$. Using~\eqref{eq:Gj<=l^m} we get
\begin{equation}
    \left|G_j(\l)\right| \leqslant C_3 |\l|^s, \quad \l \in \Omega_k,
\end{equation}
for some $C_3 > 0$, and hence
\begin{equation}
    \left|G_j(\l)\right| \leqslant C_3 |\l|^s, \quad \l \in \bC.
\end{equation}
By the Liouville theorem (cf.~\cite[Theorem 1.1]{Lev96}),
$G_j(\l)$ is a polynomial of degree at most $s$.

Now let us prove that $G_j(\cdot) \equiv 0,\ j \in
\{1,\ldots,n\}$, using equality~\eqref{eq:int.f*Psi=G.A_Psi} for
appropriate solutions $\Psi(x,\l)$ and the fact that $G_j(\l)$
is a polynomial in $\l$. Putting~\eqref{eq:Psi=Psi1...Psin}
into~\eqref{eq:int.f*Psi=G.A_Psi} we get for $k \in
\{1,\ldots,n\}$
\begin{equation} \label{eq:sum.int.fj.psijk=sum.Gj...}
    \sum_{j=1}^n\int_0^1 \overline{f_j(x)} \psi_{jk}(x,\l) dx
    = \sum_{j=1}^n G_j(\l) \sum_{l=1}^n \left(c_{jl} \psi_{lk}(0,\l)
    + d_{jl} \psi_{lk}(1,\l)\right).
\end{equation}
Consider some sector $S_{p,\eps}$. Let $Y(x,\l)$ be a matrix
solution of equation~\eqref{eq:system}
satisfying~\eqref{eq:yjk=(1+o(1)).e(i.bk.l)} in $S_{p,\eps,R}$.
It follows from~\eqref{eq:yjk=(1+o(1)).e(i.bk.l)} that
\begin{equation} \label{eq:yjk<=Cjk.e}
    \left|y_{jk}(x,\l)\right| \leqslant C e^{\Re(i b_k \l) x},
    \quad j,k \in \{1,\ldots,n\}, \quad x \in [0,1], \quad \l \in S_{p,\eps,R}, \\
\end{equation}
for some $C>0$. Hence, by the Cauchy inequality,
\begin{align}\label{eq:int.f.yjk<=C.f.e}
    \left|\sum_{j=1}^n\int_0^1 \overline{f_j(x)} y_{jk}(x,\l) dx\right|
    & \leqslant C \|f\|_{L^2([0,1]; \bC^n)}
    \left(\int_0^1 e^{2\Re(i b_k \l) x} dx \right)^{1/2} \nonumber \\
    & \leqslant \frac{C \|f\|}{\sqrt{|\l|}}\max\{e^{\Re(i b_k \l)},1\},
    \quad \l \in S_{p,\eps,R}, \quad k \in \{1,\ldots,n\}.
\end{align}
Substituting~\eqref{eq:yjk=(1+o(1)).e(i.bk.l)}
and~\eqref{eq:int.f.yjk<=C.f.e}
into~\eqref{eq:sum.int.fj.psijk=sum.Gj...} with $Y$ in place of
$\Psi$ we get
\begin{multline}\label{eq:sum.Gj<=C.f.e}
    \left|\sum_{j=1}^n G_j(\l) \left(c_{jk}+d_{jk} e^{i b_k \l}
    + \sum_{l=1}^n \bigl(c_{jl} \cdot o(1)
    + d_{jl} \cdot o(1) \cdot e^{i b_k \l}\bigr)\right)\right| \\
    \leqslant \frac{C \|f\|}{\sqrt{|\l|}}\max\{e^{\Re(i b_k \l)},1\},
    \quad\text{as}\quad \l \to \infty,\ \l \in S_{p,\eps,R}, \quad k \in \{1,\ldots,n\}.
\end{multline}
Assume that
\begin{equation}
    d := \max\{\deg G_j : j \in \{1,\ldots,n\}\} \geqslant 0
\end{equation}
and denote by $\alpha_{j}$ the coefficient of $\l^d$ in
$G_j(\l)$. So
\begin{equation} \label{eq:Gj(l)=l^d(aj+o(1))}
    G_j(\l) = \l^d (\alpha_j + o(1)), \quad\text{as}\quad \l \to \infty,
    \quad j \in \{1,\ldots,n\}.
\end{equation}
From definition of $d$ it follows that $d = \deg G_{j_0}$ for
some $j_0 \in \{1,\ldots,n\}$ and hence $\alpha_{j_0} \ne 0$.
Therefore, $\alpha := \col(\alpha_1,\ldots,\alpha_n) \ne 0$.

Let us fix $k \in \{1,\ldots,n\}$. Without loss of generality we
may assume that
\begin{equation} \label{eq:Reiakl>0}
    \Re(i b_k \l) > 0,\quad \l \in S_{p,\eps}.
\end{equation}
Relation~\eqref{eq:Reiakl>0} yields
\begin{equation} \label{eq:cjk+djk.e...=(djk+o(1)).e}
    c_{jk}+d_{jk} e^{i b_k \l} + \sum_{l=1}^n \bigl(c_{jl} \cdot
    o(1) + d_{jl} \cdot o(1) \cdot e^{i b_k \l}\bigr) = (d_{jk} + o(1)) e^{i b_k \l},
    \qquad j \in \{1,\ldots,n\},
\end{equation}
as $\l \to \infty$, $\l \in S_{p,\eps}$.
Inserting~\eqref{eq:Gj(l)=l^d(aj+o(1))}
and~\eqref{eq:cjk+djk.e...=(djk+o(1)).e}
into~\eqref{eq:sum.Gj<=C.f.e} we get
\begin{equation}
    \bigl|\alpha_1 d_{1k} + \ldots + \alpha_n d_{nk} + o(1)\bigr|
    \cdot e^{\Re(i b_k \l)} \cdot |\l|^d
    \leqslant \frac{C \|f\|}{\sqrt{|\l|}}\max\{e^{\Re(i b_k \l)},1\},
    \quad \l \in S_{p,\eps,R}.
\end{equation}
In view of~\eqref{eq:Reiakl>0} this estimate yields
\begin{equation} \label{eq:a1.d1k+...=0}
    \alpha_1 d_{1k} + \ldots + \alpha_n d_{nk} = 0.
\end{equation}
Now consider the sector $S_{\wt{p},\eps}$ which is opposite to
$S_{p,\eps}$. In this sector due to~\eqref{eq:Reiakl>0} one has
\begin{equation} \label{eq:Reiakl<0}
    \Re(i b_k \l) < 0, \quad \l \in S_{\wt{p},\eps}.
\end{equation}
Let $\wt{Y}(x,\l)$ be a solution of system~\eqref{eq:system}
having asymptotic behavior~\eqref{eq:yjk=(1+o(1)).e(i.bk.l)} in
the sector $S_{\wt{p},\eps}$. Inserting $\wt{Y}(x,\l)$
into~\eqref{eq:sum.int.fj.psijk=sum.Gj...} in place of
$\Psi(\cdot,\cdot)$ we get similarly to the previous case that
\begin{equation}
    \bigl|\alpha_1 c_{1k} + \ldots + \alpha_n c_{nk} + o(1)\bigr| \cdot |\l|^d
    \leqslant \frac{C \|f\|}{\sqrt{|\l|}}\max\{e^{\Re(i b_k \l)},1\},
    \quad \l \in S_{\wt{p},\eps,R}.
\end{equation}
This estimate is compatible with~\eqref{eq:Reiakl<0} only if
\begin{equation} \label{eq:a1.c1k+...=0}
    \alpha_1 c_{1k} + \ldots + \alpha_n c_{nk} = 0.
\end{equation}
Since $k \in \{1,\ldots,n\}$ is arbitrary, combining
relations~\eqref{eq:a1.d1k+...=0} and~\eqref{eq:a1.c1k+...=0}
yields
\begin{equation}
    D^T \alpha = 0, \quad C^T \alpha = 0,
\end{equation}
which implies $\alpha = 0$ because of the maximality
condition~\eqref{eq:rankCD}. This contradicts the assumption $d
\geqslant 0$. Hence $G_j(\cdot) \equiv 0$ for $j \in
\{1,\ldots,n\}$.

Now it follows from~\eqref{eq:int.f*Phi=G.A_Phi} that
\begin{equation} \label{eq:Phi.f=0}
    \int_0^1 \bigl\langle \Phi_j(x,\l) , f(x) \bigr\rangle \,dx \equiv 0,
    \qquad \l \in \bC, \quad j \in \{1, 2, \ldots, n\}.
\end{equation}
By~\cite[Theorem 1.2, step (vi)]{MalOri12}, the vector function
$f$ satisfying~\eqref{eq:Phi.f=0} is zero, i.e. the system of
root functions of the operator $L_{C,D}(Q)$ is complete. Its
minimality is implied by~\cite[Lemma 2.4]{MalOri12} applied to
the operator $(L_{C,D}-\l)^{-1}$ with $\l \in \rho(L_{C,D})$.
\end{proof}
%
%
\section{Asymptotic behavior of solutions and characteristic
determinant} \label{sec:asymp}
%
%
Here we refine asymptotic
formulas~\eqref{eq:yjk=(1+o(1)).e(i.bk.l)} assuming that
$Q(\cdot)$ is continuous at the endpoints 0 and 1. These
formulas will be applied to investigate asymptotic behavior of
the characteristic determinant $\Delta(\cdot)$. We start with
the following lemma.
%
%
\begin{lemma}\label{lem:int.e.q}
Let $b \in \bC\setminus\{0\}$, $C>0$ and $S \subset \bC$ be a
non-bounded subset of $\bC$ such that
\begin{equation} \label{eq:Rebl<-C|l|}
\Re(b\l) < -C|\l|, \quad \l \in S.
\end{equation}

(i) Let $\varphi \in L^1[0,1]$ and $\varphi(\cdot)$ is
continuous at zero. Then
\begin{equation}\label{eq:le1i}
    \int_0^1 e^{b\l t} \varphi(t) dt = \frac{\varphi(0)+o(1)}{-b\l},
    \quad\text{as}\quad \l \to \infty,\ \l \in S.
\end{equation}

(ii) Let $\varphi \in L^1[0,1]$ and let $\varphi(\cdot)$ be
bounded at a neighborhood of zero. Then
\begin{equation}\label{eq:le1ii}
    \int_0^1 \left|e^{b\l t} \varphi(t)\right| dt = O(|\l|^{-1}),
    \quad \l \in S.
\end{equation}
\end{lemma}
%
%
\begin{proof}[Proof]
Taking into account~\eqref{eq:Rebl<-C|l|} one has
\begin{equation} \label{eq:int01.ebt}
    \int_0^1 \left|e^{bt} \varphi(t)\right| dt \leqslant
    \left(\int_0^\delta + \int_\delta^1\right) e^{-C|\lambda|t} |\varphi(t)| dt
    \leqslant \frac 1{C|\lambda|}\sup_{t \in [0, \delta]} |\varphi(t)| + \|\varphi\|_1 e^{-C\delta |\lambda|}.
\end{equation}
This implies~\eqref{eq:le1ii}. Further,~\eqref{eq:le1i} is true
for $\varphi(\cdot) \equiv \const$. Therefore, it is sufficient
to prove it in the case $\varphi(0) = 0$.
Estimate~\eqref{eq:int01.ebt} proves this, taking into account
that $\delta$ can be chosen arbitrary small.
\end{proof}
Lemma~\ref{lem:int.e.q} allows us to refine the asymptotic
formulas~\eqref{eq:yjk=(1+o(1)).e(i.bk.l)} from
Proposition~\ref{prop:BirkSys} when $Q$ is continuous at the
endpoints of the segment $[0,1]$.
%
%
\begin{proposition}\label{prop:asymp}
Assume conditions~\eqref{eq:B.beta}--\eqref{eq:Qjj=0} and let $p
\in \{1,\ldots,\nu\}$. Assume, in addition, that $Q$ is
continuous at the endpoints $0$, $1$. Then for a sufficiently
large $R$ and small $\eps > 0$ equation~\eqref{eq:system} has a
fundamental matrix solution~\eqref{eq:Yxl} analytic with respect
to $\l \in S_{p,\eps,R}$. Moreover, $y_{jk}(x,\l)$, $j,k \in
\{1,\ldots,n\}$, satisfies~\eqref{eq:yjk=(1+o(1)).e(i.bk.l)} and
has the following asymptotic behavior at the endpoints 0 and 1
as $\l \to \infty$, $\l \in S_{p,\eps,R}$,
\begin{align}
    \label{eq:yjk(0,l)}
        y_{jk}(0,\l) &= \begin{cases}
            0, & \text{if}\ \ \ \Re(i b_j \l) < \Re(i b_k \l), \\
            \delta_{jk}, & \text{if}\ \ \ b_j = b_k, \\
            \frac{b_j q_{jk}(0) + o(1)}{b_j-b_k}\cdot\frac{1}{\l},
            & \text{if}\ \ \ \Re(i b_j \l) > \Re(i b_k \l);
        \end{cases} \\
    \label{eq:yjk(1,l)}
        y_{jk}(1,\l) &= \begin{cases}
            \frac{b_j q_{jk}(1) + o(1)}{b_j-b_k}\cdot\frac{e^{i b_k \l}}{\l}, & \text{if}\ \ \ \Re(i b_j \l) < \Re(i b_k \l), \\
            (\delta_{jk} + o(1)) e^{i b_k \l}, & \text{if}\ \ \ b_j = b_k, \\
            0, & \text{if}\ \ \ \Re(i b_j \l) > \Re(i b_k \l).
        \end{cases}
\end{align}
\end{proposition}
%
%
\begin{proof}[Proof]
According to the proof of~\cite[Proposition 2.2]{MalOri12} the
matrix solution $Y(x,\l)$ of system~\eqref{eq:system} with the
asymptotic behavior~\eqref{eq:yjk=(1+o(1)).e(i.bk.l)} in
$S_{p,\eps,R}$ was constructed as the unique solution of the
following system of integral equations
\begin{equation}\label{eq:system-int}
    y_{jk}(x,\l) = \delta_{jk}e^{i b_k \l x}
    -i b_j \int_{a_{jk}}^x e^{-i b_j \l (t-x)} \sum_{l=1}^n q_{jl}(t)y_{lk}(t,\l)dt,
\end{equation}
where
\begin{equation} \label{eq:ajk.def}
    a_{jk} := \begin{cases}
        0, \ \ \text{if}\ \ \ \Re(i b_j \l) \leqslant \Re(i b_k \l), \quad \l \in S_{p,\eps}, \\
        1, \ \ \text{if}\ \ \ \Re(i b_j \l) > \Re(i b_k \l), \quad \l \in S_{p,\eps}.
    \end{cases}
\end{equation}
In particular, $a_{jk}=0$ if $b_j=b_k$. Let us show that this
solution satisfies~\eqref{eq:yjk(0,l)},~\eqref{eq:yjk(1,l)}. It
is clear from~\eqref{eq:system-int} that for $\l \in
S_{p,\eps,R}$ we have
\begin{eqnarray}
    y_{jk}(0,\l) &=& 0, \quad \Re(i b_j \l) < \Re(i b_k \l), \\
    y_{jk}(0,\l) &=& \delta_{jk}, \quad b_j = b_k, \\
    y_{jk}(1,\l) &=& 0, \quad \Re(i b_j \l) > \Re(i b_k \l),
\end{eqnarray}
while the second relation in~\eqref{eq:yjk(1,l)} follows from
Proposition~\ref{prop:asymp}. Thus, we need to prove only the
third relation in~\eqref{eq:yjk(0,l)} and the first one
in~\eqref{eq:yjk(1,l)}.

At first we rewrite~\eqref{eq:yjk=(1+o(1)).e(i.bk.l)} in the
following form
\begin{equation} \label{eq:yjk=(1+ro.jk).e(i.bk.l)}
    y_{jk}(x,\l) = (\delta_{jk}+\rho_{jk}(x,\l))e^{i b_k \l},
    \quad j,k \in \{1,\ldots,n\},
\end{equation}
where $\rho_{jk}(x,\l) = o(1)$, as $\l \to \infty$, $\l \in
S_{p,\eps,R}$, uniformly in $x \in [0,1]$. Now inserting
expression~\eqref{eq:yjk=(1+ro.jk).e(i.bk.l)} for $y_{jk}(x,\l)$
into~\eqref{eq:system-int} we obtain
\begin{equation}\label{eq:subst}
    y_{jk}(x,\l) = \left(\delta_{jk}-i b_j \int_{a_{jk}}^x e^{i (b_k-b_j) \l (t-x)}\left(q_{jk}(t)
    +\sum_{l=1}^n q_{jl}(t)\rho_{lk}(t,\l)\right)dt\right)e^{i b_k \l x}.
\end{equation}
Let $\Re(i b_j \l) > \Re(i b_k \l)$. Setting $x=0$
in~\eqref{eq:subst} one gets
\begin{equation} \label{eq:yjk(0,l).subst}
    y_{jk}(0,\l) = i b_j \int_{0}^1 e^{i (b_k-b_j) \l t}
    q_{jk}(t) dt + i b_j \int_0^1 e^{i (b_k-b_j) \l t} \sum_{l=1}^n
    q_{jl}(t)\rho_{lk}(t,\l)dt.
\end{equation}
Clearly,
\begin{equation}
    \Re(i(b_k-b_j)\l) < -C |\l|, \quad \l \in S_{p,\eps,R},
\end{equation}
for some $C>0$. Hence, applying Lemma~\ref{lem:int.e.q}(i) with
\begin{equation}
    S=S_{p,\eps,R},\quad b=i(b_k-b_j), \quad \varphi(\cdot)=i b_j q_{jk}(\cdot),
\end{equation}
and taking into account the continuity of $q_{jk}(\cdot)$ at
zero, we derive from~\eqref{eq:le1i}
\begin{equation}\label{eq:subst1}
    i b_j \int_0^1 e^{i (b_k-b_j) \l t} q_{jk}(t)dt
    =\frac{b_j q_{jk}(0) + o(1)}{(b_j-b_k) \l},
    \quad\text{as}\quad \l \to \infty,\ \l \in S_{p,\eps,R}.
\end{equation}
Further, since $q_{jl}(\cdot)$, $l \in \{1,\ldots,n\}$, is
bounded at a neighborhood of zero and
\begin{equation}
    \sup\limits_{t\in[0,1]}|\rho_{lk}(t,\l)| = o(1)
    \quad\text{as}\quad \l \to \infty,\ \l \in S_{p,\eps,R},
\end{equation}
Lemma~\ref{lem:int.e.q}(ii) implies
\begin{equation} \label{eq:subst2}
    \int_0^1 e^{i (b_k-b_j) \l t} \sum_{l=1}^n q_{jl}(t)\rho_{lk}(t,\l)dt
    = \sum_{l=1}^n o\left(\int_0^1 \left|e^{i (b_k-b_j) \l t} q_{jl}(t) \right|dt\right) = o(\l^{-1}),
\end{equation}
as $\l \to \infty$, $\l \in S_{p,\eps,R}$. This together
with~\eqref{eq:yjk(0,l).subst} and~\eqref{eq:subst1} yields the
first relation in~\eqref{eq:yjk(0,l)}. Next, let $\Re(i b_j \l)
< \Re(i b_k \l)$. Then using~\eqref{eq:ajk.def} we obtain
from~\eqref{eq:subst}
\begin{equation}
    y_{jk}(1,\l) = -i b_j e^{i b_k \l} \int_{0}^1 e^{i (b_j-b_k) \l s}\left(q_{jk}(1-s)
    +\sum_{l=1}^n q_{jl}(1-s)\rho_{lk}(1-s,\l)\right)ds.
\end{equation}
Using the inequality
\begin{equation}
    \Re(i(b_j-b_k)\l) < -C |\l|, \quad \l \in S_{p,\eps,R},
\end{equation}
and continuity of $q_{jl}(\cdot)$, $l \in \{1,\ldots,n\}$, at
the point 1, and follow the above reasoning we arrive at the
third relation in~\eqref{eq:yjk(1,l)}.
\end{proof}
%
%
\begin{remark} \label{rem:qjk}
Fix $j,k\in \{1,\ldots,n\}$. As it is clear from the proof of
Proposition~\ref{prop:asymp}, the individual function
$y_{j,k}(x,\l)$ satisfies the third relation
in~\eqref{eq:yjk(0,l)} whenever $q_{jk}(\cdot)$ is continuous at
zero and $q_{jl}(\cdot)$ is bounded at zero for $l \in
\{1,\ldots,n\}$. Otherwise it satisfies only the weaker relation
\begin{equation}
    y_{jk}(0,\l) = o(1) \quad\text{as}\quad
    \l \to \infty, \quad \l \in S_{p,\eps}.
\end{equation}
Moreover, if $q_{jl}(\cdot)$, $l \in \{1,\ldots,n\}$, is just
bounded at zero then, by Lemma~\ref{lem:int.e.q}(ii),
$y_{jk}(0,\l) = O(\l^{-1})$, $\l \in S_{p,\eps,R}$. Similar
statements are true for $y_{jk}(1,\l)$. This allows us to weaken
assumptions on $Q(\cdot)$ in further considerations.
\end{remark}
%
%
In the next step we investigate the asymptotic behavior of the
characteristic determinant $\Delta(\cdot)$. For convenience in
applications we do not assume that equal $b_j$ are grouped into
blocks as it was in the previous paper~\cite{MalOri12}.
%
%
\begin{proposition}\label{prop:DeltaY=(c0+c1/l)}
Let $B$ be defined by~\eqref{eq:B.def}, $Q(\cdot) \in L^1([0,1];
\bC^{n \times n})$ and let $q_{jk}$ be continuous at points 0
and 1 if $b_j \ne b_k$. Let, as above, $\Delta(\cdot)$ be the
characteristic determinant~\eqref{eq:Delta(l).def} of the
problem~\eqref{eq:system}--\eqref{eq:CY(0)+DY(1)=0}. Finally,
let $p \in \{1,\ldots,\nu\}$. Then for sufficiently small $\eps
> 0$ the characteristic determinant $\Delta(\cdot)$ admits the
following asymptotic expansion
\begin{equation}\label{eq:DeltaY=(c0(zp)+c1(zp)/l)}
    \Delta(\l)=\gamma_p \cdot \left(\omega_0(z_p) \cdot(1+o(1))
    + \frac{\omega_1(z_p) + o(1)}{\l} \right) e^{i \tau_p \l},
    \quad\text{as}\quad \l \to \infty, \ \l \in S_{p,\eps}.
\end{equation}
Here $z_p$ is a fixed point in $S_{p,\eps}$,
\begin{eqnarray}
    \label{eq:delta.zp}
        \gamma_p &:=& \exp\left(\sum_{\Re(i b_j z_p) > 0} i b_j \int_0^1 q_{jj}(t) dt\right),\\
    \label{eq:tau.def}
        \tau_p &:=& \sum_{\Re(i b_j z_p) > 0} b_j, \\
    \label{eq:c0zp.def}
        \omega_0(z_p) &:=& \det T_{i z_p B}(C,D), \\
    \label{eq:c1zp.def}
        \omega_1(z_p) &:=& \sum_{\Re(i b_j z_p) < 0 \atop \Re(i b_k z_p) > 0}
        \frac{\det T_{i z_p B}^{c_j \to c_k} b_k q_{kj}(0)
        - \det T_{i z_p B}^{d_k \to d_j} b_j q_{jk}(1)}{b_k-b_j},
\end{eqnarray}
and the matrix $T_{i z_p B}^{c_j \to c_k}$ $(T_{i z_p B}^{d_j
\to d_k})$ is obtained from $T_{i z_p B}(C,D)$ by replacing its
$j$th column by the $k$th column of the matrix $C$ $($resp.
$D)$.
\end{proposition}
%
%
\begin{remark}
Denote by $c_j$ $(d_j)$ the $j$th column of the matrix $C$
$($resp. $D)$. Note that if $\Re(i b_j \l) < 0$, the $j$th
column of $T_{i z_p B}(C,D)$ coincides with $c_j$. Therefore,
the superscript $c_j\to c_k$ in the notation of the matrix $T_{i
z_p B}^{c_j \to c_k}$ means just replacement $c_j$ by $c_k$ in
$T_{i z_p B}$. The notation $T_{i z_p B}^{d_k \to d_j}$ is
justified  similarly.
\end{remark}
%
%
\begin{proof}[Proof of Proposition~\ref{prop:DeltaY=(c0+c1/l)}]
As in the proof of Theorem~\ref{th:compl.gen} we can assume that
conditions~\eqref{eq:B.beta}--\eqref{eq:Q=Qjk} are fulfilled.
Further, applying the gauge transform $W : y \to W(x)y$ with
$W(\cdot)$ given by~\eqref{eq:iBW'=Q1.W}--\eqref{eq:Wx} the
operator $L_{C,D}(Q)$ is transformed into the operator
$L_{C,\wt{D}}(\wt{Q})$ with $\wt{D}$ and $\wt{Q}(x)$ given
by~\eqref{eq:wtD.wtQ}. Due to~\eqref{eq:wtDelta=Delta} the
characteristic determinant is preserved under this transform.

Further, $Q - Q_1$ is continuous at the endpoints 0 and 1. Since
both $W(\cdot)$ and $W^{-1}(\cdot)$ are continuous on $[0,1]$,
$\wt{Q}$ is continuous at the endpoints 0 and 1 too. According
to~\eqref{eq:wtD.wtQ} $\wt{Q}$ satisfies~\eqref{eq:Qjj=0} and,
by Proposition~\ref{prop:asymp}, there exists a fundamental
matrix solution $\wt{Y}(\cdot,\l)$ of system~\eqref{eq:system}
with $\wt{Q}$ in place of $Q$, that satisfies asymptotic
relations~\eqref{eq:yjk(0,l)} and~\eqref{eq:yjk(1,l)} with
$\wt{q}_{jk}(\cdot)$ in place of $q_{jk}(\cdot)$. The
fundamental matrices $\wt{Y}(\cdot,\l)$ and $\wt\Phi(\cdot,\l)$
are related by
\begin{equation}
    \wt{Y}(x,\l) = \wt{\Phi}(x,\l) P(\l),
    \qquad x \in [0,1], \quad \l \in S_{p, \eps, R},
\end{equation}
where $P(\lambda)=:(p_{kj}(\lambda))_{k,j=1}^n$ is an analytical
invertible matrix function in $S_{p,\varepsilon,R}$. Hence
$\wt{Y}(0,\lambda) = P(\lambda)$ and due
to~\eqref{eq:yjk=(1+o(1)).e(i.bk.l)}
and~\eqref{eq:wtDelta=Delta} (cf. \cite[formula
(3.31)]{MalOri12}),
\begin{equation} \label{eq:DeltawtY.def}
    \Delta_{\wt{Y}}(\l):=\det(C \wt{Y}(0,\l) + \wt{D} \wt{Y}(1,\l))  = \wt{\Delta}(\l) \det (\wt{Y}(0,\l))
    = (1+o(1)) \Delta(\l),
\end{equation}
as $\l \to \infty$, $\l \in S_{p,\eps}$. Thus, it suffices to
prove~\eqref{eq:DeltaY=(c0(zp)+c1(zp)/l)} with
$\Delta_{\wt{Y}}(\cdot)$ instead of $\Delta(\cdot)$. Since $W(0)
= I_n,$ one has  $\wt{Q}(0) = Q(0) - Q_1(0)$ and hence
\begin{equation} \label{eq:wtY0}
    \wt{Y}(0,\l) = Y_0 := Y_0(\l) := \left(y_{jk}^{[0]}(\l)\right)_{j,k=1}^n,
\end{equation}
where $y_{jk}^{[0]}(\l)$ is given  by~\eqref{eq:yjk(0,l)}. Let
us simplify $\wt{Y}(1,\l)$. To this end let
\begin{align}
    \wt{Q}(x) &= \left(\wt{Q}_{jk}(x)\right)_{j,k=1}^r,
    \quad \wt{Q}_{jk}(x) \in \bC^{n_j \times n_k}, \\
    \wt{Y}(x,\l) &= \left(\wt{Y}_{jk}(x,\l)\right)_{j,k=1}^r,
    \quad \wt{Y}_{jk}(x,\l) \in \bC^{n_j \times n_k},
\end{align}
be the block-representations of matrices $\wt{Q}(x)$ and
$\wt{Y}(x,\l)$ with respect to the orthogonal decomposition
$\bC^n = \bC^{n_1} \oplus \ldots \oplus \bC^{n_r}$. It follows
from~\eqref{eq:Wx}--\eqref{eq:wtD.wtQ} that
\begin{equation} \label{eq:wtQjk}
    \wt{Q}_{jk}(1) = W_{jj}^{-1}(1) Q_{jk}(1) W_{kk}(1), \qquad
    j\not = k.
\end{equation}
Further, note that due to~\eqref{eq:B.beta}--\eqref{eq:Q=Qjk}
formula~\eqref{eq:yjk(1,l)} for $\wt{Y}(1,\l)$ takes the form
\begin{equation} \label{eq:wtYjk1l}
    \wt{Y}_{jk}(1,\l) = \begin{cases}
        \frac{\beta_j \wt{Q}_{jk}(1) + o(1)}{\beta_j-\beta_k}\cdot\frac{e^{i \beta_k \l}}{\l},
        & \text{if}\ \ \ \Re(i \beta_j \l) < \Re(i \beta_k \l), \\
        (I_{n_k} + o(1)) e^{i \beta_k \l}, & \text{if}\ \ \ j = k, \\
        0, & \text{if}\ \ \ \Re(i \beta_j \l) > \Re(i \beta_k \l).
    \end{cases}
\end{equation}
In view of~\eqref{eq:B.beta}--\eqref{eq:Q=Qjk} and
\eqref{eq:wtQjk}--\eqref{eq:wtYjk1l} we have
\begin{equation} \label{eQ:wtY1}
    \wt{Y}(1,\l) = W^{-1}(1) Y_1 W(1), \quad Y_1 := Y_1(\l) = \left(y_{jk}^{[1]}(\l)\right)_{j,k=1}^n,
\end{equation}
where $y_{jk}^{[1]}(\l)$ is given  by~\eqref{eq:yjk(1,l)}.
Combining
\eqref{eq:wtD.wtQ},~\eqref{eq:DeltawtY.def},~\eqref{eq:wtY0}
and~\eqref{eQ:wtY1} yields
\begin{equation} \label{eq:Delta.wtY=detJV}
    \Delta_{\wt{Y}}(\l) = \det\left(C Y_0(\l) + D Y_1(\l) W(1)\right) = \det(J \cdot V),
\end{equation}
where
\begin{equation}
    V := V(\l) := \begin{pmatrix} Y_0 \\ V_1 \\ \end{pmatrix}, \quad
    V_1 := V_1(\l) := Y_1(\l) W(1), \quad\text{and}\quad J := \begin{pmatrix} C & D \end{pmatrix}.
\end{equation}
By the Cauchy-Binet formula
\begin{equation}\label{eq:binet}
    \Delta_{\wt{Y}}(\l)=\sum_{1 \leqslant k_1 < \ldots < k_n \leqslant 2n}
        J\begin{pmatrix}
            1 &   2 & \ldots &   n \\
          k_1 & k_2 & \ldots & k_n \\
        \end{pmatrix}
    \cdot
        V\begin{pmatrix}
          k_1 & k_2 & \ldots & k_n \\
            1 &   2 & \ldots &   n \\
        \end{pmatrix}.
\end{equation}
Here $A\begin{pmatrix}
  j_1 & j_2 & \ldots & j_p \\
  k_1 & k_2 & \ldots & k_p
\end{pmatrix}$ denotes the minor of $n \times n'$ matrix $A = (a_{jk})$
composed of its entries located in the rows with indices
$j_1,\ldots,j_p \in \{1,\ldots,n\}$ and columns with indices
$k_1,\ldots,k_p \in \{1,\ldots,n'\}$.

Fix a set $\{k_1,k_2,\ldots,k_n\}$ such that $1 \leqslant k_1 <
\ldots < k_n \leqslant 2n$ and denote by $m$ the number of
entries of the set that do not exceed $n$, i.e.,
\begin{equation}
    1 \leqslant k_1 < \ldots < k_{m} \leqslant n < k_{{m}+1} < \ldots < k_n.
\end{equation}
Applying Laplace theorem to expand the second factor
in~\eqref{eq:binet} with respect to the first $m$ rows, one gets
\begin{multline}\label{eq:V()=sum}
    V\begin{pmatrix}
      k_1 & k_2 & \ldots & k_n \\
        1 &   2 & \ldots &   n \\
    \end{pmatrix}
    = \sum_{1 \leqslant j_1 < \ldots j_{m} \leqslant n \atop
        {1 \leqslant j_{{m}+1} < \ldots < j_n < n \atop
        \{j_1,\ldots,j_n\}=\{1,\ldots,n\}}}
    (-1)^{(1+\ldots+{m})+(j_1+\ldots+j_{m})} \\
    \times Y_0\begin{pmatrix}
              k_1 & \ldots & k_{m} \\
              j_1 & \ldots & j_{m} \\
        \end{pmatrix}
    \cdot V_1\begin{pmatrix}
              k_{{m}+1}-n & \ldots & k_n-n \\
              j_{{m}+1}   & \ldots & j_n \\
    \end{pmatrix}.
\end{multline}
It follows from~\eqref{eq:yjk(0,l)} and~\eqref{eq:yjk(1,l)} that
\begin{equation}\label{eq:yjk(0)=O(1),...}
    y_{jk}^{[0]}(\l)=O(1), \quad y_{jk}^{[1]}(\l) = O(1)\cdot e^{i b_k \l},
    \qquad \l \in S_{p,\eps,R},  \quad j,k \in \{1, \ldots, n\}.
\end{equation}
Setting
\begin{equation}
    (v_{jk}(\l))_{j,k=1}^n := V_1(\l)= Y_1(\l)W(1)
\end{equation}
we obtain from~\eqref{eq:yjk(0)=O(1),...} and the block-diagonal
structure of the matrices $B$ and $W(1)$ that
\begin{equation}\label{eq:vjk(1)=O(1).e}
    v_{jk}(\l)=O(1)\cdot e^{i b_k \l},
    \quad \l \in S_{p,\eps,R}, \quad j,k \in \{1, \ldots, n\}.
\end{equation}
It follows
from~\eqref{eq:wtY0},~\eqref{eQ:wtY1},~\eqref{eq:yjk(0)=O(1),...},~\eqref{eq:vjk(1)=O(1).e}
that for $\l \in S_{p,\eps,R}$
\begin{eqnarray}
    \label{eq:Y0()=O(1)}
        Y_0\begin{pmatrix}
          k_1 & \ldots & k_{m} \\
          j_1 & \ldots & j_{m} \\
        \end{pmatrix} & = & O(1), \\
    \label{eq:Y1()=O(1).e}
        V_1\begin{pmatrix}
          k_{{m}+1}-n & \ldots & k_n-n \\
          j_{{m}+1} & \ldots & j_n \\
        \end{pmatrix}
        & = & O(1)\cdot e^{i(b_{j_{{m}+1}}+\ldots+b_{j_n})\l}.
\end{eqnarray}
Let $\k$ be a number of negative values among $\Re(i b_1 \l),
\ldots, \Re(i b_n \l),\ \l \in S_{p,\eps}$. For definiteness we
assume that
\begin{equation} \label{eq:Re(i.bj.l)}
\begin{split}
    \Re(i b_j \l) < 0, & \qquad j \in \{1,\ldots,{\k}\}, \\
    \Re(i b_j \l) > 0, & \qquad j \in \{{\k+1},\ldots,n\}.
\end{split}
\end{equation}
It is clear from~\eqref{eq:Re(i.bj.l)} that for
$\{j_{{m}+1},\ldots,j_n\} \ne \{{\k+1},\ldots,n\}$ the following
inequality holds
\begin{equation}
    \Re(i b_{j_{{m}+1}}\l)+\ldots+\Re(i b_{j_{n}}\l) <
    \Re(i b_{{\k+1}}\l)+\ldots+\Re(i b_{n}\l)
    = \Re(i \tau_p \l),\quad \l \in S_{p,\eps},
\end{equation}
where $\tau_p$ is given by~\eqref{eq:tau.def}. Combining this
estimate with~\eqref{eq:Y0()=O(1)} and~\eqref{eq:Y1()=O(1).e}
yields that for $\{j_{{m}+1},\ldots,j_n\} \ne
\{{\k+1},\ldots,n\}$ and each $h \in \bN$,
\begin{equation}\label{eq:Y0.Y1=O(1/l^m).e}
    Y_0\begin{pmatrix}
      k_1 & \ldots & k_{m} \\
      j_1 & \ldots & j_{m} \\
    \end{pmatrix}
    \cdot V_1\begin{pmatrix}
      k_{{m}+1}-n & \ldots & k_n-n \\
      j_{{m}+1}   & \ldots & j_n \\
    \end{pmatrix}
    = O\left(\frac{1}{\l^{h}}\right) \cdot e^{i \tau_p \l},
    \quad \l \in S_{p,\eps,R}.
\end{equation}
Inserting~\eqref{eq:Y0.Y1=O(1/l^m).e} into~\eqref{eq:V()=sum} we
obtain for $\l \in S_{p,\eps,R}$ and each $h \in \bN$ that
\begin{equation} \label{eq:V()=O(1/l^m).e}
    V\begin{pmatrix}
      k_1 & \ldots & k_n \\
        1 & \ldots &   n \\
    \end{pmatrix}
    = O\left(\frac{1}{\l^{h}}\right) \cdot e^{i \tau_p \l},
    \quad m \ne \k;
\end{equation}
\begin{equation} \label{eq:V()=Y0().Y1()+}
    V\begin{pmatrix}
      k_1 & \ldots & k_n \\
        1 & \ldots &   n \\
    \end{pmatrix}
    = Y_0\begin{pmatrix}
      k_1 & \ldots & k_{\k} \\
        1 & \ldots &   {\k} \\
    \end{pmatrix}
    \cdot V_1\begin{pmatrix}
      k_{\k+1}-n & \ldots & k_n-n \\
        {\k+1}   & \ldots &     n \\
    \end{pmatrix}
    + O\left(\frac{e^{i \tau_p \l}}{\l^{h}}\right),
    \quad {m}=\k,
\end{equation}
Due to the block-diagonal structure of $W(1)$ one has
\begin{equation} \label{eq:V1=Y1.delta}
    V_1\begin{pmatrix}
          k_{\k+1} & \ldots & k_n \\
          {\k+1} & \ldots & n \\
        \end{pmatrix} = Y_1\begin{pmatrix}
          k_{\k+1} & \ldots & k_n \\
          {\k+1} & \ldots & n \\
        \end{pmatrix} \gamma(\l), \quad
    \gamma(\l) := \prod_{j=1 \atop{\Re(i \beta_j \l) > 0}}^{r} \det W_{jj}(1).
\end{equation}
Applying the Liouville theorem to system~\eqref{eq:iBW'=Q1.W}
and using the definition of the sector $S_{p,\eps}$ yields
\begin{equation}
    \gamma(\l) = \gamma_p, \quad \l \in S_{p,\eps},
\end{equation}
where $\gamma_p$ is given by~\eqref{eq:delta.zp}. Now it follows
from~\eqref{eq:binet},~\eqref{eq:V()=O(1/l^m).e},~\eqref{eq:V()=Y0().Y1()+}
and~\eqref{eq:V1=Y1.delta} that for $\l \in S_{p,\eps,R}$
\begin{multline}\label{eq:DeltaY=sum.J()Y0()Y1()}
    \Delta_{\wt{Y}}(\l)
    = \gamma_p \sum_{1 \leqslant k_1      < \ldots < k_{\k} \leqslant n \atop
            1 \leqslant k_{\k+1} < \ldots < k_n    \leqslant n}
    J \begin{pmatrix}
        1 & \ldots &    \k  &      \k+1  & \ldots &     n\\
      k_1 & \ldots & k_{\k} & n+k_{\k+1} & \ldots & n+k_n\\
    \end{pmatrix}\ \\
    \times \ Y_0\begin{pmatrix}
      k_1 & \ldots & k_{\k} \\
      1 & \ldots & {\k} \\
    \end{pmatrix}
    \cdot Y_1\begin{pmatrix}
      k_{\k+1} & \ldots & k_n \\
      {\k+1} & \ldots & n \\
    \end{pmatrix}
    + O\left(\frac{1}{\l^{h}}\right) \cdot e^{i \tau_p \l},
    \quad h \in \bN.
\end{multline}
Let $(k_1,\ldots,k_{\k}) \in \bN^{\k}$ be a sequence satisfying
$1 \leqslant k_1 < \ldots < k_{\k} \leqslant n$ and let $(l_1,
\ldots, l_{\k})$ be its permutation. It is easily seen that
\begin{multline} \label{eq:JY0.perm}
    J \begin{pmatrix}
        1 & \ldots &    \k  &      \k+1  & \ldots &     n\\
      k_1 & \ldots & k_{\k} & n+k_{\k+1} & \ldots & n+k_n\\
    \end{pmatrix} \cdot Y_0\begin{pmatrix}
      k_1 & \ldots & k_{\k} \\
      1 & \ldots & {\k} \\
    \end{pmatrix} \\
    = J \begin{pmatrix}
        1 & \ldots &    \k  &      \k+1  & \ldots &     n\\
      l_1 & \ldots & l_{\k} & n+k_{\k+1} & \ldots & n+k_n\\
    \end{pmatrix}  \cdot Y_0\begin{pmatrix}
      l_1 & \ldots & l_{\k} \\
      1 & \ldots & {\k} \\
    \end{pmatrix}.
\end{multline}
This identity means that for each summand in the right-hand side
of~\eqref{eq:DeltaY=sum.J()Y0()Y1()} we can choose arbitrary
permutation of the corresponding sequence $(k_1,\ldots,k_{\k})$.
Clearly, the same is true for the corresponding sequence
$(k_{\k+1},\ldots,k_n)$.

It follows from~\eqref{eq:yjk(0,l)} that
\begin{equation} \label{eq:Y0=In+o(1)}
    Y_0 = Y(0,\l) = \begin{pmatrix}
        I_{\k} + o(1) & O(\l^{-1}) \\
        O(\l^{-1}) & I_{n-\k} + o(1) \\
    \end{pmatrix}, \quad\text{as}\quad \l \to \infty,\ \l \in S_{p,\eps,R}.
\end{equation}
Hence if the intersection of the sets $\{k_1,\ldots,k_{\k}\}$
and $\{{\k+1},\ldots, n\}$ consists of $s$ elements, then the
corresponding minor $Y_0 \begin{pmatrix}
k_1 & \ldots & k_{\k} \\
1 & \ldots & {\k} \\  \end{pmatrix}$ contains exactly $s$ lines
with entries of the form $O(\l^{-1})$ while all entries of other
lines are of the form $O(1)$. Indeed, if $k_j > \k$, then $j$th
line of the considered minor coincides with the $(k_j-\k)$th
line of the lower-left block of the
block-matrix~\eqref{eq:Y0=In+o(1)}. Thus, we have
\begin{equation} \label{eq:Y0()=O(1/l^m)}
    Y_0 \begin{pmatrix}
        k_1 & \ldots & k_{\k} \\
        1 & \ldots & {\k} \\
    \end{pmatrix}
    = O\left(\frac{1}{\l^{s}}\right), \quad \l \in S_{p,\eps,R}.
\end{equation}
For the cases $s=0$ and $s=1$ we can obtain sharper estimates.
At first,~\eqref{eq:Y0=In+o(1)} directly implies
\begin{equation}\label{eq:Y0().m=0}
    Y_0\begin{pmatrix}
      1 & \ldots & {\k} \\
      1 & \ldots & {\k} \\
    \end{pmatrix}
    = 1+o(1), \quad\text{as}\quad \l \to \infty,\ \l \in S_{p,\eps,R}.
\end{equation}
Next, assume that $s=1$, i.e. the set $\{k_1,\ldots,k_{\k}\}$ is
obtained from $\{1,\ldots,{\k}\}$ by replacing its one entry by
an entry from $\{{\k+1},\ldots,n\}$. Assume that $j$ is replaced
by $k$, where $1 \leqslant j \leqslant \k < k \leqslant n$.
Then, according to~\eqref{eq:Re(i.bj.l)}, $\Re(i b_k \l) > 0 >
\Re(i b_j \l)$ and, by~\eqref{eq:yjk(0,l)},
\begin{multline}\label{eq:Y0().m=1}
    Y_0 \begin{pmatrix}
      1 & \ldots & {j-1} & k & {j+1} & \ldots & {\k} \\
      1 & \ldots & {j-1} & j & {j+1} & \ldots & {\k} \\
    \end{pmatrix} \\
    =\det \begin{pmatrix}
      1 + o(1)   & \cdots & o(1)       & o(1)                       & o(1)       & \cdots & o(1)       \\
      \vdots     & \ddots & \vdots     & \vdots                     & \vdots     & \ddots & \vdots     \\
      o(1)       & \cdots & 1 + o(1)   & o(1)                       & o(1)       & \cdots & o(1)       \\
      O(\l^{-1}) & \cdots & O(\l^{-1}) & \frac{r_{k j}(0)+o(1)}{\l} & O(\l^{-1}) & \cdots & O(\l^{-1}) \\
      o(1)       & \cdots & o(1)       & o(1)                       & 1 + o(1)   & \cdots & o(1)       \\
      \vdots     & \ddots & \vdots     & \vdots                     & \vdots     & \ddots & \vdots     \\
      o(1)       & \cdots & o(1)       & o(1)                       & o(1)       & \cdots & 1 + o(1)   \\
    \end{pmatrix} \\
    = \frac{r_{k j}(0)+o(1)}{\l},
    \quad\text{as}\quad \l \to \infty,\ \l \in S_{p,\eps,R},
\end{multline}
where we set for brevity $r_{jk}(x) := \frac{b_j q_{jk}(x)}{b_j
- b_k}$.

Further, according to~\eqref{eq:yjk(1,l)}
\begin{equation} \label{eq:Y1=In+o(1)}
    Y_1 = Y(1,\l) = \begin{pmatrix}
        I_{\k} + o(1) & O(\l^{-1}) \\
        O(\l^{-1}) & I_{n-\k} + o(1) \\
    \end{pmatrix} \cdot E(\l), \quad  E(\l) := \diag(e^{i b_1 \l}, \ldots, e^{i b_n \l}),
\end{equation}
as $\l \to \infty$, $\l \in S_{p,\eps,R}$.

Let the set $\{k_{\k+1},\ldots,k_n\}$ contain exactly $s$
entries from the set $\{1, \ldots, {\k}\}$. Then repeating the
above reasoning to $Y_1$ in place of $Y_0$ yields
\begin{equation}\label{eq:Y1()=O(1/l^m)E}
    Y_1\begin{pmatrix}
      k_{\k+1} & \ldots & k_n \\
      {\k+1} & \ldots & n \\
    \end{pmatrix}
    = O\left(\frac{1}{\l^s}\right)e^{i \tau_p \l},
    \quad \l \in S_{p,\eps,R}.
\end{equation}
Further, it is easily seen that
\begin{gather}
    \label{eq:Y1().m=0}
    Y_1\begin{pmatrix}
      {\k+1} & \ldots & n \\
      {\k+1} & \ldots & n \\
    \end{pmatrix}
    = (1+o(1))\cdot e^{i \tau_p \l}; \\
    \label{eq:Y1().m=1}
    Y_1 \begin{pmatrix}
      {\k+1} & \ldots & {k-1} & j & {k+1} & \ldots & n \\
      {\k+1} & \ldots & {k-1} & k & {k+1} & \ldots & n \\
    \end{pmatrix}
    = (r_{j k}(1)+o(1))\cdot \frac{e^{i \tau_p \l}}{\l},
\end{gather}
as $\l \to \infty$, $\l \in S_{p,\eps,R}$, where $j \in
\{1,\ldots,\k\}$ and $k \in \{\k+1,\ldots,n\}$.

Inserting formulas~\eqref{eq:Y0()=O(1/l^m)} and
\eqref{eq:Y1()=O(1/l^m)E} into~\eqref{eq:DeltaY=sum.J()Y0()Y1()}
and using~\eqref{eq:JY0.perm} we get
\begin{align} \label{eq:DeltaY=J()Y0()Y1()+sum+sum+}
    \gamma_p^{-1} \Delta_{\wt{Y}}(\l) &= \Bigl(\Bigr. J \left(\begin{smallmatrix}
      1 & \ldots & \k &   \k+1 & \ldots &   n\\
      1 & \ldots & \k & n+\k+1 & \ldots & n+n\\
    \end{smallmatrix}\right)
    \cdot Y_0 \left(\begin{smallmatrix}
      1 & \ldots & \k \\
      1 & \ldots & \k \\
    \end{smallmatrix}\right)
    \cdot Y_1 \left(\begin{smallmatrix}
      \k+1 & \ldots & n \\
      \k+1 & \ldots & n \\
    \end{smallmatrix}\right) \nonumber \\
    &+ \sum_{j=1}^{\k} \sum_{k=\k+1}^n J \left(\begin{smallmatrix}
      1 & \ldots & j-1 & j & j+1 & \ldots & \k &   \k+1 & \ldots &   n\\
      1 & \ldots & j-1 & k & j+1 & \ldots & \k & n+\k+1 & \ldots & n+n\\
    \end{smallmatrix}\right) \cdot Y_0 \left(\begin{smallmatrix}
      1 & \ldots & j-1 & k & j+1 & \ldots & \k \\
      1 & \ldots & j-1 & j & j+1 & \ldots & \k \\
    \end{smallmatrix}\right) \cdot Y_1 \left(\begin{smallmatrix}
      \k+1 & \ldots & n \\
      \k+1 & \ldots & n \\
    \end{smallmatrix}\right) \nonumber \\
    &+ \sum_{j=1}^{\k} \sum_{k=\k+1}^n J \left(\begin{smallmatrix}
      1 & \ldots & \k &   \k+1 & \ldots &   k-1 &   k &   k+1 & \ldots &   n\\
      1 & \ldots & \k & n+\k+1 & \ldots & n+k-1 & n+j & n+k+1 & \ldots & n+n\\
    \end{smallmatrix}\right) \nonumber \\
    &\times\ Y_0 \left(\begin{smallmatrix}
      1 & \ldots & \k \\
      1 & \ldots & \k \\
    \end{smallmatrix}\right)
    \cdot Y_1 \left(\begin{smallmatrix}
      \k+1 & \ldots & k-1 & j & k+1 & \ldots & n \\
      \k+1 & \ldots & k-1 & k & k+1 & \ldots & n \\
    \end{smallmatrix}\right) \Bigl.\Bigr)
    + O\left(\frac{1}{\l^2}\right) e^{i \tau_p \l},
    \quad \l \in S_{p,\eps,R}.
\end{align}
Let $z_p$ be some fixed point in $S_{p,\eps}$. Then it is clear
from inequalities~\eqref{eq:Re(i.bj.l)} and definition of
matrices $T_{i z_p B}(C,D)$, $T_{i z_p B}^{c_j \to c_k}$ and
$T_{i z_p B}^{d_k \to d_j}$ that
\begin{align}
    \label{eq:J()=T}
        J \left(\begin{smallmatrix}
          1 & \ldots & \k &   \k+1 & \ldots &   n\\
          1 & \ldots & \k & n+\k+1 & \ldots & n+n\\
        \end{smallmatrix}\right) &= \det T_{i z_p B}(C,D), \\
    \label{eq:J()=T(cj->ck)}
        J \left(\begin{smallmatrix}
          1 & \ldots & j-1 & j & j+1 & \ldots & \k &   \k+1 & \ldots &   n\\
          1 & \ldots & j-1 & k & j+1 & \ldots & \k & n+\k+1 & \ldots & n+n\\
        \end{smallmatrix}\right) &= \det T_{i z_p B}^{c_{j} \to c_{k}}, \\
    \label{eq:J()=T(dk->dj)}
        J \left(\begin{smallmatrix}
          1 & \ldots & \k &   \k+1 & \ldots &   k-1 &   k &   k+1 & \ldots &   n\\
          1 & \ldots & \k & n+\k+1 & \ldots & n+k-1 & n+j & n+k+1 & \ldots & n+n\\
        \end{smallmatrix}\right) &= \det T_{i z_p B}^{d_{k} \to d_{j}}.
\end{align}
Now
inserting~\eqref{eq:Y0().m=0},~\eqref{eq:Y0().m=1},~\eqref{eq:Y1().m=0},~\eqref{eq:Y1().m=1}
and~\eqref{eq:J()=T},~\eqref{eq:J()=T(cj->ck)},~\eqref{eq:J()=T(dk->dj)}
into~\eqref{eq:DeltaY=J()Y0()Y1()+sum+sum+} we get
\begin{align}
    \gamma_p^{-1} \Delta_{\wt{Y}}(\l)
    &= \det T_{i z_p B}(C,D) \cdot (1+o(1)) \cdot (1+o(1)) \cdot e^{i \tau_p \l} \nonumber \\
    &+ \sum_{j=1}^{\k} \sum_{k=\k+1}^n \det T_{i z_p B}^{c_{j} \to c_{k}}
    \cdot \frac{r_{k j}(0)+o(1)}{\l} \cdot (1+o(1)) \cdot e^{i \tau_p \l} \nonumber \\
    &+ \sum_{j=1}^{\k} \sum_{k=\k+1}^n \det T_{i z_p B}^{d_{k} \to d_{j}}
    \cdot (1+o(1)) \cdot \frac{r_{j k}(1)+o(1)}{\l} \cdot e^{i \tau_p \l}
    + O\left(\frac{1}{\l^2}\right) e^{i \tau_p \l} \\
    &= e^{i \tau_p \l} \cdot \Biggl(\Biggr.\omega_0(z_p) \cdot (1+o(1)) + o(\l^{-1}) \nonumber \\
    &+ \left. \sum_{j=1}^{\k} \sum_{k=\k+1}^n
    \frac{\det T_{i z_p B}^{c_{j} \to c_{k}} b_{k} q_{k j}(0)
    - \det T_{i z_p B}^{d_{k} \to d_{j}} b_{j} q_{j  k}(1)}{\l(b_{k}-b_{j})} \right),
\end{align}
as $\l \to \infty$, $\l \in S_{p,\eps,R}$. Rewriting the double
sum in the last equality with account of~\eqref{eq:Re(i.bj.l)}
we arrive at formula~\eqref{eq:DeltaY=(c0(zp)+c1(zp)/l)} with
the required form of $\omega_1(z_p)$.
\end{proof}
Next we present an asymptotic formula for the characteristic
determinant $\Delta(\cdot)$ which will be needed in the sequel.
It can be obtained by repeating the proof of
Proposition~\ref{prop:DeltaY=(c0+c1/l)} but using
Proposition~\ref{prop:BirkSys} in place of
Proposition~\ref{prop:asymp} for estimating the solution
$Y(x,\l)$.
%
%
\begin{lemma} \label{lem:Delta}
Assume that $Q(\cdot) \in L^1([0,1]; \bC^{n \times n})$. Let $p
\in \{1,\ldots,\nu\}$. Then for sufficiently small $\eps
> 0$ the characteristic determinant $\Delta(\cdot)$ admits the
following asymptotic behavior
\begin{equation}\label{eq:Delta=(T+o(1)).E}
    \Delta(\l)=\gamma_p \cdot \left(\det T_{i z_p B}(C,D) + o(1)\right) e^{i \tau_p \l},
    \quad\text{as}\quad \l \to \infty, \ \l \in S_{p,\eps}.
\end{equation}
Here $z_p$ is a fixed point in $S_{p,\eps}$, while $\gamma_p$
and $\tau_p$ are given by~\eqref{eq:delta.zp}
and~\eqref{eq:tau.def}, respectively.
\end{lemma}
%
%
Formula~\eqref{eq:Delta=(T+o(1)).E} can also be extracted from
the proof of~\cite[Theorem 1.2]{MalOri12} (cf. formula (3.38)
from~\cite{MalOri12}).
%
%
\section{Explicit completeness results} \label{sec:completeness}
%
%
\subsection{Explicit sufficient conditions of completeness.}
%
%
Now we are ready to state our main result on completeness of the
root vectors of the boundary value
problem~\eqref{eq:system}--\eqref{eq:CY(0)+DY(1)=0} in terms of
the matrices $B,C,D$ and $Q(\cdot)$.
%
%
\begin{theorem} \label{th:explicit.nxn}
Assume that $Q(\cdot) \in L^1([0,1]; \bC^{n\times n})$ and
$q_{jk}$ is continuous at points 0 and 1 if $b_j \ne b_k$. Let
$\omega_0(z_k)$ and $\omega_1(z_k)$ be given
by~\eqref{eq:c0zp.def} and~\eqref{eq:c1zp.def}, respectively.
Assume also that there exist three admissible complex numbers
$z_1, z_2, z_3$ satisfying the following conditions:

(a) the origin is an interior point of the triangle $\triangle
_{z_1z_2z_3};$

(b) $|\omega_0(z_k)|+|\omega_1(z_k)| \ne 0, \quad k\in
\{1,2,3\}$.

Then the system of root functions of the
BVP~\eqref{eq:system}--\eqref{eq:CY(0)+DY(1)=0} is complete and
minimal in $L^2([0,1]; \bC^n)$.
\end{theorem}
%
%
\begin{remark} \label{rem:omega01}
Note that $\omega_{j}(\cdot)$, $j \in \{0,1\}$, is a constant
function in each sector $\sigma_k$, $k \in \{1,\ldots,m\}$,
introduced before formula~\eqref{eq:detTizBCD}. Hence
$\omega_{j}(\cdot)$, $j \in \{0,1\}$, is piecewise constant
function in the plane $\bC$ with cuts along the lines $\partial
\sigma_k$, $k \in \{1,\ldots,m\}$. It is easily seen that the
assumptions of Theorem~\ref{th:explicit.nxn} fail if and only if
both $\omega_0(\cdot)$ and $\omega_1(\cdot)$ vanish in the open
half-plane $\{\l \in \bC : \Re(c \l) > 0\}$ for some $c \ne 0$.
\end{remark}
%
%
\begin{proof}[Proof of Theorem~\ref{th:explicit.nxn}]
Recall that the lines $l_j = \{\l \in \bC : \Re(i b_j \l) = 0\}$
divide the complex plane into $m$ sectors
$\sigma_1,\ldots,\sigma_m$. Let $k \in \{1,2,3\}$ be fixed. Note
that the point $z_k$ can be not feasible but it is clear from
definition of $\omega_0(\cdot)$ and $\omega_1(\cdot)$ that they
are constant in each sector $\sigma_j$. Hence if $z_k$ is not
feasible, that is, it lies at one of the lines $l_{jk} = \{\l
\in \bC : \Re(i b_j \l) = \Re(i b_k \l)\}$, we can replace it by
any point with arbitrary close argument to make it feasible and
to conserve the condition (a) of the theorem. Thus, we can
assume that the points $z_1,z_2,z_3$ are feasible. Then
combining condition (b) of the theorem with
Proposition~\ref{prop:DeltaY=(c0+c1/l)} implies for $k \in
\{1,2,3\}$
\begin{equation}
    \left|\Delta(\l)\right| \geqslant C\left|\omega_0(z_k)
    + \frac{\omega_1(z_k)}{\l}\right|e^{\Re(i \tau_k \l)}
    \geqslant C_1 \frac{e^{\Re(i \tau_k \l)}}{|\l|}, \quad |\l|>R,\ \arg \l = \arg z_k,
\end{equation}
where $C,C_1>0$, $\tau_k := \sum_{\Re(i b_j z_k) > 0 } b_j$ and
$R$ is sufficiently large. To complete the proof it remains to
apply Theorem~\ref{th:compl.gen}.
\end{proof}
The following result is easily derived from
Theorem~\ref{th:explicit.nxn} (cf.~\cite[Corolarry
3.2]{MalOri12}).
%
%
\begin{corollary} \label{cor:explicit.nxn}
Let $Q$ satisfy assumptions of Theorem~\ref{th:explicit.nxn},
and let $|\omega_0(\pm z)|+|\omega_1(\pm z)| \ne 0$ for some
admissible number $z$. Then the system of root functions of the
BVP~\eqref{eq:system}--\eqref{eq:CY(0)+DY(1)=0} is complete and
minimal in $L^2([0,1]; \bC^n)$.
\end{corollary}
%
%
\begin{remark}
In connection with Theorem~\ref{th:explicit.nxn} we mention the
fundamental  paper~\cite{Shk83} by A.A.~Shkalikov, where he
studied BVP for ODE~\eqref{eq:ODE} with spectral parameter in
boundary conditions. In particular, the notion of $B$-weakly
regular boundary conditions might be treated as an analogue of
the notion of normal BVP of order 0 from~\cite{Shk83}, while
conditions of Theorem~\ref{th:explicit.nxn} correlate with those
of normal BVP of order 1 from~\cite{Shk83}. Moreover, it is
proved in~\cite{Shk83} that the system of root functions of the
linearization of the normal BVP for ODE~\eqref{eq:ODE} is
complete in certain direct sums of Sobolev spaces. For certain
matrices $B = \diag(b_1, \ldots, b_n)$ with simple spectrum this
result correlate with~\cite[Theorem 1.2]{MalOri12} and
Theorem~\ref{th:explicit.nxn}.
\end{remark}
We first apply Theorem~\ref{th:explicit.nxn} to $2 \times 2$
case. Let
\begin{equation}
    \begin{pmatrix} C & D \end{pmatrix} = \begin{pmatrix}
        a_{11} & a_{12} & a_{13} & a_{14} \\
        a_{21} & a_{22} & a_{23} & a_{24}
    \end{pmatrix}, \quad J_{jk} := \det \begin{pmatrix}
        a_{1j} & a_{1k} \\
        a_{2j} & a_{2k}
    \end{pmatrix}, \quad j,k\in\{1,\ldots,4\}.
\end{equation}
%
%
\begin{proposition} \label{prop:2x2}
Let $n=2$, $\arg b_1 \ne \arg b_2$, and let $q_{12}$, $q_{21}$
be continuous at the endpoints 0 and 1. Then the system of root
functions of the boundary value
problem~\eqref{eq:system}--\eqref{eq:CY(0)+DY(1)=0} is complete
and minimal in $L^2\left([0,1]; \bC^2\right)$ whenever
\begin{eqnarray}
  \label{eq:J32} \left|J_{32}\right| + \left|b_1 J_{13} q_{12}(0) + b_2 J_{42} q_{21}(1)\right| &\ne& 0, \\
  \label{eq:J14} \left|J_{14}\right| + \left|b_1 J_{13} q_{12}(1) + b_2 J_{42} q_{21}(0)\right| &\ne& 0.
\end{eqnarray}
\end{proposition}
%
%
\begin{proof}[Proof]
Since $\arg b_1 \ne \arg b_2$ then there exists $z \in \bC$ such
that $\Re(i b_1 z) < 0 < \Re(i b_2 z)$. Then, in accordance with
definition of $J_{jk}$ and the numbers $\omega_0(z),
\omega_1(z)$,
\begin{gather}
    \omega_0(z) = J_{14}, \quad \omega_1(z) = \frac{J_{24} b_1 q_{21}(0) - J_{13} b_1 q_{12}(1)}{b_1-b_2}, \\
    \omega_0(-z) = J_{32}, \quad \omega_1(-z) = \frac{J_{31} b_2 q_{12}(0) - J_{42} b_2 q_{21}(1)}{b_2-b_1}.
\end{gather}
Conditions~\eqref{eq:J32},~\eqref{eq:J14} imply $|\omega_0(\pm
z)| + |\omega_1(\pm z)| \ne 0$. Hence
Corollary~\ref{cor:explicit.nxn} yields the result.
\end{proof}
%
%
\begin{remark}
In the case of $2 \times 2$ Dirac-type systems $(b_1<0<b_2)$
this result improves Theorem~5.1 from~\cite{MalOri12} where the
completeness was proved under the stronger assumption $q_{12},
q_{21}\in C^1[0,1]$ {while was stated for $q_{12}, q_{21}\in
C[0,1]$. It happened because the precise  version} of Lemma 5.4
from~\cite{MalOri12} requires a stronger assumption $Q(\cdot)\in
C^1([0,1]; \bC^{n\times n})$ instead of $Q(\cdot)\in C([0,1];
\bC^{n\times n})$ (cf.~\cite[Theorem 1.1]{Mal99}). In our
forthcoming paper the completeness property of
BVP~\eqref{eq:system}--\eqref{eq:CY(0)+DY(1)=0} for $2 \times 2$
Dirac-type systems will be discussed in detail.

In the case $b_2 b_1^{-1} \not\in \bR$
Proposition~\ref{prop:2x2} improves Theorems 1.4 and 1.6
from~\cite{AgiMalOri12} where the completeness property was
proved for analytic $Q(\cdot)$.
\end{remark}
%
%
The next result demonstrates that Theorem~\ref{th:explicit.nxn}
cannot be treated as a perturbation result since unperturbed
operator $L_{C,D}(0)$ may have incomplete system of root
functions.
%
%
\begin{corollary} \label{cor:y1(0)=0}
Let $\k \in \{1,\ldots,n-1\}$, $\Re\, b_j < 0$ for $j \in
\{1,\ldots, \k\}$, $\Re\, b_j
> 0$ for $j \in \{\k+1,\ldots,n\}$, and the first boundary
condition in~\eqref{eq:CY(0)+DY(1)=0} is of the form $y_1(0)=0$.
Then the following holds:

(i) Assume that $Q$ is continuous at the endpoints 0 and 1 of
the segment $[0,1]$,
\begin{equation} \label{eq:detTBCD}
    \det T_{B}(C,D) \ne 0 \qquad\text{and}\qquad \sum_{j=\k+1}^n
    \frac{\det T_{-B}^{c_j \to c_1}}{b_1-b_j}\cdot q_{1j}(0) \ne 0.
\end{equation}
Then the system of root functions of the operator $L_{C,D}(Q)$
is complete and minimal in $L^2([0,1]; \bC^n)$.

(ii) If $q_{1j}(x) = 0$ for $x \in [0,\eps]$, $j \in
\{2,\ldots,n\}$, for some $\eps > 0$, then the system of root
functions of the operator $L_{C,D}(Q)$ is incomplete in
$L^2([0,1]; \bC^n)$ and its defect is infinite. In particular,
the latter is valid for the operator $L_{C,D}(0)$ with zero
potential.
\end{corollary}
%
%
\begin{proof}[Proof]
\textbf{(i)} The condition $y_1(0)=0$ means that $c_{11}=1$,
$c_{1k}=0$ for $k\in \{2,\ldots,n\}$, and $d_{1k}=0$, $k\in
\{1,\ldots,n\}$. Therefore, the matrix $T_{-B}(C,D)$ has zero
first line and hence $\omega_0(i) = 0$. Moreover, due to the
structure of the first row of $\begin{pmatrix} C &
D\end{pmatrix}$, $\det T_{-B}^{d_k \to d_j} = 0$, $j, k \in
\{1,\ldots,n\}$, and $\det T_{-B}^{c_j \to c_k} = 0$, for $k >
1$. Now the assumption  on $\Re\, b_j$, definition of
$\omega_1(\cdot)$, and condition~\eqref{eq:detTBCD} together
imply
\begin{equation}
    \omega_1(i) = \sum_{j=\k+1}^n
        \frac{\det T_{-B}^{c_j \to c_1} \cdot b_1 q_{1j}(0)}{b_1-b_j} \ne 0.
\end{equation}
Due to the first relation in~\eqref{eq:detTBCD} $\omega_0(-i) =
\det T_B(C,D) \ne 0$. It remains to apply
Corollary~\ref{cor:explicit.nxn}.

\textbf{(ii)} Under our assumption each solution
$y=\col(y_1,\ldots,y_n)$ of the
problem~\eqref{eq:system}--\eqref{eq:CY(0)+DY(1)=0} satisfies
\begin{equation}
    y_1' = i b_1 \l y_1 + i b_1 q_{11}(x) y_1, \quad x \in [0,\eps], \quad\text{and}\quad y_1(0)=0.
\end{equation}
By the uniqueness theorem, $y_1(x) = 0$ for $x \in [0,\eps]$.
Hence each $f = \col(f_1,0,\ldots,0) \in L^2([0,1]; \bC^n)$ with
$f_1$ vanishing on $[\eps,1]$ is orthogonal to the system of
root functions of the operator $L_{C,D}(Q)$.
\end{proof}
%
%
\begin{remark}
Let $n=3$, $\k=1$ and $y_1(0)=0$. Then
condition~\eqref{eq:detTBCD} takes the form
\begin{equation} \label{eq:d22d33-d23d32}
    \begin{vmatrix} d_{22} & d_{23} \\ d_{32} & d_{33} \end{vmatrix} \ne 0 \quad\text{and}\quad
    \begin{vmatrix} d_{21} & c_{23} \\ d_{31} & c_{33} \end{vmatrix} \cdot \frac{q_{12}(0)}{b_2-b_1} +
    \begin{vmatrix} d_{21} & c_{22} \\ d_{31} & c_{32} \end{vmatrix} \cdot \frac{q_{13}(0)}{b_1-b_3} \ne 0.
\end{equation}
Therefore, if $\left|q_{12}(0)\right|+\left|q_{13}(0)\right| \ne
0$, then, in general, the system of root functions of the
operator $L_{C,D}(Q)$ with the first boundary condition
$y_1(0)=0$, is complete in $L^2([0,1]; \bC^3)$.
\end{remark}
%
%
Finally, we specify Corollary~\ref{cor:explicit.nxn} for
$4\times 4$ Dirac type equation subject to special boundary
conditions. This statement will be applied in
Section~\ref{sec:Timoshenko} for study of the Timoshenko beam
model.
%
%
\begin{corollary} \label{cor:n=4}
Let $n=4$, $B = \diag(-b_1, b_1, -b_2, b_2)$, where $b_1, b_2
> 0$, let $Q \in L^1([0,1]; \bC^{4 \times 4})$, where $Q$ is continuous at
the endpoints 0 and 1, and matrices $C$ and $D$ are of the form
\begin{equation} \label{eq:n=4.CD}
    C = \begin{pmatrix}
         1 & 1 & 0 & 0 \\
         0 & 0 & 0 & 0 \\
         0 & 0 & 1 & 1 \\
         0 & 0 & 0 & 0
    \end{pmatrix}, \quad
    D = \begin{pmatrix}
         0 & 0 & 0 & 0 \\
         d_1 & d_2 & 0 & 0 \\
         0 & 0 & 0 & 0 \\
         0 & 0 & d_3 & d_4
    \end{pmatrix}.
\end{equation}
Assume that
\begin{equation} \label{eq:d2d4+...,d1d3+...}
    |d_2 d_4| + |d_1 d_4 q_{12}(1)| + |d_2 d_3 q_{34}(1)| \ne 0, \quad
    |d_1 d_3| + |d_2 d_3 q_{21}(1)| + |d_1 d_4 q_{43}(1)| \ne 0.
\end{equation}
Then the system of root functions of the
BVP~\eqref{eq:system}--\eqref{eq:CY(0)+DY(1)=0} is complete and
minimal in $L^2([0,1]; \bC^4)$.
\end{corollary}
%
%
\begin{proof}[Proof]
By the definition of the matrix $T_B(C,D)$,
\begin{equation} \label{eq:n=4.TBCD}
    T_{B}(C,D) = \begin{pmatrix}
         1 & 0 & 0 & 0 \\
         0 & d_2 & 0 & 0 \\
         0 & 0 & 1 & 0 \\
         0 & 0 & 0 & d_4
    \end{pmatrix},
\end{equation}
and hence
\begin{equation} \label{eq:n=4.omega0}
    \omega_0(-i) = \det T_{B}(C,D) = d_2 d_4.
\end{equation}
In our case the double sum in~\eqref{eq:c1zp.def} for
$\omega_1(-i)$ involves only values $j=1,3$ and $k=2,4$. It
follows from definition of matrices $T_{izB}^{c_j \rightarrow
c_k}$ and $T_{izB}^{d_k \rightarrow d_j}$ that
\begin{eqnarray}
    \det T_{B}^{c_1 \rightarrow c_2} &= d_2 d_4, & \det T_{B}^{d_2 \rightarrow d_1} = d_1 d_4, \\
    \det T_{B}^{c_1 \rightarrow c_4} &= 0,\ \quad  & \det T_{B}^{d_4 \rightarrow d_1} = 0, \\
    \det T_{B}^{c_3 \rightarrow c_2} &= 0,\ \quad  & \det T_{B}^{d_2 \rightarrow d_3} = 0, \\
    \det T_{B}^{c_3 \rightarrow c_4} &= d_2 d_4, & \det T_{B}^{d_4 \rightarrow d_3} = d_2 d_3.
\end{eqnarray}
Inserting these expressions into~\eqref{eq:c1zp.def} we obtain
\begin{equation}
    \omega_1(-i) = \frac{1}{2} \bigl(d_2 d_4 q_{21}(0) + d_1 d_4 q_{12}(1)
    + d_2 d_4 q_{43}(0) + d_2 d_3 q_{34}(1)\bigr).
\end{equation}
Note that if $d_2 = 0$, then $\omega_1(-i) = \frac{1}{2} d_1 d_4
q_{12}(1)$. On the other hand, if $d_4=0$, then $\omega_1(-i) =
\frac{1}{2} d_2 d_3 q_{34}(1)$. This allows us to rewrite
condition $|\omega_0(-i)| + |\omega_1(-i)| \ne 0$ in the form of
the first relation in~\eqref{eq:d2d4+...,d1d3+...}.

Similarly, one verifies that condition $|\omega_0(i)| +
|\omega_1(i)| \ne 0$ turns into the second relation
in~\eqref{eq:d2d4+...,d1d3+...}. One completes the proof by
applying Corollary~\ref{cor:explicit.nxn}.
\end{proof}
The following simple lemma will be useful for us in Section 6.
%
%
\begin{lemma} \label{lem:4x4}
Condition~\eqref{eq:d2d4+...,d1d3+...} is fulfilled if and only
if each of the following conditions is satisfied
\begin{eqnarray}
    \label{eq:d1d2} |d_1| + |d_2| \ne 0, & & |d_3| + |d_4| \ne 0, \\
    \label{eq:d1d3} |d_1| + |d_3| \ne 0, & & |d_2| + |d_4| \ne 0, \\
    \label{eq:d1q21} |d_1| + |q_{21}(1)| \ne 0, & & |d_2| + |q_{12}(1)| \ne 0, \\
    \label{eq:d3q43} |d_3| + |q_{43}(1)| \ne 0, & & |d_4| + |q_{34}(1)| \ne 0.
\end{eqnarray}
\end{lemma}
%
%
\begin{proof}[Proof]
If $d_1 d_2 d_3 d_4 \ne 0$ then the statement is obvious.
Further assume that $d_j = 0$ for some $j \in \{1,2,3,4\}$. Let
for definiteness, $d_1=0$. Then
conditions~\eqref{eq:d1d2}--\eqref{eq:d3q43} are satisfied if
and only if
\begin{equation}
    d_2 d_3 q_{21}(1) \ne 0 \quad\text{and}\quad |d_4| + |q_{34}(1)| \ne 0.
\end{equation}
This, in turn, is equivalent to~\eqref{eq:d2d4+...,d1d3+...}
whenever $d_1=0$, and we are done.
\end{proof}
%
%
\subsection{Example}
%
%
Here we illustrate Proposition~\ref{prop:2x2} by investigation
of the completeness and minimality of the system of vector
functions
\begin{equation} \label{eq:levin}
    \left\{\begin{pmatrix}
        e^{a n x} \sin n x \\
        n e^{a n x} (\sin n x + i \cos n x)
    \end{pmatrix}\right\}_{n \in \bZ \setminus \{0\}},\quad  a \in \bC,
\end{equation}
in the space $L^2([0,\pi]; \bC^2)$.
%
%
\begin{corollary} \label{cor:levin}
Let
\begin{equation} \label{eq:ia.notin}
    i a \not \in (-\infty, -1] \cap [1, \infty).
\end{equation}
Then system~\eqref{eq:levin} is complete and minimal in
$L^2([0,\pi]; \bC^2)$.
\end{corollary}
%
%
\begin{proof}[Proof]
Since $a \ne \pm i$ there exists $\theta \in \bC \setminus \{\pi
n\}_{n \in \bZ}$ such that $a = \ctg \theta$. Consider the
following boundary value problem
\begin{gather}
    \label{eq:system.levin}
    \begin{cases}
        y_1' =& e^{i \theta} \l y_1 + y_2, \\
        y_2' =& e^{-i \theta} \l y_2,
    \end{cases} \\
    \label{eq:y10=y11=0} y_1(0) = y_1(1) = 0.
\end{gather}
Straightforward calculation shows that its spectrum is simple,
consists of the eigenvalues $\left\{\frac{\pi n}{\sin
\theta}\right\}_{n \in \bZ \setminus \{0\}}$, and the system of
the corresponding eigenfunctions is
\begin{equation} \label{eq:levin.pi}
    \left\{\begin{pmatrix}
        e^{a \pi n x} \sin \pi n x \\
        \pi n \cdot e^{(a - i) \pi n x}
    \end{pmatrix}\right\}_{n \in \bZ \setminus \{0\}}.
\end{equation}
It is easily seen that a potential matrix of the operator
$L_{C,D}(Q)$ associated with the boundary value
problem~\eqref{eq:system.levin}--\eqref{eq:y10=y11=0} is
constant: $Q(\cdot)=\begin{pmatrix} 0 & -e^{-i \theta} \\ 0 & 0
\end{pmatrix}$. Clearly,
\begin{equation}
    B = \diag(b_1, b_2) := -i \diag (e^{i \theta}, e^{-i \theta}).
\end{equation}
Moreover, due to~\eqref{eq:ia.notin} $\arg b_1 \ne \arg b_2$.

Clearly, boundary conditions~\eqref{eq:y10=y11=0} imply
$J_{13}=1$, while the other determinants $J_{jk}$ are zero. In
particular, boundary conditions~\eqref{eq:y10=y11=0} are
non-weakly regular and even degenerate: $\Delta_0(\cdot) \equiv
0$. However, conditions~\eqref{eq:J32}--\eqref{eq:J14} take now
the form $q_{12}(0) q_{12}(1) \ne 0$ and clearly, are fulfilled.
Hence, by Proposition~\ref{prop:2x2}, the system of
eigenvectors~\eqref{eq:levin.pi} is complete and minimal in
$L^2([0,1]; \bC^2)$. The latter is equivalent to the
completeness and minimality of the system~\eqref{eq:levin} in
$L^2([0,\pi]; \bC^2)$.
\end{proof}
\begin{remark}
In connection with Corollary~\ref{cor:levin} let us consider one
more system of functions $\cK_a = \{e^{a n x} \sin n x\}_{n \in
\bZ \setminus \{0\}}$. Clearly, it is a system of the
eigenfunctions of the problem
\begin{equation} \label{eq:y''-2bly'+'2y=0}
    y'' - 2 a \l y' + (a^2 + 1) \l^2 y = 0, \quad y(0) = y(\pi) = 0.
\end{equation}
It is known (see~\cite[Part II, Appendix A1]{Lev96},
\cite{LyubTka84} and the references therein) that this system is
twofold complete in $L^2[0,\pi]$ in the sense of
M.V.~Keldysh~\cite{Kel51}. The latter means completeness of the
system
\begin{equation}
    \left\{\col(e^{a n x} \sin n x, n e^{a n x} \sin n x) \right\}_{n \in \bZ\setminus \{0\}}
\end{equation}
in $L^2([0,\pi]; \bC^2)$. So, the statement of
Corollary~\ref{cor:levin} is in a sense close to the twofold
completeness and minimality of the system $\cK_a$. Note that
investigation of the completeness and basis property of a "half"
system $\cK_a^+ := \{e^{a n x} \sin n x\}_{n=1}^{\infty}$ in
$L^2[0,\pi]$ has been initiated by A.G.~Kostyuchenko and
constitutes his named problem.

Note also that in the case $a \in \bR$ problem
\eqref{eq:y''-2bly'+'2y=0} naturally arises in the investigation
of the solvability of the following elliptic boundary value
problem in the strip $\Omega = [0,\pi] \times \bR_+$:
\begin{equation} \label{eq:Lu=d2u/dx2+...}
    \begin{cases}
        L u := \frac{\partial^2 u}{\partial x^2}
        - 2 a \frac{\partial^2 u}{\partial x\partial t}
        + (a^2+1)\frac{\partial^2 u}{\partial t^2} = 0, \\
        u(0,t) = u(\pi,t) = 0, \quad t \geqslant 0, \\
        u(x,0) = u_0(x), \quad u_0 \in L^2[0,\pi].
    \end{cases}
\end{equation}
Since equation $Lu=0$ is elliptic, the Cauchy problem in the
strip is incorrect. Applying the Fourier method, i.e. seeking
for a solution of~\eqref{eq:Lu=d2u/dx2+...} in the form $u(x,t)
= e^{\lambda t}y(x)$, leads to
problem~\eqref{eq:y''-2bly'+'2y=0}.
\end{remark}
%
%
\subsection{Necessary conditions of completeness.}
%
%
Next we present some necessary conditions of completeness.
%
%
\begin{proposition} \label{prop:necess}
Let boundary conditions~\eqref{eq:CY(0)+DY(1)=0} be of the form
$y(0)=Ay(1)$, where $\det A \ne 0$,
\begin{equation} \label{eq:AB+BA=0}
    AB+BA=0 \quad\text{and}\quad Q(1-x)=A^{-1}Q(x)A,\quad x \in [0,\eps], \quad\text{for some}\quad \eps>0.
\end{equation}
Then the defect of the system of root functions of the operator
$L_{C,D}(Q)$ in $L^2([0,1]; \bC^n)$ is infinite.
\end{proposition}
%
%
\begin{proof}[Proof]
Let $\l$ be an eigenvalue of $L$ and let $\{u_p(x)\}_{p=1}^m$ be
a chain of the eigenfunction and associated functions of the
operator $L$ corresponding to $\l$. Put $u_0(x):=0$. It is clear
that $u_p(\cdot)$, $p \in \{0,1,\ldots,m\}$, satisfies boundary
conditions~\eqref{eq:CY(0)+DY(1)=0} and the following identity
holds
\begin{equation} \label{eq:Lu=...}
    L u_p(x) = \l u_p(x) + u_{p-1}(x), \quad x \in [0,1], \quad p \in \{1,\ldots,m\}.
\end{equation}
Denote $v_p(x) := A u_p(1-x)$. It follows from~\eqref{eq:Lu=...}
and~\eqref{eq:system} that
\begin{align}
    (u_p')(1-x)
    &= i B (\l - Q(1-x)) u_p(1-x) + i B u_{p-1}(1-x) \nonumber \\
    &= i B \left[(\l - Q(1-x)) A^{-1} v_p(x) + A^{-1} v_{p-1}(x) \right].
\end{align}
Further, combining relations~\eqref{eq:AB+BA=0} with the
definition of $v_p$ yields
\begin{align} \label{eq:Lvpx}
    L v_p(x) &= -i B^{-1} v_p'(x) + Q(x) v_p(x) \nonumber \\
             &=  i B^{-1} A \cdot (u_p')(1-x) + Q(x) v_p(x) \nonumber \\
             &= -i A B^{-1} i B \left[\left(\l - Q(1-x)\right) A^{-1} v_p(x) + A^{-1} v_{p-1}(x)\right] + Q(x) v_p(x) \nonumber \\
             &= \l v_p(x) + v_{p-1}(x) + (Q(x) - A Q(1-x) A^{-1}) v_p(x) \nonumber \\
             &= \l v_p(x) + v_{p-1}(x), \quad x \in [0,\eps].
\end{align}
Next, due to the assumption, $v_p(0)=Au_p(1)=u_p(0)$. Thus, both
$u_p$ and $v_p$ satisfy the same equation~\eqref{eq:Lu=...} for
$x \in [0,\eps]$ and have equal initial data at zero. Therefore,
by the Cauchy uniqueness theorem, %
\begin{equation}
    u_p(x) = v_p(x) = A u_p(1-x), \quad x \in [0,\eps].
\end{equation}
Further, let $f \in L^2([0,1]; \bC^n)$ and let
\begin{equation}\label{eq:f(x)=0,x<eps}
    f(x)=0 \quad \text{for}\quad x \in [\eps,1-\eps]\quad \text{and}\quad
    f(1-x)=-A^*f(x),\quad \text{for}\quad x \in [0, \eps].
\end{equation}
Then one has for $p \geqslant 0$
\begin{align}
    \int_0^1\langle u_p(x), f(x) \rangle dx
    &= \int_0^{\eps}\langle u_p(x), f(x) \rangle dx + \int_{0}^{\eps}\langle u_p(1-x), f(1-x) \rangle dx \nonumber \\
    &= \int_0^{\eps}\left(\langle u_p(x), f(x) \rangle + \langle A^{-1} u_p(x), -A^* f(x) \rangle \right) dx = 0.
\end{align}
This identity shows that each vector-function $f$ satisfying
\eqref{eq:f(x)=0,x<eps} is orthogonal to the system of root
functions of the operator $L_{C,D}(Q)$. This completes the
proof.
\end{proof}
Note that existence of a nonsingular solution of the matrix
equation $AB+BA=0$ is equivalent to the similarity of the
matrices $B$ and $-B$: $ABA^{-1} = -B$. Since $B$ is diagonal,
the latter amounts to saying that the spectra $\sigma(B)$ and
$\sigma(-B)$ coincide with their multiplicities. Thus we can
restate Proposition~\ref{prop:necess} as follows.
%
%
\begin{corollary} \label{cor:necess}
Let $n=2p$ and $B = \diag(\wt{B},-\wt{B})$, where
\begin{equation}
    \wt{B} = \diag(I_{n_1} b_1, \ldots, I_{n_r} b_r), \quad n_1+\ldots+n_r=p.
\end{equation}
Further, let
\begin{equation}
    A = \begin{pmatrix} 0 & A_1 \\ A_2 & 0\end{pmatrix}, \quad A_j = \diag(A_{j1},\ldots,A_{jr}),
    \quad A_{jk} \in \GL(n_k, \bC), \quad j \in \{1,2\},
\end{equation}
let boundary conditions~\eqref{eq:CY(0)+DY(1)=0} be of the form
$y(0)=Ay(1)$, and let
\begin{equation} \label{eq:Q(1-x)=...}
    Q(1-x) = A^{-1} Q(x) A, \quad x \in [0,\eps], \quad\text{for some}\quad \eps>0.
\end{equation}
Then the system of root functions of the operator $L_{C,D}(Q)$
is incomplete in $L^2([0,1]; \bC^n)$ and its defect is infinite.
\end{corollary}
%
%
\begin{proof}[Proof]
Due to the block structure of the matrices $\wt{B}$, $A_1$ and
$A_2$, one has $AB+BA=0$. Since $A_{jk}$ is nonsingular, $\det A
\ne 0$. Therefore, Proposition~\ref{prop:necess} completes the
result.
\end{proof}
%
%
\begin{remark}
Note that in the case of $2\times 2$ Dirac system ($B =
\diag(-1,1)$, $q_{11} \equiv q_{22} \equiv 0$)
Proposition~\ref{prop:necess} turns
into~\cite[Proposition~5.12]{MalOri12}. Indeed, consider
$2\times 2$ Dirac equation subject to the boundary conditions
$y_1(0) = \alpha_1 y_2(1)$, $y_2(0) = \alpha_2 y_1(1)$. Setting
$A = \begin{pmatrix} 0 & \alpha_1 \\ \alpha_2 & 0\end{pmatrix}$,
one rewrites these conditions as $y(0)=Ay(1)$. Moreover,
condition~\eqref{eq:Q(1-x)=...} turns into $ \alpha_1
q_{21}(1-x) = \alpha_2 q_{12}(x)$, $x \in [0,\eps] \cap
[1-\eps,1]$, for some $\eps>0$, i.e. coincides with the
respective condition from~\cite{MalOri12}. Similar result for
Sturm-Liouville operator subject to degenerate boundary
conditions was proved earlier in~\cite{Mal08}.
\end{remark}
%
%
\section{The Riesz basis property for root functions}
\label{sec:riesz}
%
%
Here we investigate the Riesz basis property for operator
$L_{C,D}(Q)$ by reduction it to the operator
$L_{\wt{C},\wt{D}}(\wt{Q})$ being a perturbation of a normal
operator. To this end we find conditions for matrices $C$ and
$D$ guarantying that $L_{C,D}(0)$ is normal.
%
%
\begin{lemma} \label{lem:normal}
(i) An operator $L:=L_{C,D}(0)$ is normal if and only if
\begin{equation} \label{eq:CBC*=DBD*}
C B C^* = D B D^*.
\end{equation}

(ii) Boundary conditions~\eqref{eq:CY(0)+DY(1)=0} are regular,
i.e. $\det T_{izB}(C,D) \ne 0$ for each admissible $z$,
whenever~\eqref{eq:CBC*=DBD*} is fulfilled.

(iii) If $Q \in L^1([0,1];\bC^{n \times n})$ and condition
\eqref{eq:CBC*=DBD*} is satisfied, then the system of root
functions of the operator $L_{C,D}(Q)$ is complete and minimal
in $L^2([0,1]; \bC^n)$.
\end{lemma}
%
%
\begin{proof}[Proof]
\textbf{(i)} It is easily seen that
\begin{equation}
    L L^* y = L^* L y = - (B B^*)^{-1} y'', \quad y \in W^{2,2}([0,1];\bC^n).
\end{equation}
Therefore, $L$ is normal if and only if $\dom(L) = \dom(L^*)$,
which is equivalent to
\begin{equation}
    (Lf,g) = (f,L^*g), \quad f,g \in \dom(L).
\end{equation}
In turn, integrating by parts one gets that this identity is
equivalent to
\begin{equation} \label{eq:Bf0=Bf1}
    \langle B^{-1} f(0), g(0) \rangle = \langle B^{-1} f(1), g(1) \rangle, \quad f,g \in \dom(L).
\end{equation}
Put $\wt{B} := \diag(B^{-1},-B^{-1})$ and equip the space $\cH =
\bC^n \oplus \bC^n$ with the bilinear form
\begin{equation} \label{eq:wfg}
    w(u,v) := \langle \wt{B} u, v \rangle = \langle B^{-1} u_1, v_1 \rangle
    - \langle B^{-1} u_2, v_2 \rangle, \quad u = \col(u_1, u_2),\ v = \col(v_1, v_2).
\end{equation}
Now condition~\eqref{eq:Bf0=Bf1} takes the form
\begin{equation} \label{eq:wuv1}
    w(u,v)=0, \quad u, v \in
    \cH_1 := \ker\begin{pmatrix} C & D \end{pmatrix} := \{\col(u_1, u_2) : C u_1 + D u_2 = 0\}.
\end{equation}
On the other hand, the equality $CBC^*=DBD^*$ can be rewritten
as
\begin{equation}
    \langle B^{-1} B C^* h, B C^* k \rangle
    = \langle B^{-1} (- B D^* h), -B D^* k \rangle, \quad h, k \in \bC^n.
\end{equation}
Using~\eqref{eq:wfg} one rewrites this equality in the form
\begin{equation} \label{eq:wuv2}
    w (u, v) = 0, \quad u, v \in
    \cH_2 := \{ \col(B C^* h, -B D^* h) : \ h \in \bC^{n} \}.
\end{equation}
Thus, to prove the statement it suffices to show
that~\eqref{eq:wuv1} is equivalent to~\eqref{eq:wuv2}. To this
end we prove that $\cH_1$ is the right $w$-orthogonal complement
of $\cH_2$,
\begin{equation}
    \cH_1 = \cH_2^{[\bot]} := \{u \in \cH : w(v, u)=0, v \in \cH_2\}.
\end{equation}
Indeed, if
\begin{equation}
    v = \col(B C^* h, -B D^* h) \in \cH_2 \quad\text{and}\quad
    u = \col(u_1, u_2) \in \cH,
\end{equation}
then
\begin{equation}
    w(v, u) = \langle {B^{-1}}(BC^*) h, u_1 \rangle -
    \langle {B^{-1}}(-BD^*) h, u_2 \rangle = \langle h, C u_1 + D u_2 \rangle,
    \qquad h \in \bC^{n}.
\end{equation}
It follows that $w(v, u) = 0$ for each $v \in \cH_2$ if and only
if $Cu_1 + Du_2 = 0$, i.e. $u \in \cH_1$. Next, maximality
condition~\eqref{eq:rankCD} yields $\dim \cH_1 = \dim \cH_2 =
n$.

Now, if~\eqref{eq:wuv2} is satisfied, then $\cH_2 \subset
\cH_2^{[\bot]} = \cH_1$. Since $\dim \cH_1 = \dim \cH_2$, one
has $\cH_1 = \cH_2$ and~\eqref{eq:wuv1} is fulfilled. The
opposite implication is derived similarly.

\textbf{(ii)} Since $L=L_{C,D}(0)$ is normal,
condition~\eqref{eq:wuv1} is satisfied. Let
\begin{equation}
    \beta_1^{-1}, \beta_2^{-1}, \ldots, \beta_{2n}^{-1}
\end{equation}
be the eigenvalues of $\wt{B}$ and let $e_1, e_2, \ldots,
e_{2n}$ be the corresponding normalized eigenvectors. Note that
\begin{equation}
    \beta_k = -\beta_{n+k} = b_k, \quad k \in \{1,\ldots,n\}.
\end{equation}
For every admissible $z$, i.e. for $z$ satisfying
\begin{equation}
    \Re(i z b_k) \ne 0, \quad k \in \{1,\ldots,n\},
\end{equation}
we put
\begin{equation}
    \cH_z := {\rm span}\{e_k : \Re(i z \beta_k) > 0 \}.
\end{equation}
Since
\begin{equation}
    \beta_{n+k} = - \beta_{k}, \quad k\in \{1,\ldots, n\},
\end{equation}
then $\dim \cH_z = n$ for every admissible $z$. Next we note
that
\begin{equation}
    T_{izB}(C,D)= \left. \begin{pmatrix} C & D \end{pmatrix} \right|_{\cH_z}.
\end{equation}
Therefore,
\begin{equation}
    \det T_{izB}(C,D)\ne 0 \quad \Leftrightarrow \quad
    \ker \begin{pmatrix} C & D \end{pmatrix} \cap \cH_z = \{0\}.
\end{equation}
Let $u \in \cH_z$. Then
\begin{equation}
    u = \sum_{\Re(i z \beta_k) > 0} c_k e_k,
\end{equation}
for some $c_1, \ldots, c_n \in \bC$, and
\begin{equation}
    \Re(i z \langle u , \wt{B} u \rangle)
    = \sum_{\Re(i z \beta_k) > 0} |c_k|^2 \Re(i z \overline{\beta_k^{-1}})
    = \sum_{\Re(i z \beta_k) > 0} \frac{|c_k|^2}{|\beta_k|^2} \Re(i z \beta_k).
\end{equation}
Hence
\begin{equation} \label{eq:Re(izh)}
    \Re(i z \langle u , \wt{B} u \rangle) > 0, \quad u \in \cH_{z} \setminus \{0\}.
\end{equation}
On the other hand, due to~\eqref{eq:wuv1},
\begin{equation}
    \langle u, \wt{B} u \rangle = \ol{\langle \wt{B} u, u \rangle} = 0,
    \quad u \in \ker \begin{pmatrix} C & D \end{pmatrix}.
\end{equation}
Combining this fact with~\eqref{eq:Re(izh)} one obtains $\ker
\begin{pmatrix} C & D
\end{pmatrix} \cap \cH_z = \{0\}$ and we are done.

\textbf{(iii)} It follows from (ii) that boundary
conditions~\eqref{eq:CY(0)+DY(1)=0} are weakly $B$-regular. Now
the completeness and minimality of the root functions of the
operator $L_{C,D}(Q)$ is implied by~\cite[Theorem
1.2]{MalOri12}.
\end{proof}
%
%
\begin{remark}
Let $Q \in L^2([0,1];\bC^{n \times n})$. Then (unbounded)
multiplication operator
\begin{equation}
    Q: f \to Q(x)f, \quad f \in L^2([0,1];\bC^{n}),
\end{equation}
is relatively compact with respect to $L_{C,D}(0)$. Therefore
statement (iii) is implied by the classical Keldysh theorem
(cf.~\cite[Theorem 4.3]{Markus86}) if in addition the spectrum
of $L_{C,D}(0)$ lies on the union of rays
\begin{equation}
    \{\l \in \bC : \arg \l = \varphi_k\}, \quad k \in \{1,\ldots,n\}.
\end{equation}
\end{remark}
%
%
Recall the following definitions from~\cite{GohKre65}
and~\cite{Markus86}.
%
%
\begin{definition} \label{def:basis}
(i) A sequence $\{f_k\}_{k=1}^{\infty}$ of vectors in $\fH$ is
called a \textbf{Riesz basis} if it admits a representation $f_k
= T e_k$, $k \in \bN$, where $\{e_k\}_{k=1}^{\infty}$ is an
orthonormal basis in $\fH$ and $T : \fH \to \fH$ is a bounded
operator with bounded inverse.

(ii) A sequence of subspaces $\{\fH_k\}_{k=1}^{\infty}$ is
called a \textbf{Riesz basis of subspaces} in $\fH$ if there
exists a complete sequence of mutually orthogonal subspaces
$\{\fH'_k\}_{k=1}^{\infty}$ and a bounded operator $T$ in $\fH$
with bounded inverse such that $\fH_k = T \fH'_k$, $k \in \bN$.

(iii) A sequence $\{f_k\}_{k=1}^{\infty}$ of vectors in $\fH$ is
called a \textbf{Riesz basis with parentheses} if each its
finite subsequence is linearly independent, and there exists an
increasing sequence $\{n_k\}_{k=0}^{\infty} \subset \bN$ such
that $n_0=1$ and the sequence $\fH_k :=
\Span\{f_j\}_{j=n_{k-1}}^{n_k-1}$, forms a Riesz basis of
subspaces in $\fH$. Subspaces $\fH_k$ are called blocks.
\end{definition}
%
%
To state the next result we need the following definition.
%
%
\begin{definition} \label{def:e-close}
Let $\{\varphi_k\}_{k=1}^n$ be a sequence of angles, $\varphi_k
\in (-\pi, \pi]$, and $\eps>0$. Numbers $\l, \mu \in \bC$ are
called $\eps$-close with respect to $\{\varphi_k\}_{k=1}^n$, if
for some $k \in \{1,\ldots,n\}$ we have
\begin{equation}
    \l, \mu \in \{z \in \bC: |\arg z - \varphi_k| < \eps\}
    \quad{and}\quad |\Re(e^{-i \varphi_k}(\l - \mu))| < \eps.
\end{equation}
In other words, $\l$ and $\mu$ are $\eps$-close if for some $k$
they belong to a small angle with the bisectrix
\begin{equation} \label{eq:l+.def}
    l_+(\varphi_k) := \{\l \in \bC : \arg \l = \varphi_k\}
\end{equation}
and their projections on this ray are close.
\end{definition}
%
%
Let $A$ be an operator with compact resolvent and let $\Omega$
be a bounded subset of $\bC$. We put
\begin{equation}
    N(\Omega,A) := \sum_{\l \in \sigma(A) \cap \Omega} m_a(\l, A) = \sum_{\l \in \sigma(A) \cap \Omega} \dim \cR_{\l}(A).
\end{equation}
Our investigation of the Riesz basis property of the operator
$L_{C,D}$ is based on the following statement that can easily be
extracted from~\cite{MarkMats84} and~\cite[\S I.6]{Markus86}.
%
%
\begin{proposition} \label{prop:basis}
Let $\fH$ be a separable Hilbert space and let $G$ be a normal
operator with compact resolvent in $\fH$. Assume that the
spectrum of $G$ lies on the union of rays $l_+(\varphi_1),
\ldots, l_+(\varphi_n)$, and
\begin{equation} \label{eq:basis.supNDrA}
    \sup_{z \in \bC} N(\bD(z), G) < \infty, \qquad \bD(z) := \{\zeta \in \bC : |\zeta-z| < 1\}.
\end{equation}
Finally, let $T$ be a bounded operator in $\fH$ and let $\eps >
0$ be arbitrarily small. Then the system of root vectors of the
operator $A=G+T$ forms a Riesz basis with parentheses in $\fH$,
where each block is constituted by the root subspaces
corresponding to the eigenvalues of $A$ that are mutually
$\eps$-close with respect to the sequence
$\{\varphi_k\}_{k=1}^n$.
\end{proposition}
%
%
\begin{proof}[Proof]
Since $T$ is bounded, it is relatively compact with respect to
$G$. Hence by~\cite[Corollary 3.7]{Markus86}, all but finitely
many eigenvalues of $A=G+T$ belong to the union of
non-overlapping sectors
\begin{equation}
    \Omega_j(\eps) := \{\l \in \bC : |\arg \l - \varphi_j| < \eps\},
    \quad j \in \{1,\ldots,n\}.
\end{equation}
Fix $j \in \{1,\ldots,n\}$ and set $G_j := e^{-i \varphi_j} G$.
Condition~\eqref{eq:basis.supNDrA} implies condition~(6.21)
of~\cite[Lemma 6.8]{Markus86},
\begin{equation}
    \sup_{k \in \bN} N\left((r_k - q r_k^p, r_k + q r_k^p), G_j\right) < \infty,
\end{equation}
with $p=0$, any $q > 0$ and any increasing sequence
$\{r_k\}_{k=1}^{\infty}$. Let $\{\l_{j,k}\}_{k=1}^{\infty}$ be
the sequence of eigenvalues of $A$ belonging to $\Omega_j(\eps)$
and ordered in ascending order of $\Re\left(e^{-i \varphi_j}
\l_{j,k}\right)$. Put
\begin{equation}
    r_k := \Re\left(e^{-i \varphi_j}\l_{j,k}\right) - \eps/2,
    \quad k \in \bN.
\end{equation}
Applying~\cite[Lemma 6.8]{Markus86} to the operator $G_j$ with
$p=0$, $q = \|T\| + 4 \eps$ and the above sequence
$\{r_k\}_{k=1}^{\infty}$, we conclude that there exists
\begin{equation}
    x_k \in (r_k - \eps/2, r_k + \eps/2), \quad k \in \bN,
\end{equation}
such that the sequence $\{x_k\}_{k=1}^{\infty}$ is strictly
monotone and the sequence of subspaces
\begin{equation}
    \fH_{j,k} := \Span \{\cR_{\l_{j,s}}(A) : x_{k} \leqslant
    \Re\left(e^{-i \varphi_j} \l_{j,s}\right) < x_{k+1}\},
    \quad k \in \bN,
\end{equation}
forms a Riesz basis of subspaces in its closed linear span. It
follows from definition of $r_k$ and $x_k$ that
\begin{equation}
    \Re\left(e^{-i \varphi_j} \l_{j,k}\right) - \eps < x_k < \Re\left(e^{-i
    \varphi_j} \l_{j,k}\right), \quad k \in \bN.
\end{equation}
Hence root subspaces of $A$ corresponding to the eigenvalues of
$A$, that are not $\eps$-close with respect to
$\{\varphi_k\}_{k=1}^n$, belong to different blocks. Let $\l'_1,
\ldots, \l'_m$ be the sequence of eigenvalues of $A$ not
belonging to the union of sectors $\cup_{j=1}^n \Omega_j(\eps)$.
Clearly, the family of subspaces
\begin{equation} \label{eq:subspaces}
    \{\cR_{\l'_k}\}_{k=1}^m, \{\fH_{1,k}\}_{k=1}^{\infty}, \ldots, \{\fH_{n,k}\}_{k=1}^{\infty},
\end{equation}
forms a Riesz basis of subspaces in its closed linear span.
Since the latter spans the system of root vectors of the
operator $A$, the Keldysh theorem (cf.~\cite[Theorem
4.3]{Markus86}) yields its completeness in $\fH$. Therefore, the
system of root vectors of the operator $A$ forms a Riesz basis
with parentheses having the required properties of the blocks.
\end{proof}
Now we are ready to prove our main result on the Riesz basis
property of BVP~\eqref{eq:system}--\eqref{eq:CY(0)+DY(1)=0}.
%
%
\begin{theorem} \label{th:basis.LCD}
Let $Q \in L^{\infty}([0,1]; \bC^{n \times n})$ and
\begin{equation} \label{eq:basis.BCD}
    B = \diag(B_j)_{j=1}^r, \quad C = \diag(C_j)_{j=1}^r, \quad D = \diag(D_j)_{j=1}^r,
\end{equation}
where
\begin{align}
    \label{eq:basis.Bj2}
    B_j &= \begin{pmatrix} b_{j1} I_{n_j} & 0 \\ 0 & b_{j2} I_{n_j} \end{pmatrix},
    \quad b_{j1} b_{j2}^{-1} \in (-\infty, 0), \quad \\
    \label{eq:basis.CjDj2}
    C_j &= \begin{pmatrix} C_{j1} & C_{j2} \\ 0 & 0 \end{pmatrix},\quad
    D_j = \begin{pmatrix} 0 & 0 \\ D_{j1} & D_{j2} \end{pmatrix},
    \quad C_{j1}, C_{j2}, D_{j1}, D_{j2} \in \GL (n_j, \bC).
\end{align}
Then the system of root functions of the operator $A :=
L_{C,D}(Q)$ forms a Riesz basis with parentheses in $L^2([0,1];
\bC^n)$, where each block is constituted by the root subspaces
corresponding to the eigenvalues of $A$ that are mutually
$\eps$-close with respect to the sequence of angles
\begin{equation}
    \{-\varphi_1,\ldots,-\varphi_r, \pi-\varphi_1,\ldots,\pi-\varphi_r\}.
\end{equation}
Here $\varphi_j = \arg(b_{j1}-b_{j2})$, $j \in \{1,\ldots,r\}$,
and $\eps>0$ is sufficiently small.
\end{theorem}
%
%
\begin{proof}[Proof]
First we show that the operator $L_{C,D}(Q)$ is similar to the
operator $L_{\wt{C},\wt{D}}(\wt{Q})$ with the same matrix $B$
and matrices $\wt{C},\wt{D}$ satisfying~\eqref{eq:CBC*=DBD*}. To
this end we use the gauge transform $W : y \to W(x) y$, with
$W(\cdot)$ satisfying
\begin{gather}
    \label{eq:WxB=BWx} W(x) B = B W(x), \quad x \in [0,1], \\
    \label{eq:W.in.C1} W(\cdot) \in C^1([0,1]; \bC^{n \times n}),
    \quad W^{-1}(\cdot) \in C([0,1]; \bC^{n \times n}).
\end{gather}
Then the operator $L_{C,D}(Q)$ is transformed into the operator
$L_{\wt{C},\wt{D}}(\wt{Q}) = W^{-1} L_{C,D}(Q) W$ with the same
$B$, and matrices $\wt{C}$, $\wt{D}$, $\wt{Q}(\cdot)$ given by
\begin{equation} \label{eq:wtC.wtD}
    \wt{C} := CW(0), \ \ \wt{D} := DW(1), \ \ \wt{Q}(x) :=
    W^{-1}(x) Q(x) W(x) -i W^{-1}(x) B^{-1} W'(x).
\end{equation}
Since $W(\cdot), W'(\cdot), W^{-1}(\cdot), Q(\cdot) \in
L^{\infty}([0,1]; \bC^{n \times n})$, then $\wt{Q} \in
L^{\infty}([0,1]; \bC^{n \times n})$.

Due to the block diagonal
structure~\eqref{eq:basis.BCD}--\eqref{eq:basis.CjDj2} of the
matrices $B$, $C_j$, and $D_j$, we can choose $W_0, W_1 \in
\GL(n, \bC)$ such that $W_k B = B W_k$, $k \in \{0,1\}$, and
\begin{eqnarray}
    \label{eq:C.W0}
    C W_0 = \diag(\wt{C}_j)_{j=1}^r, && \wt{C}_j :=
    \begin{pmatrix} I_{n_j} & b_j I_{n_j} \\ 0 & 0 \end{pmatrix},
    \quad b_j := \left(-b_{j1}b_{j2}^{-1}\right)^{1/2}, \\
    \label{eq:D.W1}
    D W_1 = \diag(\wt{D}_j)_{j=1}^r, && \wt{D}_j :=
    \begin{pmatrix} 0 & 0 \\ I_{n_j} & b_j I_{n_j} \end{pmatrix}, \quad j \in \{1,\ldots,r\}.
\end{eqnarray}
Choose any branch of logarithm and put $\wt{W} := \log (W_0^{-1}
W_1)$. Clearly, $\wt{W}$ is well defined since the matrix
$W_0^{-1} W_1$ is nonsingular. Hence $W(x) := W_0 e^{x \wt{W}}$
satisfies~\eqref{eq:WxB=BWx},~\eqref{eq:W.in.C1} and $W(0)=W_0$,
$W(1)=W_1$. Define a gauge transform $W: y \to W(x) y$. In view
of~\eqref{eq:wtC.wtD},~\eqref{eq:C.W0},~\eqref{eq:D.W1} the
matrices $\wt{C},\wt{D}$ of the new operator
$L_{\wt{C},\wt{D}}(\wt{Q}) = W^{-1} L_{C,D}(Q) W$ are $\wt{C} =
\diag(\wt{C}_j)_{j=1}^r$ and $\wt{D} = \diag(\wt{D}_j)_{j=1}^r$
where $\wt{C}_j$ and $\wt{D}_j$ are given by~\eqref{eq:C.W0} and
\eqref{eq:D.W1}, respectively.

Straightforward calculation shows that
\begin{equation}
    \wt{C}_j B_j \wt{C}_j^* = \wt{D}_j B_j \wt{D}_j^* = 0, \quad j \in \{1,\ldots,r\}.
\end{equation}
Hence $\wt{C}B \wt{C}^* = \wt{D}B \wt{D}^* = 0$. By
Lemma~\ref{lem:normal}, the operator $G :=
L_{\wt{C},\wt{D}}({0})$ is normal. Its spectrum coincides with
the set of zeros of the characteristic determinant
$\Delta(\cdot) = \det(\wt{C} + \wt{D}\wt{\Phi}(1,\cdot))$. The
fundamental matrix $\wt{\Phi}(\cdot,\l)$ of the operator
$L_{\wt{C},\wt{D}}({0})$ is $\wt{\Phi}(x,\l) = e^{i B \l x}$.
Hence, in view of the block-diagonal structure of the matrices
$B$, $\wt{C}$, $\wt{D}$, we obtain
\begin{equation}
    \Delta(\l) = \prod_{j=1}^r \det(\wt{C}_j + \wt{D}_j e^{i B_j \l})
    = \prod_{j=1}^r \det \begin{pmatrix} I_{n_j} & b_j I_{n_j} \\ e^{i b_{j1} \l} I_{n_j} & b_j e^{i b_{j2} \l} I_{n_j} \end{pmatrix}
    = \prod_{j=1}^r \left(b_j^{n_j} \cdot (e^{i b_{j2} \l} - e^{i b_{j1} \l})^{n_j}\right).
\end{equation}
Hence
\begin{equation}
    \sigma(G) = \left\{\frac{2 \pi k}{b_{j1}-b_{j2}} : k \in \bZ, j
    \in \{1, \ldots, r\}\right\}.
\end{equation}
Thus $\sigma(G)$ lies on the union of rays
$\{l_+(-\varphi_j)\}_1^r$ and $\{l_+(\pi-\varphi_j)\}_1^r$,
where
\begin{equation}
    \varphi_j = \arg(b_{j1}-b_{j2}), \quad j \in \{1,\ldots,r\}.
\end{equation}
Moreover, $\sigma(G)$ is the union of a finite number of
arithmetic progressions and multiplicities of eigenvalues are
bounded, hence condition~\eqref{eq:basis.supNDrA} is satisfied.
Since $\wt{Q}(\cdot)$ is bounded, then, by
Proposition~\ref{prop:basis}, the system of root functions of
the operator
\begin{equation}
    \wt{A} := L_{\wt{C},\wt{D}}(\wt{Q}) = L_{\wt{C},\wt{D}}(0) + \wt{Q} = G + \wt{Q}
\end{equation}
forms a Riesz basis with parentheses in $\fH$, where each block
is constituted by the root subspaces corresponding to the
mutually close eigenvalues of $A$ in the sense of
Definition~\ref{def:e-close}. Since $A = L_{C,D}(Q)$ is similar
to $\wt{A}$, the same is true for the root functions of the
operator $L_{C,D}(Q)$.
\end{proof}
As a consequence of this result we obtain \emph{the Riesz basis
property of the system of root functions for Dirac system with
general splitting boundary conditions}.
%
%
\begin{corollary} \label{cor:basis.split}
Let $n=2m$, $Q \in L^{\infty}([0,1]; \bC^{n \times n})$ and let
\begin{align}
    B &= \diag(b_1 I_m, b_2 I_m), \quad b_1 < 0 < b_2, \\
    C &= \begin{pmatrix} C_1 & C_2 \\ 0 & 0 \end{pmatrix}, \quad
    D = \begin{pmatrix} 0 & 0 \\ D_1 & D_2 \end{pmatrix}, \quad
    C_1, C_2, D_1, D_2 \in \GL (m, \bC).
\end{align}
Then the system of root functions of the operator $L_{C,D}(Q)$
forms a Riesz basis with parentheses in $L^2([0,1]; \bC^{n
\times n})$.
\end{corollary}
%
%
Similarly to Theorem~\ref{th:basis.LCD} we can obtain the
following result.
%
%
\begin{proposition} \label{prop:basis.LCD.per}
Let
\begin{gather}
  \label{eq:basis.B}
    B = \diag(b_1 I_{n_1}, \ldots, b_r I_{n_r}), \quad n = n_1 + \ldots + n_r, \\
  \label{eq:basis.CD}
    C = \diag(C_j)_{j=1}^r, \quad D = \diag(D_j)_{j=1}^r,
    \quad C_j, D_j \in GL(n_j, \bC),\ j \in \{1,\ldots,r\},
\end{gather}
and $Q \in L^{\infty}([0,1]; \bC^{n \times n})$. Then the system
of root functions of the operator $A := L_{C,D}(Q)$ forms a
Riesz basis with parentheses in $L^2([0,1]; \bC^n)$, where each
block is constituted by the root subspaces corresponding to the
eigenvalues of $A$ that are mutually $\eps$-close with respect
to the sequence of angles $\{-\varphi_1,\ldots,-\varphi_r,
\pi-\varphi_1,\ldots,\pi-\varphi_r\}$. Here $\varphi_j = \arg
b_j$, $j \in \{1,\ldots,r\}$, and $\eps>0$ is sufficiently
small.
\end{proposition}
%
%
\begin{proof}[Proof]
The proof is similar to that of Theorem~\ref{th:basis.LCD}. At
first choosing an appropriate gauge transform, we transform the
operator $L_{C,D}(Q)$ into $L_{\wt{C},\wt{D}}(\wt{Q})$ with
$\wt{C}_j = \wt{D}_j = I_{n_j}$. It follows that the operator $G
:= L_{\wt{C},\wt{D}}(0)$ is normal and its spectrum is of the
form
\begin{equation}
    \sigma(G) = \left\{2 \pi k / b_j : k \in \bZ, \ j \in \{1, \ldots, r\}\right\}.
\end{equation}
Hence the same argument as in the proof of
Theorem~\ref{th:basis.LCD} yields the result.
\end{proof}
A direct consequence of this result is the \emph{Riesz basis
property of the periodic $($reps. antiperiodic$)$ BVP} with
general matrix $B$.
%
%
\begin{corollary} \label{cor:basis.per}
Let $B = \diag(b_1,\ldots,b_n) \in \GL(n,\bC)$, $C = \pm D =
I_n$ and $Q \in L^{\infty}([0,1]; \bC^{n \times n})$. Then the
system of root functions of the operator $L_{C,D}(Q)$ forms a
Riesz basis with parentheses in $L^2([0,1]; \bC^{n \times n})$.
\end{corollary}
%
%
\begin{remark}
In the case of Dirac type systems ($B=B^*$) we can extend the
statements of Theorem~\ref{th:basis.LCD} and
Proposition~\ref{prop:basis.LCD.per} to the case of $Q \in
L^2([0,1]; \bC^{n \times n})$. Indeed, it suffices to apply
Theorem 2 from the recent paper~\cite{Shk10} instead of the
quoted results from~\cite{MarkMats84} and~\cite[\S
I.6]{Markus86}. Note however, that in~\cite[Theorem 2]{Shk10}
only the basis property instead of the Riesz basis property was
stated.
\end{remark}

\begin{remark}\label{rem:Riesz.basis}
The Riesz basis property for $2 \times 2$ Dirac equation subject
to splitting boundary conditions has been investigated in
numerous papers~\cite{TroYam01, TroYam02, HasOri09,
DjaMit10BariDir, DjaMit12UncDir}. The most general statement
covering Corollary~\ref{cor:basis.split} (for $n=1$) was
obtained by Djakov and Mityagin~\cite{DjaMit12UncDir} who
relaxed the assumption on a potential matrix to $Q \in
L^2([0,1]; \bC^2)$. Moreover, these authors proved
in~\cite{DjaMit12UncDir} the Riesz basis property for $2 \times
2$ Dirac equation subject to \textbf{general strictly regular
boundary conditions}.

For $2 \times 2$ Dirac system Corollary~\ref{cor:basis.per} was
proved in~\cite{DjaMit10BariDir} under weaker assumption $Q \in
L^2([0,1]; \bC^2)$. Moreover, these authors found
out~\cite{DjaMit12Crit} a criterion for the system of root
functions of the periodic boundary value problem for $2 \times
2$ Dirac equation to contain a Riesz basis (without
parentheses). Similar result for Sturm-Liouville operator
$-\frac{d^2}{dx^2} +q$ was obtained by different methods
in~\cite{GesTka09,GesTka11} and~\cite{DjaMit12Crit}. Both
criteria are formulated directly in terms of periodic and
Dirichlet eigenvalues. {Moreover,  in} \cite[Theorem
13]{DjaMit12TrigDir}, \cite[Theorem 19]{DjaMit12Crit} (see also
\cite{DjaMit11TrigHill}) it is established \emph{criteria} {for
eigenfunctions and associated functions to form a Riesz basis
for periodic 1D Dirac operator (resp. Sturm-Liouville operator)
in terms of the Fourier coefficients of} $Q$ (resp. $q$).
Equivalence of this formulation to that in terms of periodic and
Dirichlet eigenvalues is explained in~\cite[Theorem
24]{DjaMit12Crit}. Let us mention in this connection the paper
\cite{ShkVel09} where Riesz basis property for periodic
Sturm-Liouville operator was obtained under \emph{certain
explicit sufficient conditions} in terms of Fourier coefficients
of a potential $q$.
\end{remark}
%
%
In the simplest case $B = I_n$ we can indicate a criterion for
the system of root functions of the operator $L_{C,D}(Q)$ to
form a Riesz basis with parentheses.
%
%
\begin{corollary} \label{cor:B=In}
Let $B = I_n$ and $Q \in L^1([0,1];\bC^{n \times n})$. Then the
system of root functions of the operator $L_{C,D}(Q)$ forms a
Riesz basis with parentheses in $L_2([0,1]; \bC^n)$ if and only
if $\det(C \cdot D) \ne 0$.
\end{corollary}
%
%
\begin{proof}[Proof]
Applying the gauge transform $y \to W(x)y$ with $W(\cdot)$
described in the beginning of the proof of
Proposition~\ref{prop:DeltaY=(c0+c1/l)}, we see that the
operator $L_{C,D}(Q)$ is similar to the operator
$L_{\wt{C},\wt{D}}(0)$ with $\wt{C}=C$, $\wt{D}=DW(1)$ and zero
potential matrix. Further, since $B=I_n$, then
$T_{B}(\wt{C},\wt{D}) = DW(1)$ and $T_{-B}(\wt{C},\wt{D})=C$.
Hence, $\det(C \cdot D) \ne 0$ if and only if $\det
T_B(\wt{C},\wt{D}) \cdot \det T_{-B}(\wt{C},\wt{D}) \ne 0$.
Therefore, by~\cite[Proposition 4.6]{MalOri12}, the system of
root functions of the operator $L_{\wt{C},\wt{D}}(0)$ has
infinite defect, whenever $\det(C \cdot D) = 0$. On the other
hand, if $\det(C \cdot D) \ne 0$ then, by
Proposition~\ref{prop:basis.LCD.per}, applied with $r=1$ and
$Q=0$, the system of root functions of the operator
$L_{\wt{C},\wt{D}}(0)$ forms a Riesz basis with parentheses.
Similarity of the operators $L_{C,D}(Q)$ and
$L_{\wt{C},\wt{D}}(0)$ completes the proof.
\end{proof}
%
%
\section{General properties of the resolvent and spectral synthesis}
\label{sec:resolv}
%
%
\subsection{General properties of the resolvent}
\label{subsec:resolv}
%
%
Let $\fH := L^2([0,1]; \bC^n)$. We start with the explicit form
of the resolvent of the operator $L_{C,D}(Q)$ in terms of the
fundamental matrix solution $\Phi(x,\l)$. This statement is of
folklore character (cf.~\cite[Theorem 9.4.1]{Atk64}). We present
the proof for the sake of completeness.
%
%
\begin{lemma} \label{lem:Green}
Assume that BVP~\eqref{eq:system}--\eqref{eq:CY(0)+DY(1)=0} is
non-degenerate, i.e. $\Delta(\cdot) \not \equiv 0$. Then for any
$\l \in \bC$ with $\Delta(\l) \ne 0$ the Green function of the
problem~\eqref{eq:system}--\eqref{eq:CY(0)+DY(1)=0} is
\begin{equation} \label{eq:Gxtl=...}
    G(x,t;\l) =
    \begin{cases}
        \ \ \Phi(x,\l) (C+D\Phi(1,\l))^{-1} C \Phi^{-1}(t,\l) i B, & 0 \leqslant t \leqslant x, \\
        -\Phi(x,\l) (C+D\Phi(1,\l))^{-1} D \Phi(1,\l) \Phi^{-1}(t,\l) i B, & x < t \leqslant 1.
    \end{cases}
\end{equation}
Moreover,
\begin{equation} \label{eq:Gxx-0-Gxx+0=...}
    G(x,x-0;\l) - G(x,x+0;\l) = i B, \quad x \in (0,1).
\end{equation}
\end{lemma}
%
%
\begin{proof}[Proof]
Consider the non-homogenous system
\begin{equation}
    -i B^{-1} y' + Q(x) y = \l y + f, \quad x \in [0,1],
\end{equation}
with $f \in \fH$. The general solution of this system is
\begin{equation} \label{eq:y=int.Fx.Ft}
    y(x,\l) = \int_0^x \Phi(x,\l) \Phi^{-1}(t,\l)iBf(t)dt + \Phi(x,\l) y(0,\l),
\end{equation}
where $y(0,\l) = \col(y_1(0,\l), \ldots, y_n(0,\l)) \in \bC^n$
is arbitrary vector. Inserting this expression
into~\eqref{eq:CY(0)+DY(1)=0} we get after straightforward
calculations
\begin{equation} \label{eq:wtC(l)=...}
    y(0,\l) = - (C+D\Phi(1,\l))^{-1} D \Phi(1,\l) \int_0^1 \Phi^{-1}(t,\l) iBf(t) dt =: K(\l) f.
\end{equation}
Combining this expression with~\eqref{eq:y=int.Fx.Ft} we arrive
at
\begin{align}
    y(x,\l) &= \int_0^x \Phi(x,\l) \Bigl(I_n - (C+D\Phi(1,\l))^{-1} D \Phi(1,\l)\Bigr)\Phi^{-1}(t,\l) i B f(t)dt \nonumber \\
    &+ \int_x^1 -\Phi(x,\l) (C+D\Phi(1,\l))^{-1} D \Phi(1,\l)\Phi^{-1}(t,\l) i B f(t)dt
    = \int_0^1 G(x,t;\l)f(t)dt,
\end{align}
where $G(x,t;\cdot)$ is given by~\eqref{eq:Gxtl=...}.
Formula~\eqref{eq:Gxx-0-Gxx+0=...} directly follows
from~\eqref{eq:Gxtl=...}.
\end{proof}
Combining~\eqref{eq:y=int.Fx.Ft} with~\eqref{eq:wtC(l)=...} we
get the following alternative representation for the resolvent.
\begin{corollary} \label{cor:Green}
Assume the conditions of Lemma~\ref{lem:Green}, i.e.
$\rho(L_{C,D}(Q)) \ne \varnothing$. Then for $f \in \fH$ and $\l
\in \rho(L_{C,D}(Q))$
\begin{equation} \label{eq:LCDQ=LI0Q+}
    (L_{C,D}(Q) - \l)^{-1} f = (L_{I_n,0}(Q) - \l)^{-1} f +
    \Phi(\cdot,\l) K(\l) f,
\end{equation}
where $K(\l) : \fH \to \bC^n$ is given by~\eqref{eq:wtC(l)=...}
and $(L_{I_n,0}(Q) - \l)^{-1}$ is Volterra operator of the form
\begin{equation} \label{eq:LI0Q}
    ((L_{I_n,0}(Q) - \l)^{-1} f)(x) = \int_0^x \Phi(x,\l) \Phi^{-1}(t,\l) i B f(t) dt.
\end{equation}
\end{corollary}
%
%
\begin{theorem} \label{th:Resolv.dif.in.S1}
For $\l \in \rho(L_{C_1,D_1}(Q_1)) \cap \rho(L_{C_2,D_2}(Q_2))$
the following inclusion holds
\begin{equation} \label{eq:RL1-RL2infS1}
    (L_{C_1,D_1}(Q_1) - \l)^{-1} - (L_{C_2,D_2}(Q_2) - \l)^{-1} \in \fS_1(\fH).
\end{equation}
Moreover, the following trace formula holds
\begin{multline} \label{eq:spur}
    \tr \bigl( (L_{C_1,D_1}(Q_1) - \l)^{-1} - (L_{C_2,D_2}(Q_2) - \l)^{-1} \bigr) \\
    = \tr \int_0^1 \bigl( \Phi_1(x,\l) (C_1+D_1\Phi_1(1,\l))^{-1} C_1 \Phi_1^{-1}(x,\l) \\
    - \Phi_2(x,\l) (C_2+D_2\Phi_2(1,\l))^{-1} C_2 \Phi_2^{-1}(x,\l)\bigr) i B dx,
\end{multline}
where $\Phi_j(\cdot,\l)$ is the fundamental matrix of the
equation $L(Q_j) y = \l y$ satisfying $\Phi_j(0,\l) = I_n$, $j
\in \{1,2\}$.
\end{theorem}
%
%
\begin{proof}[Proof]
\textbf{(i)} Put $T_{11} := L_{C_1,D_1}(Q_1)$ and $T_{22} :=
L_{C_2,D_2}(Q_2)$. Consider the auxiliary operators $T_j :=
L_{C,D}(Q_j)$, $j \in \{1,2\}$, where $C=I_n$ and $D=0$.
Clearly, $T_j$ corresponds to the initial value problem
\begin{equation}
    -i B^{-1} y' + Q_j(x) y = 0, \quad y(0)=0.
\end{equation}
Hence $\rho(T_1)=\rho(T_2)=\bC$. Therefore, for $\l \in
\rho(T_{11}) \cap \rho(T_{22})$ we have
\begin{multline} \label{eq:T11-T22}
    (T_{11}-\l)^{-1} - (T_{22}-\l)^{-1} = \bigl((T_{11}-\l)^{-1} - (T_{1}-\l)^{-1}\bigr) \\
    + \bigl((T_{1}-\l)^{-1} - (T_{2}-\l)^{-1}\bigr)
    + \bigl((T_{2}-\l)^{-1} - (T_{22}-\l)^{-1}\bigr).
\end{multline}
It follows from~\eqref{eq:LCDQ=LI0Q+} that the first and the
third summands in~\eqref{eq:T11-T22} are operators of finite
rank,
\begin{align}
    \label{eq:dim(ran(RT11-RT1))<=n}
        \dim \Bigl(\range\bigl((T_{11} - \l)^{-1} - (T_{1} - \l)^{-1}\bigr)\Bigr) \leqslant n, \\
    \label{eq:dim(ran(RT2-RT22))<=n}
        \dim \Bigl(\range\bigl((T_{2} - \l)^{-1} - (T_{22} - \l)^{-1}\bigr)\Bigr) \leqslant n.
\end{align}

Next, according to~\cite[Theorem 2.7]{GesMal09}, for each $x \in
[0,1]$ the matrix $Q(x) := Q_2(x)-Q_1(x)$ admits the generalized
polar decomposition
\begin{equation}
    Q(x) = U(x) \cdot |Q(x)| = |Q^*(x)| \cdot U(x) = |Q^*(x)|^{1/2} \cdot U(x) \cdot |Q(x)|^{1/2},
\end{equation}
where $|A| := (A^* A)^{1/2}$, $A \in \bC^{n \times n}$, and
$U(x)$ is a unitary matrix, $U^*(x) = U^{-1}(x)$. Clearly,
$|Q(\cdot)|$ and $|Q^*(\cdot)|$ are measurable matrix-function
and $U(\cdot)$ can be chosen to be measurable. In turn, these
families induce a generalized polar decomposition of the
(unbounded) multiplication operator $Q: f(x)\to Q(x)f(x)$ in
$\fH$, i.e.
\begin{equation}
    Q = U \cdot |Q| = |Q^*| \cdot U = |Q^*|^{1/2} \cdot U \cdot |Q|^{1/2}, \quad
    |Q| = \bigl(Q^*Q\bigr)^{1/2}, \quad |Q^*| = \bigl(Q Q^*\bigr)^{1/2},
\end{equation}
and $|Q|$ denotes the multiplication operator in $\fH$ with the
matrix $|Q(\cdot)|$.

Let $G_{j}(\cdot,\cdot;\l)$ be the Green function of the
operator $T_{j}$, $j \in \{1,2\}$. It is easily seen that
\begin{equation} \label{eq:T1-T2}
    (T_{1}-\l)^{-1} - (T_{2}-\l)^{-1} = \ol{(T_{1}-\l)^{-1} |Q^*|^{1/2}}
    \cdot \Bigl(U |Q|^{1/2} (T_{2}-\l)^{-1}\Bigr) = K_1 K_2,
\end{equation}
where $\ol{T}$ denotes the closure of the operator $T$ and for
$f \in \fH$,
\begin{align}
    (K_1 f) (x) &:= \int_0^1 \Bigl(G_{1}(x,t;\l) |Q^*(t)|^{1/2}\Bigr) f(t) dt, \quad x \in [0,1], \\
    (K_2 f) (x) &:= \int_0^1 \Bigl(U(x) |Q(x)|^{1/2} G_{2}(x,t;\l)\Bigr) f(t) dt, \quad x \in [0,1].
\end{align}
It follows from~\eqref{eq:Gxtl=...} that the kernel
\begin{equation}
    G_{j}(\cdot,\cdot;\l)\in L^\infty([0,1] \times [0,1]; \bC^{n\times n}), \quad j\in \{1,2\}.
\end{equation}
Moreover, since
\begin{equation}
    |Q(\cdot)|^{1/2}, |Q^*(\cdot)|^{1/2} \in L^2([0,1]; \bC^{n\times n})
    \quad\text{and}\quad U(\cdot) \in L^{\infty}([0,1]; \bC^{n\times n}),
\end{equation}
the operator $K_j$ is of Hilbert-Schmidt class, $K_j \in
\fS_2(\fH)$, $j\in \{1,2\}$. Combining these relations with
factorization identity~\eqref{eq:T1-T2} yields
\begin{equation}
    (T_1-\l)^{-1} - (T_2-\l)^{-1} = K_1 K_2 \in \fS_1(\fH).
\end{equation}
In turn, combining this relation
with~\eqref{eq:T11-T22}--\eqref{eq:dim(ran(RT2-RT22))<=n} we
arrive at~\eqref{eq:RL1-RL2infS1}.

\textbf{(ii)} Let $G_{jj}(\cdot,\cdot;\l)$ be the Green function
of the operator $T_{jj}$, $j\in \{1,2\}$. By (i), the difference
\mbox{$(T_{11}-\l)^{-1} - (T_{22}-\l)^{-1}$} is of trace class
integral operator with the kernel
\begin{equation}
    {\wt G}(\cdot,\cdot;\l) := G_{11}(\cdot,\cdot;\l) - G_{22}(\cdot,\cdot;\l).
\end{equation}
In view of~\eqref{eq:Gxtl=...} and~\eqref{eq:Gxx-0-Gxx+0=...}
the kernel ${\wt G}(\cdot,\cdot;\l)$ is continuous,
\begin{equation}
    {\wt G}(\cdot,\cdot;\l)\in C([0,1] \times [0,1]; \bC^{n \times n}).
\end{equation}
On the other hand, if $K$ is of trace class integral operator in
$L^2([0,1];\bC^n)$ with continuous kernel $K(\cdot,\cdot) \in
\bC^{n \times n}$, then
\begin{equation}
    \tr K = \int_{0}^1 (\tr K(x,x))dx
\end{equation}
(see~\cite[Corollary III.10.2]{GohKre65}). Combining this result
with formula~\eqref{eq:Gxtl=...} for the Green function
$G_{jj}(\cdot,\cdot;\l)$, $j\in \{1,2\}$,
yields~\eqref{eq:spur}.
\end{proof}
\begin{corollary}
Let $B = B^*$ and let $C_*, \ D_*\in \bC^{n \times n}$ be such
that $(L_{C,D}(Q))^* = L_{C_*,D_*}(Q^*)$ (see
Lemma~\ref{lem:adjoint}). Let also $\Phi(\cdot,\l)$ and
$\Phi_*(\cdot,\l)$ be the fundamental matrices of equations
$L(Q) y = \l y$ and $L(Q^*) y = \l y$, respectively, satisfying
$\Phi(0,\l) = \Phi_*(\cdot,\l) =I_n$. Finally, let
$\{\l_n\}_{n=1}^{\infty}$ be the sequence of all eigenvalues of
$L_{C,D}(Q)$, counting multiplicity, and let the system of root
functions of the operator $L_{C,D}(Q)$ be complete in
$L^2([0,1];\bC^n)$. Then for $\l \in \bR \cap \rho(L_{C,D}(Q))$
the following identity holds
\begin{multline}
    \sum_{n=1}^{\infty} \frac{-2\Im\, \l_n}{|\l_n-\l|^2}
    = \tr \int_0^1 \bigl( \Phi(x,\l) (C+D\Phi(1,\l))^{-1} C \Phi^{-1}(x,\l) \\
    - \Phi_*(x,\l) (C_*+D_*\Phi_*(1,\l))^{-1} C_* \Phi_*^{-1}(x,\l)\bigr) B dx.
\end{multline}
\end{corollary}
\begin{proof}[Proof]
By Theorem~\ref{th:Resolv.dif.in.S1}, the imaginary part of the
resolvent $(L_{C,D}(Q)-\l)^{-1}$ is of trace class operator,
\begin{equation}
    (L_{C,D}(Q)-\l)^{-1} - (L_{C,D}(Q)^*-\l)^{-1}\in \fS_1(\fH),
    \quad \l \in \bR \cap \rho(L_{C,D}(Q)),
\end{equation}
and trace formula~\eqref{eq:spur} holds. Combining
formula~\eqref{eq:spur} with Livsic theorem (see~\cite[Theorem
V.2.1]{GohKre65}) yields the result.
\end{proof}
To state the next result we recall some properties of the
classes $\cS_p(\fH)$ and $\cS_p^0(\fH)$, $p\in (0,\infty)$,
introduced in the following definition.
%
%
\begin{definition} \label{def:Sp}
Define for $p>0$
\begin{align}
    \label{eq:Sp.def}
        \cS_p(\fH) &= \{T\in \fS_\infty(\fH) \,|\,
        s_j(T) = O(j^{-1/p}) \quad\text{as}\quad j \to \infty \}, \\
    \label{eq:Sp0.def}
        \cS_p^0(\fH) &= \{T\in \fS_\infty(\fH) \,|\,
        s_j(T) = \ o(j^{-1/p}) \quad\text{as}\quad j \to \infty \},
\end{align}
where $s_j(T)$, $j \in \bN$, denote the singular values
($s$-numbers) of $T$ (i.e., the eigenvalues of $(T^*T)^{1/2}$
ordered in decreasing magnitude, counting multiplicity).
\end{definition}
%
%
Clearly $\fS_p \subset \cS_p^0 \subset \cS_p$. It is known that
$\cS_p(\fH)$ ($\cS_p^0(\fH)$) is a two-sided (non-closed) ideal
in $\cB(\fH)$. Clearly, $\cS_{p_1}\subset \cS_{p_2}$ and
$\cS_{p_1}^0\subset \cS_{p_2}^0$ if $p_1 > p_2$. The main
property of the classes $\cS_p(\fH)$ and $\cS_p^0(\fH)$ we need
in the sequel, is (see~\cite[\S II.2]{GohKre65})
\begin{equation} \label{eq:Sp1.Sp2.in.Sp}
    \cS_{p_1}\cdot \cS_{p_2} \subset \cS_{p}, \quad \text{and} \quad
    \cS_{p_1}\cdot \cS_{p_2}^0 \subset \cS_{p}^0, \quad \text{where} \quad
    p^{-1} = p_1^{-1} + p_2^{-1}.
\end{equation}
We need also a generalization of the known Ky-Fan lemma
(see~\cite[Theorem II.2.3]{GohKre65}).
%
%
\begin{lemma} \label{lem:Ky-Fan}
Let $A, B \in \fS_{\infty}(\fH)$, $r>0$, and let the following
conditions be satisfied
\begin{equation} \label{eq:lim.k.snk(A)}
    \lim_{k \to \infty} \left( k^r \cdot s_{n_k}(A) \right) = a,
    \qquad \lim_{n \to \infty} \left( n^r \cdot s_{n}(B) \right) = 0,
\end{equation}
where $\{n_k\}_{k=1}^{\infty}$ is an increasing sequence of
positive integers. Then
\begin{equation}
    \lim_{k \to \infty} \left( k^r \cdot s_{n_k}(A+B) \right) = a.
\end{equation}
\end{lemma}
%
%
\begin{proof}[Proof]
We follow the proof of Ky-Fan lemma~\cite[Theorem
II.2.3]{GohKre65}. Fix $\eps \in (0, 1)$ and for any $k \in \bN$
define $j = j_{\eps,k} := k - \lfloor \eps k \rfloor \in \bN$.
Then the Ky-Fan inequality,
\begin{equation}
    s_n(A+B) \leqslant s_m(A) + s_{n-m+1}(B), \qquad n \geqslant m,
\end{equation}
implies for each $k \in \bN$
\begin{align} \label{eq:k.snk(A+B)}
    k^r \cdot s_{n_k}(A+B)
    &\leqslant \left(\frac{k}{j}\right)^r \cdot \left(j^r \cdot s_{n_j}(A)\right)
    + \left(\frac{k}{n_k-n_j+1}\right)^r \cdot \left( (n_k-n_j+1)^r \cdot s_{n_k-n_j+1}(B)\right) \nonumber \\
    &\leqslant \frac{1}{(1-\eps)^r} \cdot \left(j^r \cdot s_{n_j}(A)\right)
    + \frac{1}{\eps^r} \cdot \left( (n_k-n_j+1)^r \cdot s_{n_k-n_j+1}(B)\right).
\end{align}
Here we have used that
\begin{equation}
    k - \eps k \leqslant j < k - \eps k + 1 \quad\text{and}\quad
    n_k - n_j \geqslant k-j.
\end{equation}
Note that
\begin{equation}
    j_{\eps,k} \to \infty \quad\text{and}\quad n_k - n_j + 1 \to \infty
    \quad\text{as}\quad k \to \infty.
\end{equation}
Hence tending $k$ to infinity in~\eqref{eq:k.snk(A+B)} and
using~\eqref{eq:lim.k.snk(A)} we derive
\begin{equation}
    \varlimsup_{k \to \infty} \left( k^r \cdot s_{n_k}(A+B) \right) \leqslant \frac{a}{(1-\eps)^r}.
\end{equation}
Tending $\eps$ to zero here we get
\begin{equation} \label{eq:ol.lim.snk}
    \varlimsup_{k \to \infty} \left( k^r \cdot s_{n_k}(A+B) \right) \leqslant a.
\end{equation}
Using inequality
\begin{equation}
    s_{n_j}(A) \leqslant s_{n_k}(A+B) + s_{n_j-n_k+1}(B), \quad j = k + \lfloor \eps k \rfloor,
\end{equation}
we obtain in a similar way that
\begin{equation}
    \varliminf_{k \to \infty} \left( k^r \cdot s_{n_k}(A+B) \right) \geqslant a.
\end{equation}
One completes the proof by combining this inequality
with~\eqref{eq:ol.lim.snk}.
\end{proof}
Now we are ready to find the asymptotic behavior of the
$s$-numbers of the resolvent operator $(L_{C,D}(Q)-\l)^{-1}$ for
each fixed $\l\in \rho(L_{C,D}(Q))$.
%
%
\begin{proposition} \label{prop:RL.in.S1+KyFan}
Let $\rho(L_{C,D}(Q)) \ne \varnothing$. Then for any $\l \in
\rho(L_{C,D}(Q))$ the following inclusion holds
\begin{equation} \label{eq:RL.in.S1}
    (L_{C,D}(Q)-\l)^{-1} \in \cS_1(\fH) \setminus \fS_1(\fH).
\end{equation}
Moreover, the sequence $\{s_k\}_{k \in \bN}$ of singular values
of the operator $(L_{C,D}(Q)-\l)^{-1}$ can be decomposed into
the union of $n$ disjoint non-increasing subsequences
$\{s_{j,k}\}_{k \in \bN}$, $j \in \{1,\ldots,n\}$, satisfying
\begin{equation} \label{eq:sjk}
    s_{j,k} = \frac{|b_j|+o(1)}{\pi k} \quad\text{as}\quad k \to \infty.
\end{equation}
\end{proposition}
%
%
\begin{proof}[Proof]
Put $L := L_{C,D}(Q)$. Since $\rho(L) \ne \varnothing$, its
spectrum is discrete. Alongside the operator $L_{C,D}(Q)$
consider the auxiliary operator $L_0:=L_{I_n,-I_n}(0)$
corresponding to the periodic boundary value problem
\begin{equation}
    -iB^{-1}y'=0, \quad y(0)=y(1).
\end{equation}
Straightforward calculation shows that the spectrum
$\sigma(L_0)$ is decomposed into $n$ series
\begin{equation}
    \{ 2 \pi k / b_j \}_{k \in \bZ}, \quad j \in \{1,\ldots,n\}.
\end{equation}
It is easily seen that the operator $L_0$ is normal. Hence the
singular values of the operator $(L_0 - \l)^{-1}$ coincide with
the absolute values of its eigenvalues. Hence the sequence
$\{s_k^{(0)}\}_{k \in \bN}$ of singular values of $(L_0 -
\l)^{-1}$ can be reordered and decomposed into the union of $n$
disjoint subsequences $\{\wt{s}_{j,k}^{(0)}\}_{k \in \bZ}$, $j
\in \{1,\ldots,n\}$, such that $\wt{s}_{j,k}^{(0)} =
\left|\frac{2 \pi k}{b_j} - \l\right|^{-1}$. Reordering the
sequence $\{\wt{s}_{j,k}^{(0)}\}_{k \in \bZ}$ in decreasing
order of magnitude we obtain the sequence $\{s_{j,k}^{(0)}\}_{k
\in \bN}$ satisfying
\begin{equation} \label{eq:sjk0}
    s_{j,k}^{(0)} = \frac{|b_j|+o(1)}{\pi k} \qquad\text{as}\qquad k \to \infty.
\end{equation}
Therefore, $s_k^{(0)} = O(k^{-1})$ as $k \to \infty$ and hence
\begin{equation}
    (L_0 - \l)^{-1} \in \cS_1(\fH), \quad \l \in \rho(L_0).
\end{equation}
Since $\sigma(L_0)$ and $\sigma(L)$ are at most countable,
$\rho(L_0) \cap \rho(L) \ne \varnothing$. By
Theorem~\ref{th:Resolv.dif.in.S1},
\begin{equation} \label{eq:RL0-RL.in.fS}
    (L - \l_0)^{-1} - (L_0 - \l_0)^{-1} \in \fS_1(\fH),
    \qquad \l_0 \in \rho(L_0) \cap \rho(L).
\end{equation}
Combining this relation with just established inclusion $(L_0 -
\l_0)^{-1} \in \cS_1(\fH)$ yields
\begin{equation}
    (L - \l_0)^{-1} = (L_0 - \l_0)^{-1} + \bigl((L - \l_0)^{-1} - (L_0 - \l_0)^{-1}\bigr) \in \cS_1(\fH).
\end{equation}
As an immediate consequence of~\eqref{eq:RL0-RL.in.fS} one gets
that the sequence $\{s_k^{(1)}\}_{k \in \bN}$ of $s$-numbers of
the operator $(L_0 - \l_0)^{-1} - (L - \l_0)^{-1}$ satisfies
\begin{equation} \label{eq:sjk1}
    s_k^{(1)} = o(k^{-1}) \quad\text{as}\quad k \to \infty.
\end{equation}
Combining~\eqref{eq:sjk0} with~\eqref{eq:sjk1} and applying
Lemma~\ref{lem:Ky-Fan} we arrive at the desired asymptotic
formula~\eqref{eq:sjk} for $s$-numbers of the operator
$(L-\l_0)^{-1}$.

Next, noting that $\cS_1(\fH)$ is two-sided ideal in $\fH$ and
using the Hilbert identity for the resolvent,
\begin{equation}
    (L-\l)^{-1} = (L-\l_0)^{-1} + (\l_0 - \l) (L-\l)^{-1} (L-\l_0)^{-1},
\end{equation}
one gets $(L - \l)^{-1} \in \cS_1(\fH)$ for $\l \in \rho(L)$.
Moreover, since $(L-\l)^{-1}, (L-\l_0)^{-1} \in \cS_1(\fH)$, we
obtain from~\eqref{eq:Sp1.Sp2.in.Sp} that
\begin{equation}
    (L-\l)^{-1} \cdot (L-\l_0)^{-1} \in \cS_1(\fH) \cdot \cS_1(\fH)
    \subset \cS_{1/2}(\fH) \subset \cS_1^0(\fH).
\end{equation}
Combining this relation with Lemma~\ref{lem:Ky-Fan} yields the
desired asymptotic formula for the $s$-numbers of the operator
$(L-\l)^{-1}$, $\l \in \rho(L)$. This implies that $(L-\l)^{-1}
\not\in \fS_1(\fH)$, which completes the proof.
\end{proof}
Next we improve Theorem~\ref{th:Resolv.dif.in.S1} (see
formula~\eqref{eq:RL1-RL2infS1})) assuming that $Q_2 - Q_1 \in
L^2([0,1]; \bC^{n \times n})$.
%
%
\begin{corollary} \label{cor:res_dif_in_S2/3}
Let $Q_2 - Q_1 \in L^2([0,1]; \bC^{n \times n})$. Then for $\l
\in \rho(L_{C_1,D_1}(Q_1)) \cap \rho(L_{C_2,D_2}(Q_2))$ the
following inclusion holds
\begin{equation} \label{eq:RL1-RL2infS2/3}
    (L_{C_1,D_1}(Q_1) - \l)^{-1} - (L_{C_2,D_2}(Q_2) - \l)^{-1} \in \cS_{2/3}^0(\fH).
\end{equation}
\end{corollary}
%
%
\begin{proof}[Proof]
Following the proof of Theorem~\ref{th:Resolv.dif.in.S1} it
suffices to show that
\begin{equation}
    (T_1-\l)^{-1} - (T_2-\l)^{-1} \in \cS_{2/3}^0(\fH), \quad \l \in \bC,
\end{equation}
where $T_j = L_{I_n,0}(Q_j)$, $j \in \{1,2\}$. Let
$G_2(\cdot,\cdot;\l)$ be the Green function of the operator
$T_2$. By Lemma~\ref{lem:Green} (cf. formula
\eqref{eq:Gxtl=...}),
\begin{equation}
    G_2(\cdot,\cdot;\l) \in L^{\infty}([0,1] \times [0,1]; \bC^{n \times n}).
\end{equation}
Combining this fact with the assumption
\begin{equation}
    Q(\cdot) := Q_2(\cdot) - Q_1(\cdot) \in L^2([0,1]; \bC^{n \times n}),
\end{equation}
one gets
\begin{equation}
    Q(x) \cdot G_2(x,t;\l) \in L^2([0,1] \times [0,1]; \bC^{n \times n}).
\end{equation}
The latter means that $Q \bigl(T_2 - \l \bigr)^{-1}$ is
Hilbert-Schmidt operator,
\begin{equation}
    Q \bigl(T_2 - \l \bigr)^{-1} \in \fS_2(\fH) \subset \cS_2^0(\fH).
\end{equation}
By Proposition~\ref{prop:RL.in.S1+KyFan}, $(T_1-\l)^{-1} \in
\cS_1(\fH)$. Combining this fact with
property~\eqref{eq:Sp1.Sp2.in.Sp} of classes $\cS_p(\fH)$,
yields
\begin{equation}
    (T_1-\l)^{-1} - (T_2-\l)^{-1} = (T_1-\l)^{-1} \cdot \Bigl(Q (T_2-\l)^{-1}\Bigr)
    \in \cS_1(\fH) \cdot \cS_2^0(\fH) \subset \cS_{2/3}^0(\fH),
\end{equation}
which completes the proof.
\end{proof}
%
%
\subsection{Spectral synthesis for dissipative Dirac type
operators} \label{subsec:dissip}
%
%
Recall that an operator $T$ in a Hilbert space $\fH$ is called
accumulative (dissipative) if
\begin{equation}
    \Im(Tf, f) \leqslant 0\ (\geqslant 0), \quad f \in \dom(T).
\end{equation}
Note that the accumulativity (dissipativity) of
BVP~\eqref{eq:system}--\eqref{eq:CY(0)+DY(1)=0} implies $B=B^*$.
Therefore, to investigate accumulative (dissipative) BVP we are
forced to consider Dirac type operators only. At first we
express accumulativity (dissipativity) of the operator
$L_{C,D}(Q)$ in terms of matrices $B,C,D$ and $Q(\cdot)$.
%
%
\begin{lemma}\label{lem:accum.CB*C-DB*D}
Let $B=B^*$. The operator $L_{C,D}(Q)$ is accumulative
(dissipative) if and only if $\Im\, Q \leqslant 0\ (\Im\, Q
\geqslant 0)$ and
\begin{equation} \label{eq:CBC*-DBD*}
    C B C^* - D B D^* \geqslant 0\ \ ( \leqslant 0).
\end{equation}
In particular, the operator $L_{C,D}(Q)$ is selfadjoint if and
only if $Q = Q^*$ and $CBC^* = DBD^*$.
\end{lemma}
%
%
\begin{proof}[Proof]
Integrating by parts and noting that $B=B^*$ one easily gets for
$f \in \dom(L_{C,D}(Q))$,
\begin{equation} \label{eq:2ImLQ}
    2 \Im \bigl(L_{C, D}(Q)f, f\bigr) =
    \langle B^{-1} f(0), f(0)\rangle  - \langle B^{-1} f(1), f(1)\rangle
    + \int_0^1 \langle 2 \Im\, Q(x) f(x), f(x) \rangle dx.
\end{equation}
Let us show that $L_{C,D}(Q)$ is accumulative if and only if
$\Im\, Q \leqslant 0$ and
\begin{equation} \label{eq:2ImL0}
    \langle B^{-1} h_0, h_0\rangle - \langle B^{-1} h_1, h_1\rangle \leqslant 0
    \quad\text{whenever}\quad C h_0 + D h_1 = 0, \quad h_0, h_1 \in \bC^n.
\end{equation}
Indeed, if $\Im\, Q \leqslant 0$ then due to~\eqref{eq:2ImLQ}
the inequality $\Im \bigl(L_{C, D}(Q)f, f\bigr) \le 0$ is
implied by~\eqref{eq:2ImL0}. Conversely, let $L_{C,D}(Q)$ be
accumulative. Choose any $f\in \dom(L_{C,D}(Q))$ with $f(0) =
f(1) = 0$ and substitute it in~\eqref{eq:2ImLQ}. Then the
inequality $\Im \bigl(L_{C, D}(Q)f, f\bigr) \le 0$ turns into
\begin{equation}
    \int_0^1 \langle \Im\, Q(x) f(x), f(x) \rangle dx \le 0, \quad
    f \in \dom(L_{C,D}(Q)) \cap W^{1,1}_0([0,1]; \bC^n),
\end{equation}
which yields $\Im\, Q(x) \leqslant 0$, $x \in [0,1]$. Further,
to extract~\eqref{eq:2ImL0} from the inequality $\Im\, L_{C,
D}(Q) \le 0$ we fix $h_0, h_1 \in \bC^n$ with $C h_0 + D h_1 =
0$ and substitute in~\eqref{eq:2ImLQ} function $f \in
\dom(L_{C,D}(Q))$ such that
\begin{equation}
    f(0)=h_0, \quad f(1)=h_1 \quad\text{and}\quad \supp f \subset [0,\eps] \cup [1-\eps,1],
\end{equation}
where $\eps > 0$ is sufficiently small. Thus, to prove the
statement it suffices to show that inequality~\eqref{eq:2ImL0}
is equivalent to \eqref{eq:CBC*-DBD*}.

As in the proof of Lemma~\ref{lem:normal}, we put $\wt{B} :=
\diag(B^{-1}, -B^{-1})$ and consider $\cH = \bC^n \oplus \bC^n$
as the Pontryagin space equipped with the bilinear form $w$
given by~\eqref{eq:wfg}. Clearly, the inertia indices of $\cH$
are $\k_{\pm} = \k_{\pm}(\wt{B})=n$, where $\k_+(A)$ ($\k_-(A)$)
denotes the number of positive (resp. negative) eigenvalues of a
matrix $A=A^*$. Let $\cH_1 := \ker
\begin{pmatrix} C & D \end{pmatrix} \subset \cH$. Then it is
clear that~\eqref{eq:2ImL0} is satisfied if and only if the
subspace $\cH_1$ is non-positive in $\cH$, i.e.
\begin{equation}
    \cH_1 \subset \{u \in \cH : \langle \wt{B} u, u \rangle \leqslant 0 \}.
\end{equation}
Further, condition~\eqref{eq:CBC*-DBD*} rewritten as
\begin{equation}
    \langle C^* h, B C^* h \rangle \geqslant \langle D^* h, B D^* h \rangle, \ h \in \bC^n,
\end{equation}
is equivalent to
\begin{equation}\label{eq:wtB.u.u}
    \langle \wt{B} u, u \rangle \geqslant 0,
    \quad u \in \cH_2 := \{ \col(B C^* h, -B D^* h) : \ h \in \bC^{n} \},
\end{equation}
meaning the non-negativity of the subspace $\cH_2$. Note that
maximality condition~\eqref{eq:rankCD} yields $\dim \cH_1 = \dim
\cH_2 = n$. As it is proved in Lemma~\ref{lem:normal}, $\cH_1$
is $w$-orthogonal complement of the subspace $\cH_2$. Since
$w$-orthogonal complement of a maximal non-positive subspace is
the maximal non-negative and vice versa, and taking into account
that the inertia indices of $\cH$ are $\k_{\pm} = n = \dim \cH_1
= \dim \cH_2$, one derives that $\cH_1$ is non-positive in $\cH$
if and only if $\cH_2$ is non-negative in $\cH$.
\end{proof}
%
%
\begin{lemma} \label{lem:accum.detTBCD}
Let $B=B^*$ and let the operator $L_{C,D}(0)$ be accumulative
(dissipative). Then
\begin{equation} \label{eq:detT-}
    \det T_{-B}(C, D) \ne 0 \quad (\det T_{B}(C, D) \ne 0).
\end{equation}
\end{lemma}
%
%
\begin{proof}[Proof]
Denote by $P_+$ and $P_-$ the spectral projections onto
"positive"\ and "negative"\ parts of the spectrum of a
selfadjoint matrix $B=B^*$, respectively. Then
\begin{equation}
    T_- := T_{-B}(C,D) = C P_+ + D P_-.
\end{equation}
By Lemma~\ref{lem:accum.CB*C-DB*D}, it suffices to show that
$\det T_- \ne 0$ is implied by~\eqref{eq:CBC*-DBD*}. Let $h_0
\in \ker T_-^*$. Since
\begin{equation}
    T^*_- = P_+ C^* + P_-D^* \quad\text{and}\quad P_+ P_- = P_- P_+ = 0,
\end{equation}
one gets
\begin{equation} \label{eq:P+C*h0=0}
    P_+ C^* h_0 = P_- D^* h_0 = 0.
\end{equation}
Setting $B_{\pm}:=\pm P_{\pm} B$, and noting that $B = B_+ -
B_-$, we rewrite inequality~\eqref{eq:CBC*-DBD*} in the
following form
\begin{equation}
    \|B^{1/2}_+ P_+ C^* h\|^2+\|B_-^{1/2}P_- D^* h\|^2
    \geqslant \|B^{1/2}_- P_- C^* h\|^2+\|B_+^{1/2}P_+ D^* h\|^2,
    \quad h \in \bC^n.
\end{equation}
Substituting in this inequality $h_0$ in place of $h$ and using
\eqref{eq:P+C*h0=0} we get $P_- C^* h_0 = P_+ D^* h_0 = 0$.
Combining these relations with~\eqref{eq:P+C*h0=0}, yields $C^*
h_0 = D^* h_0 = 0$. Hence the maximality condition $\ker(C C^* +
D D^*) = \{0\}$ implies $h_0 = 0$. Therefore, $\ker T_-^* =
\{0\}$, which yields~\eqref{eq:detT-}.
\end{proof}
Thus, in the case of dissipative (accumulative) boundary
conditions their regularity~\eqref{eq:detT+-} is reduced to the
solo condition $\det T_- \ne 0$ ($\det T_+ \ne 0$).

Passing to the spectral synthesis we recall the following
definition.
%
%
\begin{definition} \label{def:synt}
$(i)$ A compact operator $T$ in a separable Hilbert space $\fH$
is called complete if the system of its root vectors is complete
in $\fH$.

$(ii)$ A compact complete operator $T$ in $\fH$ admits the
spectral synthesis if for any invariant subspace $\fH_1$ of $T$
the restriction $T \upharpoonright_{\fH_1}$ is complete in
$\fH_1$.

$(iii)$ A closed operator $T$ in $\fH$ with $\rho(T) \ne
\varnothing$ is called complete if its resolvent is compact and
complete. We say that $T$ admits the spectral synthesis if its
resolvent admits the spectral synthesis.
\end{definition}
%
%
Recall that the operator $T$ is called $m$-accumulative
($m$-dissipative) if it has no accumulative (dissipative)
extensions. It is well known that accumulative operator $T$ is
$m$-accumulative if and only if $\rho(T) \ne \varnothing$, or
equivalently, $\bC_+ \subset \rho(T)$. Now we are ready to prove
our main result on the spectral synthesis.
%
%
\begin{theorem} \label{th:synthesis}
Let $B=B^*$ and let $L_{C,D}(Q)$ be a complete accumulative
(dissipative) operator. Then for any $\l \in \bC_+ (\bC_-)$ the
operator $T(\l) := (L_{C,D}(Q)-\l)^{-1}$ exists and admits the
spectral synthesis.
\end{theorem}
%
%
\begin{proof}
For definiteness we confine ourself to the accumulative case.
Since $L_{C,D}(Q)$ is accumulative then by
Lemma~\ref{lem:accum.CB*C-DB*D} condition~\eqref{eq:CBC*-DBD*}
is satisfied and $\Im\,Q \leqslant 0$. Hence
Lemma~\eqref{lem:accum.CB*C-DB*D} implies accumulativity of
$L_{C,D}(0)$. Hence by Lemma~\ref{lem:accum.detTBCD}
condition~\eqref{eq:detT-} is satisfied. Therefore, it follows
from Lemma~\ref{lem:Delta} (see
formula~\eqref{eq:Delta=(T+o(1)).E}) that the characteristic
determinant $\Delta(\cdot)$ is not identically zero. Thus,
$\rho(L_{C,D}(Q)) \ne \varnothing$, and the operator
$L_{C,D}(Q)$ is $m$-accumulative. Therefore, $\bC_+ \subset
\rho(L_{C,D}(Q))$ and operator $T(\l) = (L_{C,D}(Q) - \l)^{-1}$
exists for all $\l \in \bC_+$.

Since $B=B^*$, then, by Lemma~\ref{lem:adjoint}, the adjoint
operator $L_{C,D}(Q)^*$ is $L_{C,D}(Q)^* = L_{C_*,D_*}(Q^*)$
with appropriate $n \times n$ matrices $C_*$ and $D_*$. By
Theorem~\ref{th:Resolv.dif.in.S1},
\begin{equation}
    \left(L_{C,D}(Q)-\ol{\l}\right)^{-1} - \left(L_{C_*,D_*}(Q^*)-\ol{\l}\right)^{-1}\in \fS_1(\fH).
\end{equation}
Further, combining Hilbert identity with Proposition
\ref{prop:RL.in.S1+KyFan} and taking into account property
\eqref{eq:Sp1.Sp2.in.Sp} of the classes $\cS_{p}$, one gets
\begin{multline}
    (L_{C,D}(Q)-\l)^{-1} - (L_{C,D}(Q)-\ol{\l})^{-1}
    = 2 \Im\, \l \cdot \left(L_{C,D}(Q)-\l\right)^{-1} \cdot
    (L_{C,D}(Q)-\ol{\l})^{-1} \\
    \in \cS_1(\fH) \cdot \cS_1(\fH) \subset \cS_{1/2}(\fH) \subset \fS_1(\fH).
\end{multline}
In turn, combining last two relations we obtain
\begin{align}
    2 i \cdot \Im \left( (L_{C,D}(Q)-\l)^{-1} \right)
    &= \left( L_{C,D}(Q) - \l \right)^{-1} - \left( L_{C,D}(Q) - \ol{\l} \right)^{-1} \nonumber \\
    &+ \left( L_{C,D}(Q) - \ol{\l} \right)^{-1} - \left( L_{C_*,D_*}(Q^*) - \ol{\l} \right)^{-1} \in \fS_1(\fH).
\end{align}
Thus, $T(\l)$, $\l \in \bC_+$, is a compact accumulative
operator with the imaginary part of trace class. Moreover,
$T(\l)$ is complete simultaneously with the operator
$L_{C,D}(Q)$. Therefore, by the M.S.~Brodskii
theorem~\cite{Brod66} (see also~\cite[Corollary
2.2]{Markus70},~\cite[Chapter 4.5]{Nik80}), $T(\l)$ admits the
spectral synthesis.
\end{proof}
%
%
\begin{remark}
Emphasize, that Theorem~\ref{th:synthesis} is stated only for
the resolvent of $L := L_{C,D}(Q)$. In fact, the (unbounded)
operator $L$ itself does not admit the spectral synthesis.
Indeed, the subspace $\fH_a = \{y \in L^2[0,1] : y(x)=0,\ x \in
[0,a]\}$, $a \in (0,1)$, is invariant for $L$. However, it
contains no eigenfunctions of $L$ since, by the Cauchy
uniqueness theorem, each solution of~\eqref{eq:system} vanishing
at zero is identically zero.
\end{remark}
%
%
Further, we apply Theorem~\ref{th:compl.gen} and
Lemma~\ref{lem:accum.detTBCD} to obtain more explicit result on
spectral synthesis.
%
%
\begin{proposition} \label{prop:synth.gen}
Let $B=B^*$, $\Im\, Q \leqslant 0$ and let the operator
$L_{C,D}(0)$ be accumulative. Assume also that for some $C, R
> 0$ and $s \in \bZ_+$
\begin{equation} \label{eq:Delta(-it)}
    \left|\Delta(-i t)\right| \geqslant \frac{C e^{\tau t}}{t^s},
    \quad \tau = \sum_{b_j > 0} b_j, \quad t > R.
\end{equation}
Then the resolvent $(L_{C,D}(Q)-\l)^{-1}$, $\l \in \bC_+$,
admits the spectral synthesis.
\end{proposition}
%
%
\begin{proof}
Since $L_{C,D}(0)$ is accumulative, then, by
Lemma~\ref{lem:accum.detTBCD}, $\det T_{-B}(C,D) \ne 0$. Hence
by Lemma~\ref{lem:Delta} the estimate~\eqref{eq:Delta>=e/l}
holds for $\Delta(\l)$ with $z_1 = 1+i$ and $z_2 = -1+i$ and
$s=0$. Estimate~\eqref{eq:Delta(-it)}
yields~\eqref{eq:Delta>=e/l} with $z_3 = -i$. Hence by
Theorem~\ref{th:compl.gen} the system of root of functions of
the operator $L_{C,D}(Q)$ is complete in $L^2([0,1]; \bC^n)$.
Since $\Im\, Q \leqslant 0$ then, by
Lemma~\ref{lem:accum.CB*C-DB*D}, the operator $L_{C,D}(Q)$ is
accumulative. Therefore, Theorem~\ref{th:synthesis} implies the
spectral synthesis.
\end{proof}
For $2m\times 2m$ Dirac operator
Proposition~\ref{prop:synth.gen} reads as follows.
\begin{corollary} \label{cor:synth.gen}
Let $B = \diag(-I_m, I_m)$, $\Im\, Q \leqslant 0$ and let the
operator $L_{C,D}(0)$ be accumulative. Assume also that for some
$C, R > 0$ and $s \in \bZ_+$
\begin{equation} \label{eq:Delta(-it).Dirac}
    \left|\Delta(-i t)\right| \geqslant \frac{C e^{m t}}{t^s},
    \qquad t > R.
\end{equation}
Then the resolvent $(L_{C,D}(Q)-\l)^{-1}$, $\l \in \bC_+$,
admits the spectral synthesis.
\end{corollary}
Next, we clarify and complete Theorem~\ref{th:explicit.nxn} in
the case of the accumulative (dissipative) operator
$L_{C,D}(0)$. For definiteness we confine ourselves to the case
of an accumulative operator only.
%
%
\begin{theorem} \label{th:accum}
Let $B=B^*$ and let the operator $L_{C,D}(0)$ be accumulative.
Assume also that one of the following conditions is satisfied:

(i) $\det T_B(C,D) \ne 0$;

(ii) $Q$ is continuous at the endpoints $0$ and $1$ of the
segment $[0,1]$ and
\begin{equation}\label{eq:omega.accum}
    \sum_{b_j<0 \atop{b_k>0}} \frac{\det T_B^{c_j \to c_k} b_k q_{kj}(0)
    - \det T_B^{d_k \to d_j} b_j q_{jk}(1)}{b_k-b_j} \ne 0.
\end{equation}
Then the system of root functions of the operator $L_{C,D}(Q)$
is complete and minimal in $L^2([0,1]; \bC^n)$.

Moreover, if in addition, $\Im\, Q \leqslant 0$, then the
resolvent $(L_{C,D}(Q)-\l)^{-1}$, $\l \in \bC_+$, admits the
spectral synthesis.
\end{theorem}
%
%
\begin{proof}[Proof]
Since $L_{C,D}(0)$ is accumulative, then, by
Lemma~\ref{lem:accum.detTBCD}, $\det T_{-B}(C,D) \ne 0$. Hence,
if condition (i) is satisfied, then, by~\cite[Corollary
3.2]{MalOri12}, the system of root functions of the operator
$L_{C,D}(0)$ is complete and minimal in $L^2([0,1]; \bC^n)$.
Further, if $\det T_{B}(C,D) = 0$, then condition (ii) is
satisfied. Clearly, condition~\eqref{eq:omega.accum} means that
$\omega_1(-i) \ne 0$. Since $\omega_0(i) = \det T_{-B}(C,D) \ne
0$, it remains to apply Theorem~\ref{th:explicit.nxn}.

If $\Im\, Q \leqslant 0$ then, by
Lemma~\ref{lem:accum.CB*C-DB*D}, the operator $L_{C,D}(Q)$ is
accumulative. Since the system of root functions of $L_{C,D}(Q)$
is complete in $L^2([0,1]; \bC^n)$, Theorem~\ref{th:synthesis}
implies the spectral synthesis.
\end{proof}
%
%
\begin{remark} \label{rem:ODE.synth}
In connection with Theorem~\ref{th:synthesis} we briefly discuss
the spectral synthesis for $m$-dissipative BVP for $n$th order
ordinary differential equation~\eqref{eq:ODE} generated by $n$
linearly independent boundary forms. First we note that if the
resolvent set of a BVP for equation~\eqref{eq:ODE} is non-empty,
then the resolvent $R(\l)$ is trace class operator. For $n=2$,
i.e. for the operator $-D^2 + q$, $D := \frac{d}{dx}$, this fact
is implied by~\cite[III.10.4.3]{GohKre65} since the Green
function $G(t,s)$ has essentially bounded derivative in mean
with respect to $t$, for $n \geqslant 3$ it is even $C^1([0,1]
\times [0,1])$-kernel.

Note also, that if $q \in L^{2}[0,1]$ then $\dom(-D^2 +q)
\subset W^{2,2}[0,1]$, and hence $R(\l) \in \cS_{1/2}(L^2[0,1])$
due to the properties of the embedding $W^{2,2}[0,1]
\hookrightarrow L^2[0,a]$. Alternatively for $n=2$ and $q \in
L^1[0,1]$ one can adopt the proof of
Theorem~\ref{th:Resolv.dif.in.S1} and
Proposition~\ref{prop:RL.in.S1+KyFan} taking the Dirichlet
realization of $-D^2$ in place of the periodic Dirac operator.

Further, by the Keldysh-Lidskii
theorem~\cite[Theorem~V.6.1]{GohKre65}, the dissipative operator
$R(\l)$, $\l \in \bC_-$, is complete. To obtain the spectral
synthesis it remains to apply the M.S.~Brodskii
theorem~\cite{Brod66} (see also~\cite[IV.5]{Nik80}).
\end{remark}
%
%
%
\section{Application to the Timoshenko beam model}
\label{sec:Timoshenko}
%
%
Here we obtain some important geometric properties of the system
of root functions for the dynamic generator of the Timoshenko
beam model. Consider the following linear system of two coupled
hyperbolic equations for $t \geqslant 0$
\begin{eqnarray}
    \label{eq:Tim.Ftt}
        I_{\rho}(x) \Phi_{tt} &=& K(x)(W_x-\Phi) + (EI(x) \Phi_x)_x - p_1(x) \Phi_t, \quad x \in [0, \ell],\\
    \label{eq:Tim.Wtt}
        \rho(x) W_{tt} &=& (K(x)(W_x-\Phi))_x - p_2(x) W_t, \qquad \qquad \qquad x \in [0, \ell].
\end{eqnarray}
The vibration of the Timoshenko beam of the length $\ell$
clamped at the left end is governed by the system
\eqref{eq:Tim.Ftt}--\eqref{eq:Tim.Wtt} subject to the following
boundary conditions for $t \geqslant 0$~\cite{Tim55}:
\begin{eqnarray}
    \label{eq:Tim.W0F0}
        W(0,t) = \Phi(0,t) &=& 0, \\
    \label{eq:Tim.WLFLa1}
        \bigl(EI(x) \Phi_x(x,t) + \alpha_1 \Phi_t(x,t) + \beta_1 W_t(x,t)\bigr)\bigr|_{x=l} &=& 0, \\
    \label{eq:Tim.WLFLa2}
        \bigl(K(x)(W_x(x,t)-\Phi(x,t)) + \alpha_2 W_t(x,t) + \beta_2 \Phi_t(x,t)\bigr)\bigr|_{x=l} &=& 0.
\end{eqnarray}
Here $W(x,t)$ is the lateral displacement at a point $x$ and
time $t$, $\Phi(x,t)$ is the bending angle at a point $x$ and
time $t$, $\rho(x)$ is a mass density, $K(x)$ is the shear
stiffness of a uniform cross-section, $I_{\rho}(x)$ is the
rotary inertia, $EI(x)$ is the flexural rigidity at a point $x$,
$p_1(x)$ and $p_2(x)$ are locally distributed feedback
functions, $\alpha_j, \beta_j \in \bC$, $j \in \{1,2\}$.
Boundary conditions at the right end contain as partial cases
most of the known boundary conditions if $\alpha_1, \alpha_2$
are allowed to be infinity.

Regarding the coefficients we assume that they satisfy the
following general conditions:
\begin{eqnarray}
    \label{eq:Tim.coef.cond1}
        \rho, I_{\rho}, K, EI \in C[0,\ell], \qquad p_1, p_2 \in L^1[0,\ell],\\
    \label{eq:Tim.coef.cond2}
        0 < C_1 \leqslant \rho(x), I_{\rho}(x), K(x), EI(x) \leqslant C_2, \quad x \in [0,\ell].
\end{eqnarray}
The energy space associated with the
problem~\eqref{eq:Tim.Ftt}--\eqref{eq:Tim.WLFLa2} is
\begin{equation} \label{eq:cH.def}
    \fH := \wt{H}^1_0[0,\ell] \times L^2[0,\ell] \times \wt{H}^1_0[0,\ell] \times L^2[0,\ell],
\end{equation}
where $\wt{H}^1_0[0,\ell] := \{f \in W^{1,2}[0,\ell] :
f(0)=0\}$. The norm in the energy space is defined as follows:
\begin{equation} \label{eq:Tim.|y|H}
    \|y\|_{\fH}^2 = \int_0^\ell \bigl(EI|y_1'|^2+I_{\rho}|y_2|^2 + K|y_3'-y_1|^2+\rho|y_4|^2\bigr)dx,
    \quad y =\col(y_1,y_2,y_3,y_4).
\end{equation}
The problem~\eqref{eq:Tim.Ftt}--\eqref{eq:Tim.WLFLa2} can be
rewritten as
\begin{equation} \label{eq:Tim.yt=i.cLy}
    y_t = i \cL y, \quad y(x,t)|_{t=0} = y_0(x),
\end{equation}
where $y$ and $\cL$ are given by
\begin{equation} \label{eq:Tim.Ly.def}
    y = \begin{pmatrix} \Phi(x,t) \\ \Phi_t(x,t) \\ W(x,t) \\ W_t(x,t) \end{pmatrix}, \ \
    \cL \begin{pmatrix} y_1 \\ y_2 \\ y_3 \\ y_4 \end{pmatrix} = \frac{1}{i} \begin{pmatrix}
        y_2 \\ \frac{1}{I_{\rho}(x)}\Bigl(K(x)(y_3'-y_1) + \bigl(EI(x) y_1'\bigr)' - p_1(x) y_2\Bigr) \\
        y_4 \\ \frac{1}{\rho(x)}\Bigl(\bigl(K(x)(y_3'-y_1)\bigl)' - p_2(x) y_4\Bigr)
    \end{pmatrix}
\end{equation}
on the domain
\begin{eqnarray} \label{eq:Tim.dom.cL}
    \dom(\cL) &=& \left\{ y = \col(y_1,y_2,y_3,y_4) : y_1, y_2, y_3, y_4 \in \wt{H}^1_0[0,\ell]\right., \nonumber \\
    && EI \cdot y_1' \in AC[0,\ell], \ \ (EI\cdot y_1')' - p_1 y_2 \in L^2[0,\ell], \nonumber \\
    && K\cdot(y_3'-y_1) \in AC[0,\ell], \ \ (K\cdot(y_3'-y_1))' - p_2 y_4 \in L^2[0,\ell], \nonumber \\
    && \bigl(EI \cdot y_1'\bigr)(\ell) + \alpha_1 y_2(\ell) + \beta_1 y_4(\ell)= 0, \nonumber \\
    && \Bigl.\bigl(K \cdot (y_3'-y_1)\bigr)(\ell) + \alpha_2 y_4(\ell) + \beta_2 y_2(\ell)= 0 \Bigr\}.
\end{eqnarray}
Timoshenko beam model is investigated in numerous papers
(see~\cite{Tim21,Tim55,KimRen87,Shub02,Souf03,XuYung04,XuHanYung07,WuXue11}
and the references therein). A number of stability,
controllability, and optimization problems were studied. Note
also that the general
model~\eqref{eq:Tim.Ftt}--\eqref{eq:Tim.WLFLa2} of spatially
non-homogenous Timoshenko beam with both boundary and locally
distributed damping covers the cases studied by many authors.
Geometric properties of the system of root functions of the
operator $\cL$ play important role in investigation of different
properties of the
problem~\eqref{eq:Tim.Ftt}--\eqref{eq:Tim.WLFLa2}.

Below we establish completeness and the Riesz basis property
with parentheses of the operator $\cL$, without analyzing its
spectrum. For convenience we impose the following additional
algebraic assumption on $\cL$:
\begin{equation} \label{eq:Tim.Irho=C0 EI/K rho}
    \nu(x) :=  \frac{EI(x) \rho(x)} {K(x) I_{\rho}(x)} = \const, \quad x \in [0,\ell],
\end{equation}
Clearly,~\eqref{eq:Tim.Irho=C0 EI/K rho} is satisfied whenever
$I_{\rho}(x) = R \rho(x)$, where $R = \rm{const}$ is a
cross-sectional area of the beam, $EI$ and $K$ are constant
functions, while $\rho \in AC[0,\ell]$ and is arbitrary positive
(cf. condition~\eqref{eq:Tim.h1,h2.in.AC}). Our approach to the
spectral properties of the operator $\cL$ is based on the
similarity reduction of $\cL$ to a special $4\times 4$
Dirac-type operator. To state the result we need some additional
preparations.

Let $\gamma(\cdot)$ be given by
\begin{equation} \label{eq:Tim.Irho/EI=...}
    \sqrt{\frac{I_{\rho}(x)}{EI(x)}} = b_1 \gamma(x), \quad\text{where}\quad b_1>0
    \quad\text{and}\quad \int_0^\ell \gamma(x) dx = 1.
\end{equation}
Conditions~\eqref{eq:Tim.coef.cond1}
and~\eqref{eq:Tim.coef.cond2} imply together that $\gamma \in
C[0,\ell]$ and is positive. Further, in view
of~\eqref{eq:Tim.Irho=C0 EI/K rho} we have
\begin{equation} \label{eq:Tim.rho/K=...}
    \sqrt{\frac{\rho(x)}{K(x)}} = b_2 \gamma(x), \quad\text{where}\quad b_2>0.
\end{equation}
Let
\begin{eqnarray}
    \label{eq:Tim.B}
        B &:=& \diag(-b_1, b_1, -b_2, b_2). \\
    \label{eq:Tim.Theta(x)}
        \Theta(x) &:=& -2i \diag(I_{\rho}(x), I_{\rho}(x), \rho(x), \rho(x)), \\
    \label{eq:Tim.g1.g2.def}
        h_1(x) &:=& \sqrt{EI(x) I_{\rho}(x)}, \qquad h_2(x):=\sqrt{K(x) \rho(x)}.
\end{eqnarray}
In the sequel we assume that
\begin{equation} \label{eq:Tim.h1,h2.in.AC}
    h_1, h_2 \in AC[0,\ell].
\end{equation}
Therefore, according
to~\eqref{eq:Tim.coef.cond1}--\eqref{eq:Tim.coef.cond2} the
following matrix function is well-defined:
\begin{equation}
    \label{eq:Tim.Q(x)}
        \widehat{Q}(x) := \Theta^{-1}(x)
        \begin{pmatrix}
            p_1+h_1' & p_1-h_1' &     h_2  &    -h_2  \\
            p_1+h_1' & p_1-h_1' &     h_2  &    -h_2  \\
               -h_2  &    -h_2  & p_2+h_2' & p_2-h_2' \\
                h_2  &     h_2  & p_2+h_2' & p_2-h_2'
        \end{pmatrix}.
\end{equation}
Next, we set
\begin{equation} \label{eq:Tim.t=t(x)}
    t(x) = \int_0^x \gamma(s) ds, \quad x \in [0,\ell].
\end{equation}
Since $\gamma \in C[0,\ell]$ and is positive, the function
$t(\cdot)$ strictly increases on $[0,\ell]$, $t(\cdot) \in
C^1[0,\ell]$, and due to~\eqref{eq:Tim.Irho/EI=...} $t(\ell)=1$.
Hence, the inverse function $x(\cdot) := t^{-1}(\cdot)$ is well
defined, strictly increasing on $[0,1]$, and $x(\cdot) \in
C^1[0,1]$. Next, we put
\begin{equation} \label{eq:Tim.Q}
    Q(t) := \widehat{Q}(x(t)) =: (q_{jk}(t))_{j,k=1}^4, \quad t \in [0,1].
\end{equation}
Finally, let
\begin{equation} \label{eq:Tim.C.D}
    C = \begin{pmatrix}
        1 & 1 & 0 & 0 \\
        0 & 0 & 0 & 0 \\
        0 & 0 & 1 & 1 \\
        0 & 0 & 0 & 0
    \end{pmatrix}, \ \
    D = \begin{pmatrix}
        0                    & 0                    & 0                    & 0                    \\
        \alpha_1 - h_1(\ell) & \alpha_1 + h_1(\ell) & \beta_1              & \beta_1              \\
        0                    & 0                    & 0                    & 0                    \\
        \beta_2              & \beta_2              & \alpha_2 - h_2(\ell) & \alpha_2 + h_2(\ell) \\
    \end{pmatrix}.
\end{equation}
%
%
\begin{proposition} \label{prop:Tim.similar}
Let functions $\rho, I_{\rho}, K, EI, p_1, p_2, h_1, h_2$
satisfy
conditions~\eqref{eq:Tim.coef.cond1},~\eqref{eq:Tim.coef.cond2},~\eqref{eq:Tim.Irho=C0
EI/K rho} and~\eqref{eq:Tim.h1,h2.in.AC}. Then the operator
$\cL$ is similar to the $4 \times 4$ Dirac-type operator $L :=
L_{C,D}(Q)$ with the matrices $B,C,D,Q(\cdot)$ given
by~\eqref{eq:Tim.B},~\eqref{eq:Tim.C.D} and~\eqref{eq:Tim.Q}.
\end{proposition}
%
%
\begin{proof}[Proof]
Introduce the following operator
\begin{equation}
    U y = \col(EI(x)y_1',\ y_2,\ K(x)(y_3'-y_1),\ y_4), \qquad y = \col(y_1,y_2,y_3,y_4),
\end{equation}
that maps the Hilbert space $\fH$ given by~\eqref{eq:cH.def}
into $L^2([0,\ell]; \bC^4)$. Since $\frac{d}{dx}$ isometrically
maps
\begin{equation}
    \wt{H}_0^1[0,\ell] = \{f \in W^{1,2}[0,\ell] : f(0)=0\}
\end{equation}
onto $L^2[0,\ell]$, it follows from
conditions~\eqref{eq:Tim.coef.cond2} that the operator $U$ is
bounded with bounded inverse. It is easy to check that for
$y=\col(y_1,y_2,y_3,y_4)$
\begin{equation} \label{eq:Tim.cLU-1}
    \cL U^{-1} y =
    \frac{1}{i}\begin{pmatrix} y_2 \\ \frac{1}{I_{\rho}}(y_1'-p_1y_2+y_3) \\
    y_4 \\ \frac{1}{\rho}(y_3' - p_2y_4)\end{pmatrix}, \ \
    \wt{L} y := U \cL U^{-1} y = \frac{1}{i}\begin{pmatrix} EI \cdot y_2' \\ \frac{1}{I_{\rho}}(y_1'-p_1y_2+y_3) \\
    K \cdot (y_4'-y_2) \\ \frac{1}{\rho}(y_3' - p_2 y_4)\end{pmatrix},
\end{equation}
and
\begin{align} \label{eq:Tim.dom.wtL}
    \dom(\wt{L}) = U \dom(\cL) = \bigl\{ &\bigr.y = \col(y_1,y_2,y_3,y_4)
        \in W^{1,1}([0,\ell]; \bC^4) : \nonumber \\
    &\wt{L}y \in L^2([0,\ell]; \bC^4), \quad y_2(0) = y_4(0) = 0, \nonumber \\
    &y_1(\ell) + \alpha_1 y_2(\ell) + \beta_1 y_4(\ell) = 0, \quad
        y_3(\ell) + \alpha_2 y_4(\ell) + \beta_2 y_2(\ell) = 0 \bigl.\bigr\}.
\end{align}
Thus, the operator $\cL$ is similar to the operator $\wt{L}$,
\begin{equation}
    \wt{L}y = -i\wt{B}(x)y'+\wt{Q}(x)y
\end{equation}
with the domain $\dom(\wt{L})$ given by~\eqref{eq:Tim.dom.wtL},
and the matrix functions $\wt{B}(\cdot)$, $\wt{Q}(\cdot)$, given
by
\begin{equation}
    \wt{B}(x) := \begin{pmatrix}
        0 & EI(x) & 0 & 0 \\
        \frac{1}{I_{\rho}(x)} & 0 & 0 & 0 \\
        0 & 0 & 0 & K(x) \\
        0 & 0 & \frac{1}{\rho(x)} & 0
    \end{pmatrix}, \ \
    \wt{Q}(x) := i \begin{pmatrix}
        0 & 0 & 0 & 0 \\
        0 & \frac{p_1(x)}{I_{\rho}(x)} & -\frac{1}{I_{\rho}(x)} & 0 \\
        0 & K(x) & 0 & 0 \\
        0 & 0 & 0 & \frac{p_2(x)}{\rho(x)}
    \end{pmatrix}.
\end{equation}
Note, that $\wt{Q} \in L^1([0,\ell]; \bC^{4 \times 4})$ in view
of
conditions~\eqref{eq:Tim.coef.cond1}--\eqref{eq:Tim.coef.cond2}.
Next we diagonalize the matrix $\wt{B}(\cdot)$. Namely, setting
\begin{equation} \label{eq:Tim.W(x)}
    \wt{U}(x) := \begin{pmatrix}
        -h_1(x) & h_1(x) & 0 & 0 \\
        1 & 1 & 0 & 0 \\
        0 & 0 & -h_2(x) & h_2(x) \\
        0 & 0 & 1 & 1
    \end{pmatrix},
\end{equation}
and noting that
\begin{equation}
    \wt{U}^{-1}(x) = \frac12 \begin{pmatrix}
        -\frac{1}{h_1(x)} & 1 & 0 & 0 \\
        \frac{1}{h_1(x)} & 1 & 0 & 0 \\
        0 & 0 & -\frac{1}{h_2(x)} & 1 \\
        0 & 0 & \frac{1}{h_2(x)} & 1
    \end{pmatrix},
\end{equation}
we easily get after straightforward calculations
\begin{equation} \label{eq:Tim.W1BW}
    \wt{U}^{-1}(x) \wt{B}(x) \wt{U}(x) = \diag\left(-\sqrt{\frac{EI(x)}{I_{\rho}(x)}},\sqrt{\frac{EI(x)}{I_{\rho}(x)}},
    -\sqrt{\frac{K(x)}{\rho(x)}},\sqrt{\frac{K(x)}{\rho(x)}}\right) = \frac{1}{\gamma(x)} B^{-1},
\end{equation}
Here we have used definition~\eqref{eq:Tim.g1.g2.def} of $h_1$,
$h_2$, and definitions~\eqref{eq:Tim.Irho/EI=...}
and~\eqref{eq:Tim.rho/K=...} of $b_1$, $b_2$, and $\gamma(x)$,
respectively. Further, note that
\begin{equation}
    \wt{U}(\cdot) \in W^{1,1}([0,\ell]; \bC^{4\times 4})
    \quad\text{and}\quad \widehat{Q} \in L^1([0,\ell]; \bC^{4 \times 4})
\end{equation}
in view of~\eqref{eq:Tim.coef.cond1},~\eqref{eq:Tim.coef.cond2}
and~\eqref{eq:Tim.h1,h2.in.AC}, where $\widehat{Q}(\cdot)$ is
given by~\eqref{eq:Tim.Q(x)} and~\eqref{eq:Tim.Theta(x)}. Hence,
it is easily seen that
\begin{equation} \label{eq:Tim.W1QW}
    \wt{U}^{-1}(x) \wt{Q}(x) \wt{U}(x) - i \wt{U}^{-1}(x) \wt{B}(x) \wt{U}'(x) = \widehat{Q}(x), \qquad x \in [0,\ell].
\end{equation}
Introducing the operator $\wt{U} : y \to \wt{U}(x) y$ in
$L^2([0,\ell]; \bC^4)$ and taking into
account~\eqref{eq:Tim.W1BW} and~\eqref{eq:Tim.W1QW} we obtain
that for any $y \in W^{1,1}([0,\ell]; \bC^4)$ and satisfying
$\wt{U} y \in \dom(\wt{L})$
\begin{equation} \label{eq:Tim.wtLy}
    \widehat{L} y := \wt{U}^{-1} \wt{L} \wt{U} y = -i {\gamma(x)}^{-1} B^{-1} y' + \widehat{Q}(x) y.
\end{equation}
Next, taking into account formulas~\eqref{eq:Tim.C.D} and
\eqref{eq:Tim.W(x)} for matrices $C$, $D$, and $\wt{U}(\cdot)$,
respectively, we derive
\begin{equation} \label{eq:Tim.dom.whL}
    \dom(\widehat{L}) = \{y \in W^{1,1}([0,\ell]; \bC^4) :
    \widehat{L}y \in L^2([0,\ell];\bC^4),\ Cy(0)+Dy(\ell) = 0 \}.
\end{equation}
Finally, we apply similarity transformation $S$ that realizes
the change of variable $x=x(t)$,
\begin{equation}
    S : L^2([0,\ell]; \bC^4) \to L^2([0,1]; \bC^4), \qquad (Sf)(t)=f(x(t)), \quad t \in [0,1].
\end{equation}
Since both $t(\cdot)$ and $x(\cdot)$ are strictly increasing and
continuously differentiable, the following implications hold
\begin{eqnarray}
    \label{eq:Tim.f->fx} f(\cdot) \in W^{1,1}([0,\ell];\bC^4) &\Rightarrow& f(x(\cdot)) \in W^{1,1}([0,1];\bC^4), \\
    \label{eq:Tim.g->gt} g(\cdot) \in W^{1,1}([0,1];\bC^4) &\Rightarrow& g(t(\cdot)) \in W^{1,1}([0,\ell];\bC^4).
\end{eqnarray}
Hence~\eqref{eq:Tim.dom.whL} and~\eqref{eq:dom} implies $\dom(L)
= S \dom(\widehat{L})$. Next, it follows
from~\eqref{eq:Tim.t=t(x)} that $t'(x)=\gamma(x)$, $x \in
[0,\ell]$. Hence for $f \in \dom(L)$ and $x \in [0,\ell]$ one
has
\begin{align} \label{eq:Tim.LS-1}
    (\widehat{L} S^{-1} f)(x) \nonumber
    &= -i \gamma(x)^{-1} B^{-1} \frac{d}{dx}\bigl[f(t(x))\bigr] + \widehat{Q}(x) f(t(x)) \\
    &= -i B^{-1} f'(t(x)) + \widehat{Q}(x) f(t(x)),
\end{align}
which directly implies that $L = S \widehat{L} S^{-1}$.
Combining this identity with~\eqref{eq:Tim.cLU-1}
and~\eqref{eq:Tim.wtLy} one concludes that $\cL$ is similar to
$L = L_{C,D}(Q)$.
\end{proof}
%
%
\begin{remark} \label{rem:Tim.AC}
Proposition~\ref{prop:Tim.similar} remains valid if we replace
condition~\eqref{eq:Tim.coef.cond1} by the weaker assumption
$\rho, I_{\rho}, K, EI \in L^{\infty}[0,\ell]$ and assume in
addition that the inverse function $x(\cdot) = t^{-1}(\cdot)$ is
absolutely continuous. Otherwise
implication~\eqref{eq:Tim.f->fx} fails, since in general the
inverse function of absolutely continuous function is not
necessarily absolutely continuous. For instance, the function
$h(x) := x + C(x)$, $x \in [0,1]$, where $C(\cdot)$ is the
Cantor function, strictly increases and is not absolutely
continuous. At the same time, the inverse function is absolutely
continuous.
\end{remark}
%
%
Applying~\cite[Corollary 3.2]{MalOri12} and
Theorem~\ref{th:basis.LCD} to the operator $L$ we obtain the
following result.
%
%
\begin{theorem} \label{th:Tim.weak}
Let
conditions~\eqref{eq:Tim.coef.cond1},~\eqref{eq:Tim.coef.cond2},~\eqref{eq:Tim.Irho=C0
EI/K rho},~\eqref{eq:Tim.h1,h2.in.AC} be satisfied and let also
\begin{equation} \label{eq:Tim.a1!=h1,a1!=h2}
    (\alpha_1 + h_1(\ell)) (\alpha_2 + h_2(\ell)) \ne \beta_1 \beta_2 \quad\text{and}\quad
    (\alpha_1 - h_1(\ell)) (\alpha_2 - h_2(\ell)) \ne \beta_1 \beta_2.
\end{equation}

\item[(i)] Then the system of root functions of $\cL$ is
    complete and minimal in $\fH$.

\item[(ii)] Assume in addition that
\begin{equation} \label{eq:Tim.p1,p2.inLinf}
    p_1, p_2 \in L^{\infty}[0,\ell], \quad h_1, h_2 \in \Lip_1[0,\ell]
    \quad\text{and}\quad \beta_1 = \beta_2 = 0.
\end{equation}
Then the system of root functions of the operator $\cL$ forms a
Riesz basis with parentheses in $\fH$.
\end{theorem}
%
%
\begin{proof}[Proof]
\textbf{(i)} Consider the operator $L_{C,D}(Q)$ defined in
Proposition~\ref{prop:Tim.similar}. Combining
expressions~\eqref{eq:Tim.B} and~\eqref{eq:Tim.C.D} for the
matrices $B, C$, $D$ with definition of $T_A(C,D)$ yields
\begin{equation} \label{eq:Tim.TBCD}
    \det T_{B}(C,D) = \det\begin{pmatrix}
         1 & 0 & 0 & 0 \\
         0 & \alpha_1+h_1(\ell) & 0 & \beta_1 \\
         0 & 0 & 1 & 0 \\
         0 & \beta_2 & 0 & \alpha_2+h_2(\ell)
    \end{pmatrix} = (\alpha_1 + h_1(\ell))(\alpha_2 + h_2(\ell)) - \beta_1 \beta_2.
\end{equation}
Similarly one gets
\begin{equation}
    \det T_{-B}(C,D) = (\alpha_1-h_1(\ell)) (\alpha_2-h_2(\ell)) - \beta_1 \beta_2.
\end{equation}
Conditions~\eqref{eq:Tim.a1!=h1,a1!=h2} implies $\det T_{B}(C,D)
\ne 0$ and $\det T_{-B}(C,D) \ne 0$. Therefore,
by~\cite[Corollary 3.2]{MalOri12}, the system of root functions
of the operator $L_{C,D}(Q)$ is complete and minimal in
$L^2([0,1]; \bC^4)$. Since, by
Proposition~\ref{prop:Tim.similar}, $\cL$ is similar to the
operator $L_{C,D}(Q)$, the system of root functions of the
operator $\cL$ is complete and minimal in $\fH$.

\textbf{(ii)} Again consider the operator $L_{C,D}(Q)$ defined
in Proposition~\ref{prop:Tim.similar}. Since $\beta_1 = \beta_2
= 0$ and~\eqref{eq:Tim.a1!=h1,a1!=h2} is fulfilled, then
according to~\eqref{eq:Tim.B} and~\eqref{eq:Tim.C.D} the
matrices $B$, $C$, $D$ have the block structure described
in~\eqref{eq:basis.BCD}--\eqref{eq:basis.CjDj2} with $r=2$.
Moreover,~\eqref{eq:Tim.p1,p2.inLinf} implies $Q \in
L^{\infty}([0,1]; \bC^{4 \times 4})$. Therefore, combining
Theorem~\ref{th:basis.LCD} with
Proposition~\ref{prop:Tim.similar} yields the statement.
\end{proof}
Applying Corollary~\ref{cor:n=4} we can improve
Theorem~\ref{th:Tim.weak}(i) assuming that $\widehat{Q}(\cdot)$
is continuous at the endpoints $0$, $\ell$. For simplicity we
assume that $\beta_1 = \beta_2 = 0$.
%
%
\begin{theorem} \label{th:Tim.nonweak}
Let the functions $\rho$, $I_{\rho}$, $K$, $EI$, $p_1$, $p_2$,
$h_1$, $h_2$ satisfy
conditions~\eqref{eq:Tim.coef.cond1},~\eqref{eq:Tim.coef.cond2},~\eqref{eq:Tim.Irho=C0
EI/K rho} and~\eqref{eq:Tim.h1,h2.in.AC}. Let also the functions
$p_1$, $p_2$, $h_1'$, $h_2'$ be continuous at the endpoints 0
and $\ell$. Assume in addition that $\beta_1 = \beta_2 = 0$ and
the following assumptions are fulfilled:
\begin{itemize}
\item[(i)] $|\alpha_1 - h_1(\ell)| + |\alpha_2 - h_2(\ell)|
    \ne 0$ \ \ and \ \ $|\alpha_1 + h_1(\ell)| + |\alpha_2 +
    h_2(\ell)| \ne 0$;
\item[(ii)] for each $j \in \{1,2\}$ one of the following
    conditions is satisfied:
    \begin{itemize}
        \item[(a)] $\alpha_j^2 \ne h_j^2(\ell)$;
        \item[(b)] $\alpha_j = h_j(\ell)$ and
            $h_j'(\ell) \ne -p_j(\ell)$;
        \item[(c)] $\alpha_j = -h_j(\ell)$ and
            $h_j'(\ell) \ne p_j(\ell)$.
    \end{itemize}
\end{itemize}
Then the system of root functions of $\cL$ is complete and
minimal in $\fH$.
\end{theorem}
%
%
\begin{proof}[Proof]
Consider the operator $L_{C,D}(Q)$ defined in
Proposition~\ref{prop:Tim.similar}. Since $\rho, I_{\rho} \in
C[0,\ell]$ and $p_1$, $p_2$, $h_1'$, $h_2'$ are continuous at
the endpoints 0 and $\ell$, it follows
from~\eqref{eq:Tim.Theta(x)}--\eqref{eq:Tim.Q} that the matrix
function $Q(\cdot)$ is continuous at the endpoints 0 and 1.
Since $\beta_1 = \beta_2 = 0$, the block matrix
representations~\eqref{eq:Tim.B} and~\eqref{eq:Tim.C.D} of the
matrices $B$, $C$, $D$, allow to apply Corollary~\ref{cor:n=4}
and Lemma~\ref{lem:4x4}. Let us verify
conditions~\eqref{eq:d1d2}--\eqref{eq:d3q43} of
Lemma~\ref{lem:4x4}. First, comparing~\eqref{eq:n=4.CD}
with~\eqref{eq:Tim.C.D} yields
\begin{eqnarray}
    && d_1 = \alpha_1 - h_1(\ell), \quad d_2 = \alpha_1 + h_1(\ell), \\
    && d_3 = \alpha_2 - h_2(\ell), \quad d_4 = \alpha_2 + h_2(\ell).
\end{eqnarray}
Therefore, condition~\eqref{eq:d1d2} is always satisfied, since
$h_j(\ell) \ne 0$, $j \in \{1,2\}$, while
condition~\eqref{eq:d1d3} is equivalent to the condition (i) of
the theorem. Further, it follows from~\eqref{eq:Tim.Q(x)}
and~\eqref{eq:Tim.Q} that
\begin{eqnarray}
    && q_{12}(1) = \frac{p_1(\ell) - h_1'(\ell)}{-2 i I_{\rho}(\ell)}, \quad
    q_{21}(1) = \frac{p_1(\ell) + h_1'(\ell)}{-2 i I_{\rho}(\ell)}, \\
    && q_{34}(1) = \frac{p_2(\ell) - h_2'(\ell)}{-2 i \rho(\ell)}, \quad
    q_{43}(1) = \frac{p_2(\ell) + h_2'(\ell)}{-2 i \rho(\ell)}.
\end{eqnarray}
Hence, conditions~\eqref{eq:d1q21} and~\eqref{eq:d3q43} are
equivalent to the conditions (a)-(c) of the theorem for $j=1$
and $j=2$, respectively. Therefore, by Lemma~\ref{lem:4x4},
condition~\eqref{eq:d2d4+...,d1d3+...} is satisfied and, by
Corollary~\ref{cor:n=4}, the system of root functions of the
operator $L_{C,D}(Q)$ is complete and minimal in $L^2([0,1];
\bC^4)$. Therefore, Proposition~\ref{prop:Tim.similar} completes
the proof.
\end{proof}
\begin{remark} \label{rem:Tim.cond}
The main results remain also valid if the function $\nu(\cdot)$
given by~\eqref{eq:Tim.Irho=C0 EI/K rho} satisfies $\nu(x) \ne
1$ for $x \in [0,\ell]$.
\end{remark}
%
%
\begin{remark}
\textbf{(i)} In connection with Theorem~\ref{th:Tim.weak} we
mention the paper~\cite{Shub02} where the operator $\cL$ was
investigated under the following assumptions on the parameters
of the model:
\begin{equation} \label{eq:Tim.p1=p2=0,...}
    EI, K \in W^{3,2}[0,\ell],
    \quad \rho, I_{\rho} \in W^{4,2}[0,\ell], \quad p_1 = p_2 = 0, \quad \beta_1 = \beta_2 = 0,
\end{equation}
but without algebraic assumption~\eqref{eq:Tim.Irho=C0 EI/K
rho}. The completeness of the root functions was stated
in~\cite{Shub02} under the
condition~\eqref{eq:Tim.a1!=h1,a1!=h2} and the additional
assumption
\begin{equation} \label{eq:Tim.nu!=1}
    I_{\rho}(x) K(x) \ne \rho(x) EI(x), \quad x \in [0,\ell],
\end{equation}
which in our notations means that $\nu(x) \ne 1$, $x \in
[0,\ell]$. Unfortunately, the proof of the completeness
in~\cite{Shub02} fails because of the incorrect application of
the Keldysh theorem. Namely, the representation $\cL^{-1} =
\cL_{00}^{-1} (I_{\fH} + T)$ used in~\cite{Shub02}, where $T$ is
of finite rank bounded operator and $\cL_{00} = \cL_{00}^{*}$,
fails since it leads to the inclusion $\dom(\cL) \subset
\dom(\cL_{00})$, which holds if only if $\cL = \cL_{00}$.

Moreover, under
conditions~\eqref{eq:Tim.p1=p2=0,...},~\eqref{eq:Tim.nu!=1}
and~\eqref{eq:Tim.a1!=h1,a1!=h2} the Riesz basis property for
the system of root functions of $\cL$ was stated
in~\cite{Shub02}. The proof is based on the fact that under the
above restrictions the eigenvalues of $\cL$ are asymptotically
simple and separated. However, it is not the case. For instance,
if $K \equiv EI \equiv \rho \equiv 1$, $I_{\rho} \equiv 4$,
$\alpha_1 = 5/2$ and $\alpha_2 = 13/12$, then according
to~\cite[Theorem 4.2]{Shub02} the sequence of the eigenvalues of
$\cL$ splits into two families
\begin{equation}
    \l_n^{(1)} = \frac{\pi n}{2} + \frac{i}{2} \ln 3 + O(n^{-1}) \quad\text{and}\quad
    \l_n^{(2)} = \pi n + \frac{i}{2} \ln 3 + O(n^{-1}), \quad n \in \bZ \setminus \{0\}.
\end{equation}
Clearly, in this case the sequence of the eigenvalues of $\cL$
is not asymptotically simple and separated. Note, however, that
according to Theorem~\ref{th:Tim.weak}(ii) the system of root
functions of the operator $\cL$ always forms a Riesz basis with
parentheses under the
restrictions~\eqref{eq:Tim.coef.cond1},~\eqref{eq:Tim.coef.cond2},~\eqref{eq:Tim.Irho=C0
EI/K
rho},~\eqref{eq:Tim.h1,h2.in.AC},~\eqref{eq:Tim.a1!=h1,a1!=h2}
and~\eqref{eq:Tim.p1,p2.inLinf}.

\textbf{(ii)} In connection with Theorem~\ref{th:Tim.weak} we
also mention the paper~\cite{XuYung04}. In this paper the
operator $\cL$ was investigated under the following stronger
assumptions on the parameters of the model:
\begin{equation}
    \label{eq:Tim.Xu} EI, K, \rho, I_{\rho} \ \text{are constant,} \quad p_1 = p_2 = 0, \quad
    \alpha_1, \alpha_2, \beta_1, \beta_2 \geqslant 0, \quad  4 \alpha_1 \alpha_2 \geqslant (\beta_1 + \beta_2)^2.
\end{equation}
The last condition in~\eqref{eq:Tim.Xu} ensures the
dissipativity of the operator $\cL$. The completeness of the
system of root functions of the operator $\cL$ was proved
in~\cite{XuYung04} under the restrictions~\eqref{eq:Tim.Xu}
and~\eqref{eq:Tim.a1!=h1,a1!=h2}. So, our
Theorem~\ref{th:Tim.weak}(i) generalizes this result to a
broader class of boundary conditions and improves it in the
dissipative case. Note also that under additional assumptions,
guarantying that the eigenvalues of $\cL$ are asymptotically
simple and separated, it was proved in~\cite{XuYung04} that the
root functions of $\cL$ contains the Riesz basis. Moreover, this
fact was applied to show the exponential stability of the
problem~\eqref{eq:Tim.Ftt}--\eqref{eq:Tim.WLFLa2}.
\end{remark}
%
%
\noindent {\bf Acknowledgments.} We are indebted to D.
Yakubovich for the reformulation of the condition (a) of Theorem
\ref{th:explicit.nxn} mentioned in Remark~\ref{rem:omega01}. We
are also indebted to A. Shkalikov for useful remarks helping us
to improve the exposition.
%
%

%
%

\begin{thebibliography}{10}
%
%
\bibitem{AgiMalOri12} A.~V.~Agibalova, M.~M.~Malamud and
    L.~L.~Oridoroga, On the completeness of general boundary value
    problems for $2 \times 2$ first-order systems of ordinary
    differential equations. \emph{Methods Funct. Anal.
    and Topology} (1) \textbf{18} (2012), 4--18.
%
\bibitem{Atk64} F.V.~Atkinson, Discrete and continuous boundary
    problems, Academic Press, New York, 1964.
%
\bibitem{BarBelBor12} A.~Baranov, Y.~Belov and A.~Borichev,
    Hereditary completeness for systems of exponentials and
    reproducing kernels, \emph{Adv. Math.} \textbf{235} (2013),
    pp. 525--554.
%
\bibitem{BarBelBorYak12} A. Baranov, Y. Belov, A. Borichev and
    D. Yakubovich, Recent developments in spectral synthesis for
    exponential systems and for non-self-adjoint operators,
    arXiv:1212.6014 (Submitted on 25 Dec 2012), to appear in
    \emph{Recent Trends in Analysis}, Proceedings of the conference
    in honor of N. Nikolski (Bordeaux, 2011), Theta Foundation,
    Bucharest, 2013.
%
\bibitem{BarYak12} A.~Baranov and D.~Yakubovich,
    Completeness and spectral synthesis of nonselfadjoint
    one-dimensional perturbations of selfadjoint operators,
    arXiv:1212.5965 (Submitted on 24 Dec 2012).
%
\bibitem{BarYak13} A.~Baranov and D.~Yakubovich, One-dimensional
    perturbations of unbounded selfadjoint operators with empty
    spectrum, arXiv:1304.5800 (Submitted on 21 Apr 2013).
%
\bibitem{Bask11} A.~G.~Baskakov, A.~V.~Derbushev and
    A.~O.~Shcherbakov, The method of similar operators in the
    spectral analysis of non-self-adjoint Dirac operators with
    non-smooth potentials. \emph{Izv. Math.} (3) \textbf{75}
    (2011), 445--469.
%
\bibitem{Bir08} G.~D.~Birkhoff, On the asymptotic character of
    the solution of the certain linear differential equations
    containing parameter. \emph{Trans. Amer. Math. Soc.}
    (2) \textbf{9} (1908), 219--231.
%
\bibitem{Bir08exp} G.~D.~Birkhoff, Boundary value and expansion
    problems of ordinary linear differential equations.
    \emph{Trans. Amer. Math. Soc.} \textbf{9} (1908), 373--395.
%
\bibitem{BirLan23} G.~D.~Birkhoff and R.~E.~Langer, The boundary
    problems and developments associated with a system of
    ordinary differential equations of the first order.
    \emph{Proc. Amer. Acad. Arts Sci.} \textbf{58} (1923),
    49--128.
%
\bibitem{Brod66} M.S.~Brodskii, On operators with trace class
    imaginary components, \emph{Acta Sci. Math. Szeged}, \textbf{27}
    (3--4) (1966), 147--155.
%
\bibitem{DjaMit06Zone} P.~Djakov and B.~Mityagin, Instability
    zones of 1D periodic Schr$\ddot{{\rm o}}$dinger and Dirac
    operators, \emph{Russian Math. Surveys} \textbf{61} (4)
    (2006), 663--766.
%
\bibitem{DjaMit10BariDir} P.~Djakov and B.~Mityagin,
    Bari-Markus property for Riesz projections of
    1D periodic Dirac operators. \emph{Math. Nachr.}
    (3) \textbf{283} (2010), 443--462.
%
\bibitem{DjaMit12UncDir} P.~Djakov and B.~Mityagin,
    Unconditional convergence of spectral decompositions of 1D
    Dirac operators with regular boundary conditions.
    \emph{Indiana Univ. Math. J.} (1) \textbf{61} (2012),
    359--398.
%
\bibitem{DjaMit11TrigHill} P.~Djakov and B.~Mityagin,
    Convergence of spectral decompositions of Hill operators
    with trigonometric polynomial potentials, \emph{Math. Ann.},
    \textbf{351} (3) (2011), 509--540.
%
\bibitem{DjaMit12TrigDir} P.~Djakov and B.~Mityagin, 1D Dirac
    operators with special periodic potentials. \emph{Bull.
    Polish Acad. Sci. Math.} (1) \textbf{60} (2012),
    59--75.
%
\bibitem{DjaMit12Equi} P.~Djakov and B.~Mityagin,
    Equiconvergence of spectral decompositions of 1D Dirac
    operators with regular boundary conditions.
    \emph{J. Approximation Theory} (7) \textbf{164} (2012),
    879--927.
%
\bibitem{DjaMit12Crit} P.~Djakov and B.~Mityagin, Criteria for
    existence of Riesz bases consisting of root functions of
    Hill and 1D Dirac operators. \emph{J. Funct. Anal.} (8)
    \textbf{263} (2012), 2300--2332.
%
\bibitem{DjaMit13CritDir} P.~Djakov and B.~Mityagin, Riesz
    bases consisting of root functions of 1D Dirac operators.
    \emph{Proc. Amer. Math. Soc.} (4) \textbf{141} (2013),
    1361--1375.
%
\bibitem{Dun58} N.~Dunford, A Survey of the Theory of Spectral
    Operators, \emph{Bull. Am. Math. Soc.} \textbf{64} (1958),
    217--274.
%
\bibitem{DunSch58} N.~Dunford and J.~Schwartz, Linear Operators,
    Part I, General Theory, Wiley, New York, 1958.
%
\bibitem{DunSch71} N.~Dunford and J.~Schwartz, \emph{Linear
    Operators, Part III, Spectral Operators}. Wiley, New York
    1971.
%
\bibitem{GesMal09} F.~Gesztesy, M.~Malamud, M.~Mitrea and
    S.~Naboko, Generalized Polar Decompositions for Closed
    Operators in Hilbert Spaces and Some Applications,
    \emph{Integral Equat. Oper. Theor.} \textbf{64} (2009),
    83--113.
%
\bibitem{GesTka09} F.~Gesztesy and V.~Tkachenko, A criterion for
    Hill operators to be spectral operators of scalar type.
    \emph{J. Analyse Math.} \textbf{107} (2009), 287--353.
%
\bibitem{GesTka11} F.~Gesztesy and V.~Tkachenko, A Schauder and
    Riesz basis criterion for non-selfadjoint Schr\"{o}dinger
    operators with periodic and anti-periodic boundary
    conditions. \emph{J. Diff. Equat.} (2) \textbf{253}
    (2012), 400--437.
%
\bibitem{Gin71} Yu.~P.~Ginzburg, The almost invariant spectral
    propeties of contractions and the multiplicative properties
    of analytic operator-functions. \emph{Funct. Anal. Appl.} (3)
    \textbf{5} (1971), 197--205.
%
\bibitem{GohKre65} I.~C.~Gohberg and M.~G.~Krein,
    \emph{Introduction to the theory of linear nonselfadjoint
    operators in Hilbert space}. Transl. Math. Monographs,
    \textbf{18}, Amer. Math. Soc., Providence, R.I. 1969.
%
\bibitem{Gub03} G.~M.~Gubreev, On the spectral decomposition of
    finite-dimensional perturbations of dissipative Volterra
    operators. \emph{Tr. Mosc. Mat. Obs.} \textbf{64} (2003)
    90--140; translation in \emph{Trans. Moscow  Math. Soc.}
    2003, 79--126.
%
\bibitem{HasOri09} S.~Hassi and L.~Oridoroga, Theorem of
    Completeness for a Dirac-Type Operator with Generalized
    $\l$-Depending Boundary Conditions. \emph{Integral
    Equat. Oper. Theor.} \textbf{64} (2009), 357--379.
%
\bibitem{Kel51} M.~V.~Keldysh, On the characteristic values and
    characteristic functions of certain classes of
    non-self-adjoint equations. \emph{Doklady Akad.
    Nauk SSSR (N.S.)} (1) \textbf{77} (1951), 11--14 (in Russian).
%
\bibitem{Kes64} G.~M.~Keselman, On the unconditional convergence
    of eigenfunction expansions of certain differential
    operators. \emph{Izv. Vyssh. Uchebn. Zaved. Mat.}
    (2) \textbf{39} (1964), 82--93 (in Russian).
%
\bibitem{Khr77} A.P.~Khromov, Generating functions of
    Volterra operators  \emph{Mat. Sb. (N.S.)}
    \textbf{102(144)}, (3) (1977), 457--472 (in Russian).
%
\bibitem{Khr03} A.~P.~Khromov, Finite dimensional perturbations
    of Volterra operators. \emph{J. Math. Sci.} (N. Y.)
    (5) \textbf{138} (2006), 5893--6066.
%
\bibitem{KimRen87} J.~U.~Kim and Y.~Renardy, Boundary Control of
    the Timoshenko Beam. \emph{SIAM J. Control and
    Optimization}, (6) \textbf{25} (1987), 1417--1429.
%
\bibitem{KosShk78} A.~G.~Kostyuchenko and A.~A.~Shkalikov,
    Summability of expansions in eigenfunctions of differential
    operators and of convolution operators. \emph{Funct. Anal.
    Appl.} (4) \textbf{12} (1978), 262--276.
%
\bibitem{Lev96} B.~Ya.~Levin, \emph{Lectures on Entire
    Functions}, Transl. Math. Monographs, \textbf{150}, Amer.
    Math. Soc., Providence, R.I. 1996 (in collaboration with Yu.
    Lyubarskii, M. Sodin, and V. Tkachenko).
%
\bibitem{LevSar88} B.~M.~Levitan and I.~S.~Sargsyan,
    \emph{Sturm-Liouville And Dirac Operators}. Kluwer,
    Dordrecht 1991.
%
\bibitem{LunMal13Dokl} A.~A.~Lunyov and M.~M.~Malamud, On the
    completeness of the root vectors for first order systems.
    \emph{Dokl. Math.} (3) \textbf{88} (2013), 678--683.
%
\bibitem{LyubTka84} Yu.~I.~Lyubarskii and V.~A.~Tkachenko,
    System $\{e^{\alpha n x} \sin n x\}$. \emph{Funct. Anal.
    Appl.} (2) \textbf{18} (1984), 144--146.
%
\bibitem{Mak12} A.~S.~Makin, On Summability of Spectral
    Expansions Corresponding to the Sturm-Liouville Operator.
    \emph{Inter. J. Math. and Math. Sci.} \textbf{2012}
    (2012), ID 843562, 13 p.
%
\bibitem{Mal99} M.~M.~Malamud, Questions of uniqueness in
    inverse problems for systems of differential equations on a
    finite interval. \emph{Trans. Moscow Math. Soc.} \textbf{60}
    1999, 173--224.
%
\bibitem{Mal08} M.~M.~Malamud, On the completeness of a system
    of root vectors of the Sturm-Liouville operator with general
    boundary conditions. \emph{Funct. Anal. Appl.} (3)
    \textbf{42} (2008), 198--204.
%
\bibitem{MalOri00} M.~M.~Malamud and L.~L.~Oridoroga,
    Completeness theorems for systems of differential equations.
    \emph{Funct. Anal. Appl.} (4) \textbf{34} (2000), 308--310.
%
\bibitem{MalOri10} M.M.~Malamud and L.L.~Oridoroga, On the
    completeness of the root vectors of first order systems,
    \emph{Dokl. Math.}, \textbf{82} (3) (2010), 899--905.
%
\bibitem{MalOri12} M.~M.~Malamud and L.~L.~Oridoroga, On the
    completeness of root subspaces of boundary value problems
    for first order systems of ordinary differential
    equations. \emph{J. Funct. Anal.} \textbf{263} (2012),
    1939--1980; arXiv:0320048.
%
\bibitem{Mar77} V.~A.~Marchenko, \emph{Sturm-Liouville operators
    and applications}. Operator Theory: Advances and Appl.
    \textbf{22}, Birkh\"{a}user Verlag, Basel 1986.
%
\bibitem{Markus70} A.S.~Markus, The problem of spectral
    synthesis for operators with point spectrum, \emph{Mathematics
    of the USSR-Izvestiya}, \textbf{4} (3) (1970),
    670--696.
%
\bibitem{Markus86} A.~S.~Markus, \emph{An Introduction to the
    Spectral Theory of Polynomial Operator Pencils}. Shtiintsa,
    Chisinau, 1986 (in Russian).
%
\bibitem{MarkMats84} A.~S.~Markus and V.~I.~Matsaev, Comparison
    Theorems for Spectra of Linear Operators and Spectral
    Asymptotics. \emph{Tr. Mosk. Mat. Obs.} \textbf{45}
    (1982), 133--181.
%
\bibitem{Mikh62} V.~P.~Mikhailov, On Riesz bases in $L^2(0, 1)$.
    \emph{Dokl. Akad. Nauk SSSR} \textbf{144} (1962),
    981--984 (in Russian).
%
\bibitem{Mit04} B.~Mityagin, Spectral expansions of
    one-dimensional periodic Dirac operators, \emph{Dyn. Partial
    Differ. Equ.} \textbf{1} (2004), 125--191.
%
\bibitem{Nai69} M.~A.~Naimark, \emph{Linear differential
    operators, Part I}. Frederick Ungar Publishing Co., New York
    1967.
%
\bibitem{Nik80} N.K.~Nikolski, Treatise on the Shift Operator,
    Springer-Verlag, Berlin, 1986.
%
\bibitem{ZMNovPit80} S.~P.~Novikov, S.~V.~Manakov,
    L.~P.~Pitaevskij and V.~E.~Zakharov, \emph{Theory of solitons.
    The inverse scattering method}. Springer-Verlag 1984.
%
\bibitem{Ryh09} V.~S.~Rykhlov, Completeness of eigenfunctions of
    one class of pencils of differential operators with constant
    coefficients. \emph{Russian Mathematics (Iz VUZ)} (6)
    \textbf{53} (2009), 33--43.
%
\bibitem{Shk76} A.~A.~Shkalikov, The completeness of the eigen-
    and associated functions of an ordinary differential
    operator with nonregular splitting boundary conditions.
    \emph{Funct. Anal. Appl.} (4) \textbf{10} (1976),
    305--316.
%
\bibitem{Shk79} A.A.~Shkalikov, On the basis problem of the
    eigenfunctions of an ordinary differential operator,
    \emph{Russ. Math. Surv.} \textbf{34} (5) (1979),
    249--250.
%
\bibitem{Shk82} A.~A.~Shkalikov, The basis problem of the
    eigenfunctions of ordinary differential operators with
    integral boundary conditions. \emph{Moscow Univ. Math.
    Bull.} (6) \textbf{37} (1982), 10--20.
%
\bibitem{Shk83} A.~A.~Shkalikov, Boundary Problems for
    Ordinary Differential Equations with Parameter in the
    Boundary Conditions, \emph{Tr. Semin. im. I.G. Petrovskogo}
    \textbf{9} (1983), 190--229; English transl. in \emph{J.
    Soviet Math.} (6) \textbf{33} (1986), 1311--1342.
%
\bibitem{Shk10} A.~A.~Shkalikov, On the basis property of root
    vectors of a perturbed self-adjoint operator. \emph{Proc. Steklov
    Inst. Math.} \textbf{269} (2010), 284--298.
%
\bibitem{ShkVel09} A. A. Shkalikov and O. A. Veliev, On the
    Riesz Basis Property of the Eigen- and Associated Functions
    of Periodic and Antiperiodic Sturm-Liouville Problems.
    \emph{Mathematical Notes} (5) \textbf{85}, (2009), 647--660.
%
\bibitem{Shub02} M.~A.~Shubov, Asymptotic and spectral analysis
    of the spatially nonhomogeneous Timoshenko beam model.
    \emph{Math. Nachr.} \textbf{241} (2002), 125--162.
%
\bibitem{Souf03} A.~Soufyane and A.~Wehbe, Uniform stabilization
    for the Timoshenko beam by a locally distributed damping,
    \emph{Electronic J. Differ. Equ.}, \textbf{2003} (29)
    (2003), 1--14.
%
\bibitem{Tam12} J.D.~Tamarkin, Sur quelques points de la
    theorie des equations differentielles lineaires ordinaires
    et sur la generalisation de la serie de Fourier, \emph{Rend.
    Circ. Mat. Palermo} \textbf{34} (2) (1912), 345--382.
%
\bibitem{Tam17} J.~D.~Tamarkin, \emph{On some general problems
    of the theory of ordinary linear differential operators and
    the expansion of arbitrary functions into series}.
    Petrograd 1917.
%
\bibitem{Tam28} J.~D.~Tamarkin, Some general problems of the
    theory of linear differential equations and expansions of an
    arbitrary functions in series of fundamental functions.
    \emph{Math. Z.} \textbf{27} (1928), 1--54.
%
\bibitem{Tim21} S.~Timoshenko, On the correction for shear of
    the differential equation for transverse vibrations of
    prismatic bars, \emph{Philisophical magazine} \textbf{41} (1921),
    744--746.
%
\bibitem{Tim55} S.~Timoshenko, \emph{Vibration Problems in
    Engineering}. Van Norstrand, New York 1955.
%
\bibitem{TroYam01} I.~Trooshin and M.~Yamamoto, Riesz basis of
    root vectors of a nonsymmetric system of first-order
    ordinary differential operators and application to inverse
    eigenvalue problems, \emph{Appl. Anal.} \textbf{80} (2001),
    19--51.
%
\bibitem{TroYam02} I.~Trooshin and M.~Yamamoto, Spectral
    properties and an inverse eigenvalue problem for
    nonsymmetric systems of ordinary differential operators.
    \emph{J. Inverse Ill-Posed Probl.} (6) \textbf{10} (2002),
    643--658.
%
\bibitem{Wermer52} J.~Wermer, On invariant subspaces of normal
    operators, \emph{Proc. Amer. Math. Soc.} \textbf{3} (2)
    (1952), pp. 270--277.
%
\bibitem{WuXue11} Y.~Wu and X.~Xue, Decay rate estimates for the
    quasi-linear Timoshenko system with nonlinear control and
    damping terms. \emph{J. Math. Physics} 093502 \textbf{52}
    (2011), 18 p.
%
\bibitem{XuHanYung07} G.~Q.~Xu, Z.~J.~Han and S.~P.~Yung, Riesz
    basis property of serially connected Timoshenko beams.
    \emph{Inter. J. Control} (3) \textbf{80} (2007), 470--485.
%
\bibitem{XuYung04} G.~Q.~Xu and S.~P.~Yung, Exponential Decay
    Rate for a Timoshenko Beam with Boundary Damping. \emph{J.
    Optimiz. Theory Appl.} (3) \textbf{123} (2004), 669--693.
%
\end{thebibliography}
\end{document}